\documentclass[12pt]{article}
\usepackage{amsmath,amsthm,amsfonts,amssymb,epsfig,verbatim,xcolor,constants,subcaption}
\usepackage[left=1in,top=1in,right=1in]{geometry}
\usepackage[normalem]{ulem}

\newconstantfamily{smc}{symbol=\mathsf{c}}
\newconstantfamily{lgc}{symbol=\mathsf{C}}
\newconstantfamily{smac}{symbol=\mathsf{k}}
\newconstantfamily{lgac}{symbol=\mathsf{K}}

\newtheorem{thm}{Theorem}[section]
\newtheorem{lem}[thm]{Lemma}
\newtheorem{prop}[thm]{Proposition}
\newtheorem{cor}[thm]{Corollary}

\newtheorem{rem}{Remark}

\numberwithin{equation}{section}

\begin{document}

\title{Asymptotics for first passage percolation on logarithmic subgraphs of $\mathbb{Z}^2$}
\author{Michael Damron \thanks{The research of M. D. was supported by NSF grant DMS-2054559.} \\ \small{Georgia Tech}  \and Wai-Kit Lam\thanks{The research of W.-K. L. is supported by the National Science and Technology Council in Taiwan grant number 113-2115-M-002-009-MY3.} \\ \small{National Taiwan University} }

	\maketitle 
	\begin{abstract}
For $a>0$ and $b \geq 0$, let $\mathbb{G}_{a,b}$ be the subgraph of $\mathbb{Z}^2$ induced by the vertices between the first coordinate axis and the graph of the function $f = f_{a,b}(u) = a \log (1+u) + b \log(1+\log(1+u))$, $u \geq 0$. It is known that for $a>0$, the critical value for Bernoulli percolation on $\mathbb{G}_f = \mathbb{G}_{a,b}$ is strictly between $1/2$ and $1$, and that if $b>2a$ then the percolation phase transition is discontinuous. We study first-passage percolation (FPP) on $\mathbb{G}_{a,b}$ with i.i.d.~edge-weights $(\tau_e)$ satisfying $p = \mathbb{P}(\tau_e=0) \in [1/2,1)$ and the ``gap condition'' $\mathbb{P}(\tau_e \leq \delta) = p$ for some $\delta>0$. We find the rate of growth of the expected passage time in $\mathbb{G}_f$ from the origin to the line $x=n$, and show that, while when $p=1/2$ it is of order $n/(a \log n)$, when $p>1/2$ it can be of order (a) $n^{c_1}/(\log n)^{c_2}$, (b) $(\log n)^{c_3}$, (c) $\log \log n$, or (d) constant, depending on the relationship between $a,b,$ and $p$. For more general functions $f$, we prove a central limit theorem for the passage time and show that its variance grows at the same rate as the mean. As a consequence of our methods, we improve the percolation transition result by showing that the phase transition on $\mathbb{G}_{a,b}$ is discontinuous if and only if $b > a$, and improve ``sponge crossing dimensions'' asymptotics from the '80s on subcritical percolation crossing probabilities for tall thin rectangles.
	\end{abstract}

\section{Introduction}

\subsection{Background and the model}\label{sec: percolation}

\subsubsection{FPP on wedge graphs}
In this paper, we consider first-passage percolation (FPP) on wedge-like subgraphs of the square lattice $(\mathbb{Z}^2, \mathbb{E}^2)$, defined using a function 
\begin{equation}\label{eq: f_conditions}
f: [0,\infty) \to [0,\infty) \text{ satisfying }f(0)=0.
\end{equation}
The relevant graph is $\mathbb{G}_f = (\mathbb{V}_f, \mathbb{E}_f)$, where $\mathbb{V}_f$ is the vertex set
\[
\mathbb{V}_f = \{(x_1,x_2) \in \mathbb{Z}^2 : x_1 \geq 0 \text{ and } 0 \leq x_2 \leq f(x_1)\}
\]
and the edge set consists of all nearest-neighbor edges between vertices of $\mathbb{V}_f$:
\[
\mathbb{E}_f = \{\{x,y\} : x,y \in \mathbb{V}_f, |x-y|=1\}.
\]
This $\mathbb{G}_f$ is the subgraph of the square lattice $\mathbb{Z}^2$ with its nearest neighbor edges $\mathbb{E}^2$ induced by the vertices below the graph of $f$.

Let $F$ be a distribution function with $F(0^-)=0$ and let $(\tau_e)_{e \in \mathbb{E}^2}$ be an i.i.d.~family with common distribution function $F$, defined on some probability space $(\Omega, \Sigma, \mathbb{P})$, and assigned to the edges of the full lattice $\mathbb{Z}^2$. Any path $\gamma$ in $\mathbb{G}_f$ with edges $e_0, \dots, e_{k-1}$ in order is assigned the passage time $T_f(\gamma) = \sum_{j=0}^{k-1} \tau_{e_j}$, and the passage time between vertices $x,y \in \mathbb{V}_f$ is defined as
\[
T_f(x,y) = \inf_{\gamma : x \to y} T_f(\gamma),
\]
where the infimum is over all paths $\gamma$ in $\mathbb{G}_f$ that start at $x$ and end at $y$. 

Our main object of study is the growth rate of $T_f(x,y)$ as $|x-y| \to \infty$, and how this rate depends on $f$. The usual setting for these questions is full lattice $(\mathbb{Z}^2,\mathbb{E}^2)$ and in that context, the order of the passage time is well-understood. Writing $T = T(x,y)$ for the corresponding passage time in the full lattice $(\mathbb{Z}^2, \mathbb{E}^2)$ and $\mathbf{e}_1 = (1,0)$, one typically defines the \underline{time constant} $\mu$, as
\[
\mu = \lim_{n \to \infty} \frac{T(0,n\mathbf{e}_1)}{n},
\]
which exists almost surely under a mild moment condition on $\tau_e$ \cite[Thm.~2.1]{50_years}, and exists in probability under no moment condition \cite[Thm.~2.23]{50_years}. Kesten showed \cite[Thm.~2.5]{50_years} that the behavior of $T$ is different in three different phases, depending on the relationship between $F(0)$ and $1/2$, the critical probability for bond percolation on $\mathbb{Z}^2$. If $F(0)<1/2$, all components induced by zero-weight edges are subcritical, thus finite, and $T$ grows linearly, so $\mu>0$. If $F(0) >1/2$, there is an infinite component induced by zero-weight edges and $T(0,n\mathbf{e}_1)$ is bounded, giving $\mu=0$. If $F(0) = 1/2$, zero-weight edges form finite but critical clusters, so they exist on all scales, and $\mu=0$. This last, critical, phase was studied in more detail by Zhang \cite{Z99}, who showed that there exist $F$ with $F(0)=1/2$ for which $T(0,n\mathbf{e}_1)$ diverges, and $F$ with $F(0)=1/2$ for which it is stochastically bounded. A necessary and sufficient condition for boundedness was given by Damron-Lam-Wang in \cite[Cor.~1.3]{DLW17}.

Ahlberg \cite[Prop.~3]{A15} was the first to consider FPP on the graphs $\mathbb{G}_f$. In contrast to the situation on the full lattice, he showed that this critical phase expands to an entire interval when $f(u) = a \log(1+u)$:
\begin{thm}[Ahlberg]\label{thm: Ahlberg}
For every $a>0$, the critical probability $p_c(f)$ for bond percolation on $\mathbb{G}_f$ is strictly between $1/2$ and $1$. However, if $F(0)\geq 1/2$ and $\mathbb{E}\tau_e<\infty$, then
\[
\lim_{n \to \infty} \frac{T_f(0,n\mathbf{e}_1)}{n} = 0 \text{ almost surely}.
\]
\end{thm}
\noindent
The theorem in \cite{A15} is stated with a different moment condition and with $n\mathbf{e}_1$ replaced by $n \mathbf{e}_1+\mathbf{e}_2$, but his methods apply to the above modification. The fact that $1/2 < p_c(f) < 1$ was originally found by Grimmett \cite{G83} (we explain this in detail in the next section), so the main content of the theorem is that for $F(0) \in (1/2,p_c(f))$, even though all zero-weight components are subcritical, the growth for $T_f$ is at most sublinear, unlike the linear growth of the subcritical phase on $\mathbb{Z}^2$. An obvious question follows from Ahlberg's result: what is the true growth rate of $T_f$ in these cases? The main contribution of this paper is to answer this question in the case of distribution functions $F$ with a ``gap'' near zero; that is, those $F$ satisfying \eqref{eq: gap_var_condition_1}. In Thm.~\ref{thm: main_mean}, we consider the more general 
\begin{equation}\label{eq: our_f}
f(u) = a \log(1+u) + b \log(1+ \log(1+u)) \text{ for } a > 0 \text{ and } b \geq 0,
\end{equation}
and provide asymptotics for the expected passage time for all  values of $F(0)$, $a$, and $b$, with constant prefactors that depend precisely on $a$ and $b$. We find that the growth depends on the relationship between $a,b,$ and the correlation length $\xi$ for subcritical percolation on $\mathbb{Z}^2$. In Thm.~\ref{thm: main_variance}, we prove asymptotics for the variance of the passage time which are of the same order as those for the mean, and give a central limit theorem, both for values of $a,b$ such that the expected value of $T_f$ diverges. This second result is a specific case of Thm.~\ref{thm: updated_KZ_variance} and Thm.~\ref{thm: updated_KZ_CLT}, valid for a larger class of functions $f$ satisfying assumptions A1-A3. A major tool in our proofs is Prop.~\ref{prop: rectangle_crossing_probability}, which gives asymptotics for rectangle crossing probabilities in subcritical Bernoulli percolation. These estimates, which use work of Campanino-Chayes-Chayes in \cite{CCC91}, also imply an improvement to ``sponge crossing'' dimension results of Grimmett from \cite{G81} from '81 in Cor.~\ref{cor: sponge}. Such connectivity results can be used directly for our study of FPP because of duality due to our assumed gap condition for $F$. We leave the case of general $F$ without a gap to future study.

\subsubsection{Percolation on wedge graphs}\label{sec: percolation_history}
The more general function $f$ in \eqref{eq: our_f} was introduced in an example of a discontinuous percolation phase transition on $\mathbb{G}_f$ by Chayes-Chayes \cite{CC86}. To explain this, and to describe the correlation length that we will use throughout, we now give background on percolation on $\mathbb{G}_f$. Wedge graphs were first considered by Grimmett \cite{G83}, where he explicitly solved for the critical probability of Bernoulli percolation on $\mathbb{G}_f$ when $f$ grows logarithmically. For a fixed $p \in [0,1]$, we declare each nearest neighbor edge of $\mathbb{Z}^2$ to be \underline{open} with probability $p$ and \underline{closed} with probability $1-p$, independently from edge to edge. Writing $\mathbb{P}_p$ for the probability measure corresponding to parameter $p$, the function
\[
\theta_f(p) = \mathbb{P}_p\left( 0 \text{ is in an infinite self-avoiding path of open edges in } \mathbb{G}_f\right)
\]
equals zero when $p$ is small and increases to 1 as $p \uparrow 1$. There is a critical value
\[
p_c(f) = \sup\{ p \in [0,1] : \theta_f(p) = 0\}
\]
satisfying $1/2 \leq p_c(f) \leq 1$ for any $f$, the lower bound following from the fact that the critical probability for Bernoulli percolation in $\mathbb{Z}^2$ is $1/2$ \cite{kesten_pc}.

The results of \cite{G83} give a formula for $p_c(f)$ in the case that $f(u) = a \log(1+u)$ for $a \in (0,\infty)$ and involve the \underline{correlation length} $\xi$, defined as
\begin{equation}\label{eq: correlation_length_def}
\xi(p) = \left( \lim_{n \to \infty} -\frac{1}{n} \log \mathbb{P}_p(0 \leftrightarrow n\mathbf{e}_1)\right)^{-1}.
\end{equation}
Here, $\{0 \leftrightarrow n\mathbf{e}_1\}$ is the event that there is a path of open edges in the full lattice $\mathbb{Z}^2$ connecting $0$ and $n\mathbf{e}_1 = (n,0)$, and the definition of $\xi(p)$ we give here is different, but equivalent to that in \cite{G83}. It is known that this limit exists and that $0 < \xi(p) < \infty$ for $p \in (0,1/2)$. Using $f(u) = a \log (1+u)$, \cite{G83} proves that
\[
\xi(1-p_c(f)) = a.
\]
Using properties of $\xi$, one can therefore show that $p_c(f)$ is a continuous and strictly decreasing function of $a$ satisfying $p_c(f) \uparrow 1$ as $a \downarrow 0$ and $p_c(f) \downarrow 1/2$ as $a \uparrow \infty$. A generalization of this result appears in \cite[Thm.~11.55]{grimmettbook} (see also the earlier \cite[Thm.~5.2]{CC86}) for functions $f$ satisfying $f(u)/\log u \to a$ as $u \to \infty$. In \cite{CC86}, Chayes-Chayes show that it is possible to choose $f$ in such a way that $\theta_f$ undergoes a discontinuous transition. Namely, as noted in \cite[p.~306]{grimmettbook}, if $f$ is given by \eqref{eq: our_f}, then for $b > 2a$, we have $\theta_f(p_c(f)) > 0$. (We improve this to $\theta_f(p_c(f)) > 0$ when $b >a$ and $\theta_f(p_c(f)) = 0$ when $b \leq a$; see Cor.~\ref{cor: phase_transition} below.) This result is in contrast to the (expected) continuous phase transition for percolation on $\mathbb{Z}^d$, currently proved only for $d=2$ \cite[Thm.~1]{kesten_pc} and $d \geq 11$ \cite[Cor.~1.3]{FH17}, but believed to be true for all $d \geq 2$.

In \cite{CC86}, Chayes-Chayes further analyze the phase diagram for percolation on (graphs very similar to) $\mathbb{G}_f$ for various $f$. They define the limit
\[
m_f(p) = \lim_{n \to \infty} -\frac{1}{n} \log \mathbb{P}_p(0 \leftrightarrow n\mathbf{e}_1 \text{ on } \mathbb{G}_f)
\]
and corresponding limit $m(p)$ for connection on the full lattice $\mathbb{Z}^2$. These $m_f(p)$ and $m(p)$ are reciprocals of correlation lengths like \eqref{eq: correlation_length_def}. Their argument in \cite[Thm.~3.1]{CC86} shows that $m_f(p) = m(p)$ for all $p \in [0,1]$ so long as $f(u) \to \infty$ as $u \to \infty$. Therefore their ``high temperature critical point''
\[
\pi_c(f) = \sup\{p \in [0,1] : m_f(p) > 0\}
\]
is equal to the corresponding critical point $\pi_c = 1/2$ for the full lattice. On the other side, however, the usual critical points $p_c(f)$ and $p_c$, which they call ``low temperature critical points,'' are different when $f$ grows logarithmically. In total, for logarithmic $f$, we have $\pi_c(f) = \pi_c = p_c = 1/2 < p_c(f)$ and there are three phases: a high temperature phase $(p < 1/2)$ where $\mathbb{P}_p(0 \leftrightarrow n\mathbf{e}_1 \text{ on } \mathbb{G}_f)$ decays to zero exponentially, an intermediate phase $(1/2 < p < p_c(f))$ where it decays to zero subexponentially, and a low temperature phase $(p_c(f) < p)$, where it is bounded away from zero.

There is a similar situation in FPP, and the corresponding quantities are the time constant and the probability that the passage time is zero. Namely, one can define the limit
\[
\mu_f = \lim_{n \to \infty} \frac{1}{n} \mathbb{E}T_f(0,n\mathbf{e}_1)
\]
and prove (see the discussion before \cite[Prop.~3]{A15} and also both \cite[Prop.~8]{A15} and \cite[Eq.~(3.4)]{CC86}) that $\mu_f = \mu$ so long as $f(u) \to \infty$ as $u \to \infty$. Because $\mu> 0$ if and only if $F(0) < 1/2$, the corresponding high temperature critical point, where $\mu$ transitions from positive to zero, is $1/2$ both for FPP on $\mathbb{Z}^2$ and for FPP on $\mathbb{G}_f$, assuming $f$ diverges. Regarding low temperature, $\liminf_{n \to \infty} \mathbb{P}(T_f(0,n\mathbf{e}_1) = 0)>0$ exactly when there is positive probability for an infinite component of zero-weight edges in $\mathbb{G}_f$, so like for percolation, the critical point is $p_c(f)$. Therefore if $f$ grows logarithmically, the low and high temperature critical points are different for FPP on $\mathbb{G}_f$. There are consequently also three phases: a high temperature phase ($F(0) < 1/2$) where $\mathbb{E} T_f(0,n\mathbf{e}_1)$ grows linearly, an intermediate phase ($1/2 < F(0) < p_c(f)$) where it grows sublinearly but diverges, and a low temperature phase $(p_c(f)<F(0))$ where it is bounded.

\subsection{Main results}

\subsubsection{Results for logarithmic wedges}

Our main results focus on the specific function $f$ satisfying \eqref{eq: f_conditions} given in \eqref{eq: our_f}. In this case, we write $\mathbb{G}_{a,b} = (\mathbb{V}_{a,b},\mathbb{E}_{a,b})$ for the graph $\mathbb{G}_f$ and $T_{a,b}$ for the passage time $T_f$. The weights $(\tau_e)$ will be assumed to have distribution function $F$ satisfying the ``gap'' condition
\begin{equation}\label{eq: gap_var_condition_1}
F(0^-) = 0 \text{ and } F(0) = F(\delta) = p \text{ for some } \delta>0 \text{ and } p \in \bigg[ \frac{1}{2},1\bigg),
\end{equation}
and the moment condition
\begin{equation}\label{eq: finite_mean_tau}
\int x~\text{d}F(x) < \infty.
\end{equation}

Our first result gives asymptotics for the expected passage time $\mathbb{E} T_{a,b}(0,P(n))$, where
\[
P(n) = \{(x_1,x_2) \in \mathbb{Z}^2 : x_1 = n\}.
\]
To state it, we will use the notation $f_1(n) \lesssim f_2(n)$ to mean that $\limsup_{n \to \infty} f_1(n)/f_2(n) \leq 1$.  If $x \in \mathbb{V}_{a,b}$ and $S \subset \mathbb{Z}^2$ with $S \cap \mathbb{V}_{a,b} \neq \emptyset$, we write $T_{a,b}(x,S)$ for $\inf_{y \in S \cap \mathbb{V}_{a,b}} T_{a,b}(x,y)$. Similarly for $S' \subset \mathbb{Z}^2$ with $S' \cap \mathbb{V}_{a,b} \neq \emptyset$, we write $T_{a,b}(S,S') = \inf_{x \in S \cap \mathbb{V}_{a,b}} T_{a,b}(x,S')$.

\begin{thm}\label{thm: main_mean}
Suppose $F$ satisfies \eqref{eq: gap_var_condition_1} and \eqref{eq: finite_mean_tau}. There exist $\Cl[smc]{c: ab_lower_bound}, \Cl[lgc]{c: ab_upper_bound} > 0$ such that for all $a > 0$ and $b \geq 0$, the following hold.
\begin{enumerate}
\item If $p = 1/2$, then
\[
\Cr{c: ab_lower_bound} \frac{n}{a \log n} \lesssim \mathbb{E}T_{a,b}(0,
P(n)) \lesssim \Cr{c: ab_upper_bound} \frac{n}{a \log n}.
\]
\item If $p > 1/2$, then
\begin{enumerate}
\item if $a < \xi(1-p)$, then
\[
\frac{\Cr{c: ab_lower_bound} }{\xi(1-p)-a}  \cdot \frac{n^{1- \frac{a}{\xi(1-p)}}}{\left( \log n \right)^{\frac{b}{\xi(1-p)}}} \lesssim \mathbb{E} T_{a,b}(0,P(n)) \lesssim \frac{\Cr{c: ab_upper_bound}}{\xi(1-p)-a}  \cdot \frac{n^{1- \frac{a}{\xi(1-p)}}}{\left( \log n \right)^{\frac{b}{\xi(1-p)}}},
\]
\item if $b < a = \xi(1-p)$, then
\[
\frac{\Cr{c: ab_lower_bound}}{\xi(1-p)-b} \left( \log n \right)^{1- \frac{b}{\xi(1-p)}} \lesssim \mathbb{E} T_{a,b}(0,P(n)) \lesssim \frac{\Cr{c: ab_upper_bound}}{\xi(1-p)-b} \left( \log n \right)^{1- \frac{b}{\xi(1-p)}},
\]
\item and if $b = a = \xi(1-p)$, then
\[
\Cr{c: ab_lower_bound} \log \log n \lesssim \mathbb{E} T_{a,b}(0,P(n)) \lesssim \Cr{c: ab_upper_bound} \log \log n.
\]
\end{enumerate}
\item Otherwise, $\sup_n \mathbb{E}T_{a,b}(0,P(n)) < \infty$.
\end{enumerate}
\end{thm}

We observe that $\mathbb{E}T_{a,b}(0,P(n)) \geq \mathbb{E}\tau_e$ because each path from $0$ to $P(n)$ in $\mathbb{G}_{a,b}$ contains the edge $\{0,\mathbf{e}_1\}$ (due to \eqref{eq: f_conditions}). Furthermore $\mathbb{E}T_{a,b}(0,P(n)) \leq n \mathbb{E}\tau_e$ by considering the horizontal path starting at 0. Therefore the condition \eqref{eq: finite_mean_tau} is necessary for the statement of Thm.~\ref{thm: main_mean}.

Although the estimates in Thm.~\ref{thm: main_mean} have constant factors that depend precisely on $a$ and $b$, it would still be of interest to show even stronger asymptotics. For example, if $p>1/2$ and $b=a=\xi(1-p)$, one could try to show that $\mathbb{E}T_{a,b}(0,P(n)) / \log \log n$ converges as $n \to \infty$ to a $p$-dependent constant (and similar statements for the other cases). This problem seems approachable for site-FPP on the triangular lattice in the case $p=1/2$, as Yao \cite[Thm.~1.1]{Y18} has already shown a law of large numbers for the passage time in critical FPP on the full lattice, and Jiang-Yao \cite[Cor.~1]{JY19} have shown corresponding results for sectors. An improvement in the case $p>1/2$ would seem to require stronger asymptotics for the expected number of left-right open crossings of tall thin rectangles in subcritical percolation than those we give in Cor.~\ref{cor: X_n_bounds_subcritical}. Our bounds use the point-to-plane connection probability from Prop.~\ref{prop: point_to_plane_probability}, which is estimated from below and  above by the functions $\mathbb{H}_n(p)$ and $\mathbb{G}_n(p)$ respectively. Although the authors of \cite{CCC91} show strong asymptotics for both of these functions in \cite[Thm.~4.4(I)]{CCC91} and \cite[Thm.~6.2(I)]{CCC91}, their $p$-dependent constant prefactors are different. 

\begin{rem}
The function $f$ from \eqref{eq: our_f} nonnegative and increasing for $u \geq 0$ when $a > 0$ and $b \geq -a$. Although Thm.~\ref{thm: main_mean} is stated for $a > 0$ and $b \geq 0$, the conclusions still hold when $b \geq -a$. For $b < -a$, $f$ has a global minimum at $u=\exp(b/a - 1)-1$ that is negative, but $f$ is nonnegative and increasing for $u$ at least equal to some $u_{a,b}$. In that case, one can still consider the passage time between $(m,0)$ and $P(n)$ for fixed $m \geq u_{a,b}$ and prove a version of Thm.~\ref{thm: main_mean}. The degree of $(m,0)$ in $\mathbb{G}_f$ may not be one and therefore the optimal moment condition can be weaker.
\end{rem}

\begin{rem}
One can consider FPP on the graph $\mathbb{G}_f$ when $f$ is defined using more iterated logarithms, for example $f(u) = a\log(1+u) + b \log (1+ \log (1+u)) + c \log (1+\log(1 + \log(1+u)))$. The methods of this paper, along with more complicated computations, can be used to give asymptotics for $\mathbb{E}T_f(0,P(n))$ for such $f$ as well.
\end{rem}

\begin{rem}\label{rem: subcritical}
Thm.~\ref{thm: main_mean} deals with the case $p \in [1/2,1)$. To deal with smaller $p$, observe that $|T_{a,b}(0,P(n)) - T_{a,b}(0,n\mathbf{e}_1)|$ is no larger than the sum of $\tau_e$ for $e$ with both endpoints in $\mathbb{V}_{a,b} \cap P(n)$. This implies that $\mathbb{E}T_{a,b}(0,P(n))/n$ and $\mathbb{E}T_{a,b}(0,n\mathbf{e}_1)/n$ have the same limit. Because the latter has limit $\mu>0$ when $p<1/2$ (from Sec.~\ref{sec: percolation_history}), we conclude that in this case, $\mathbb{E}T_{a,b}(0,P(n))$ grows linearly. 
\end{rem}

The second result gives asymptotics for the variance of $T_{a,b}(0,P(n))$ and a CLT in the case that the mean diverges. We will assume the moment condition
\begin{equation}\label{eq: gap_var_condition_2}
\int x^\eta~\text{d}F(x) < \infty \text{ for some } \eta > 4.
\end{equation}

\begin{thm}\label{thm: main_variance}
Suppose $F$ satisfies \eqref{eq: gap_var_condition_1} and \eqref{eq: gap_var_condition_2}. For all $a>0$ and $b \geq 0$ such that $\mathbb{E}T_{a,b}(0,P(n)) \to \infty$ as $n \to \infty$, we have
\[
\frac{T_{a,b}(0,P(n))- \mathbb{E}T_{a,b}(0,P(n))}{\sqrt{\mathrm{Var}~T_{a,b}(0,P(n))}} \Rightarrow N(0,1),
\]
where $N(0,1)$ denotes the standard normal distribution. Furthermore, there exist $\Cl[smc]{c: var_lower_main}, \Cl[lgc]{c: var_upper_main}$ such that for all such $a,b$,
\begin{equation}\label{eq: variance_asymptotic_main}
\Cr{c: var_lower_main} \mathbb{E}T_{a,b}(0,P(n)) \lesssim \mathrm{Var}~T_{a,b}(0,P(n)) \lesssim \Cr{c: var_upper_main} \mathbb{E}T_{a,b}(0,P(n)).
\end{equation}
\end{thm}

\begin{rem}
Under the assumptions of Thm.~\ref{thm: main_variance}, if $a > 0$ and $b \geq 0$ are such that $\sup_n \mathbb{E}T_{a,b}(0,P(n)) < \infty$  (as in case 3 of Thm.~\ref{thm: main_mean}), then also $\sup_n \mathrm{Var}~T_{a,b}(0,P(n)) < \infty$. This statement follows from the argument in the proof of Lem.~\ref{lem: plane_to_plane_estimate}, and we briefly outline it here. Writing $T_{a,b}^B$ for the coupled Bernoulli passage time defined at the beginning of Sec.~\ref{sec: mean_proof}, for such $a,b$, there must be $K = K_{a,b}$ such that $\mathbb{E}T_{a,b}^B(0,P(n)) \leq K$ for all $n$. By representing $T_{a,b}^B(0,P(n))$ in terms of disjoint closed separating sets as in \eqref{eq: more_separating_equivalence} and applying the BK inequality, we obtain $C,c>0$ such that for all $n$ and all integers $z$,
\[
\mathbb{P}(T_{a,b}^B(0,P(n)) \geq z) \leq C \exp\left( - c \frac{z}{K}\right).
\]
(Compare to \eqref{eq: z_bound}.) This implies that all moments of $T_{a,b}^B(0,P(n))$ are bounded. We apply this in the inequality 
\begin{align*}
\mathbb{P}(T_{a,b}(0,P(n)) \geq y) &\leq \mathbb{P}(T_{a,b}^B(0,P(n)) \geq z) \\
&+ \mathbb{P}(T_{a,b}(0,P(n)) < z, T_{a,b}(0,P(n)) \geq y)
\end{align*}
for integer $z$ and real $y \geq 0$ (see \eqref{eq: plane_to_plane_term_1} and \eqref{eq: plane_to_plane_term_2}) and follow the argument in the proof of Lem.~\ref{lem: plane_to_plane_estimate} step-by-step from \eqref{eq: z_bound} onward. This reasoning results in the inequality
\[
\mathbb{P}(T_{a,b}(0,P(n)) \geq y) \leq C \exp\left( - c \frac{y}{K}\right) + \frac{C}{y^{\frac{\eta}{2}}}
\]
for some $c,C>0$, all $y \geq 0$, and all $n$, which corresponds to that in the statement of Lem.~\ref{lem: plane_to_plane_estimate}. Because $\eta$ is assumed to be $>4$, this suffices to show that $T_{a,b}(0,P(n))$ has uniformly bounded second moment.
\end{rem}

\begin{rem}\label{rem: moment_condition}
The moment condition \eqref{eq: gap_var_condition_2} in Thm.~\ref{thm: main_variance} can be improved to $\mathbb{E}t_e^2<\infty$, which is optimal because $\mathbb{E}T_{a,b}(0,P(n))^2 \geq \mathbb{E}t_e^2$, as the origin has degree one in $\mathbb{G}_{a,b}$. The way to do this is to improve the exponent in the upper bound $C/y^{\eta/2}$ in Lem.~\ref{lem: plane_to_plane_estimate} by a factor of three to $C/y^{s/2}$ for some $C = C(s)$ and any $s \in [2,3\eta)$. We state this improved inequality and outline the modifications needed to prove it in \eqref{eq: plane_to_plane_improved}. These modifications require $f(m) \geq 3$ between two vertical lines (that is, $r_i' \leq m \leq r_j'$). Because our $f$ in \eqref{eq: our_f} is only eventually $\geq 3$, \eqref{eq: plane_to_plane_improved} leads to a reduction of the condition $\eta>4$ in \eqref{eq: gap_var_condition_2} to $\eta>4/3$ for parts of the proof of Thm.~\ref{thm: main_variance} that take place sufficiently far from the origin. We must use the assumption $\mathbb{E}t_e^2<\infty$ near the origin, resulting in the overall moment condition $\mathbb{E}t_e^2<\infty$. 
\end{rem}

The last result is a corollary that characterizes continuity of the percolation phase transition on $\mathbb{G}_{a,b}$, and improves on the results of Chayes-Chayes listed in \cite[p.~306]{grimmettbook}.
\begin{cor}\label{cor: phase_transition}
On $\mathbb{G}_{a,b}$ for $a>0$ and $b \geq 0$, the percolation phase transition is discontinuous if and only if $b > a$.
\end{cor}
\begin{proof}
If $b > a$, then when $p = p_c(f)$, we have $b > a = \xi(1-p)$. By item 3 of Thm.~\ref{thm: main_mean}, $\lim_{n \to \infty} T_{a,b}(0,P(n)) < \infty$ a.s., and because of the gap condition \eqref{eq: gap_var_condition_1}, there a.s.~exists an infinite component of zero-weight edges. We conclude that if $b>a$, then the percolation phase transition on $\mathbb{G}_{a,b}$ is discontinuous.

If $b =a$, then when $p = p_c(f)$, we have $b=a=\xi(1-p)$. In this case, by Thm.~\ref{thm: main_mean}, $\mathbb{E}T_{a,b}(0,P(n))$ is of order $\log \log n$, and by Thm.~\ref{thm: main_variance}, $\mathrm{Var}~T_{a,b}(0,P(n))$ is also of order $\log \log n$. It follows that $T_{a,b}(0,P(n)) \to \infty$ in probability, and by monotonicity, $\lim_{n \to \infty}T_{a,b}(0,P(n)) = \infty$ almost surely. This implies that a.s., there is no infinite component of zero-weight edges. By monotonicity, the percolation phase transition is continuous on $\mathbb{G}_{a,b}$ when $b \leq a$.
\end{proof}

\subsubsection{Results for general wedges}\label{sec: KZ_main_section}

To state our general result on the variance and CLT, we need some definitions and assumptions. Suppose that $f$ satisfies \eqref{eq: f_conditions}. Given a sequence of integers $(r_i)_{i=0}^\infty$ satisfying
\begin{equation}\label{eq: r_i_increasing}
0 = r_0 < r_1 < r_2 < \dots,
\end{equation}
we write $R_0, R_1, R_2, \dots$ for the vertex sets
\begin{equation}\label{eq: R_i_def}
R_i = \{(x_1,x_2) \in \mathbb{V}_f : r_i \leq x_1 \leq r_{i+1}\}.
\end{equation}
Our assumptions on $(r_i)$ will involve passage times in the related Bernoulli FPP model. For this purpose, let $(t_e)_{e \in \mathbb{E}^2}$ be a family of i.i.d.~Bernoulli random variables with $\mathbb{P}(t_e=0) = p = 1-\mathbb{P}(t_e=1)$, where $p=F(0)$ and $F$ is again the common distribution function of the variables $(\tau_e)$. We will write $T_f^B(x,y)$ (similarly $T_f^B(x,S)$ and $T_f^B(S,S')$) for the passage time $T_f(x,y)$, using the weights $(t_e)$ instead of $(\tau_e)$. 

We say that a path $\gamma$ with all of its vertices in $R_i$ is a \underline{top-down crossing of $R_i$} if the first vertex of $\gamma$ is in the ``top'' set $\{(x_1,x_2) \in R_i : x_2 + 1 \notin R_i\}$ and the last vertex of $\gamma$ is in the ``bottom'' set $\{(x_1,x_2) \in R_i : x_2=0\}$. In a given configuration $(t_e)$, we say that a top-down crossing of $R_i$ is an open top-down crossing of $R_i$ if all its edges $e$ have $t_e = 0$.

Our main assumptions on the sequence $(r_i)$ are as follows. There exist constants $\Cl[smac]{c: r_i_assumption_1} > 0, \Cl[lgac]{c: r_i_assumption_2} \in [1,\infty)$ satisfying
\begin{enumerate}
\item[A1.] $\liminf_{i \to \infty} \mathbb{P}(\exists \text{ top-down open crossing of } R_i) \geq \Cr{c: r_i_assumption_1}$,
\item[A2.] $\limsup_{i \to \infty} \left[ \mathbb{E} T_f^B(0, P(r_{i+1}))-\mathbb{E} T_f^B(0,P(r_i))\right] \leq \Cr{c: r_i_assumption_2}$,
\end{enumerate}
and furthermore, there exists a function $\Cr{c: r_i_assumption_3} : [0,\infty) \to (0,\infty)$ such that
\begin{enumerate}
\item[A3.] $\liminf_{i \to \infty} \mathbb{P}(T_f^B(P(r_i),P(r_{i+1})) \geq M) \geq \Cl[smac]{c: r_i_assumption_3}(M)$ for each $M \geq 0$.
\end{enumerate}

Our first result gives asymptotics for the variance of $T_f(0,P(n))$ for general $f$ under assumptions A1-A3. To state it, if $n \geq 0$, we define
\begin{equation}\label{eq: iota_def}
\iota(n) = \min\{i : r_{2i} \geq n\}.
\end{equation}

\begin{thm}\label{thm: updated_KZ_variance}
Suppose that \eqref{eq: gap_var_condition_1} and \eqref{eq: gap_var_condition_2} hold. Given $\Cr{c: r_i_assumption_1} > 0, \Cr{c: r_i_assumption_2} \in [1,\infty)$, and $\Cr{c: r_i_assumption_3}:[0,\infty) \to (0,\infty)$, there exist $\Cl[lgc]{c: variance_upper_bound_constant}, \Cl[smc]{c: variance_lower_bound_constant} > 0$ such that the following holds. For any $f$ and $(r_i)$ satisfying \eqref{eq: f_conditions}, \eqref{eq: r_i_increasing}, and assumptions A1-A3 for $\Cr{c: r_i_assumption_1}, \Cr{c: r_i_assumption_2}$, and $\Cr{c: r_i_assumption_3}$, we have
\[
\Cr{c: variance_lower_bound_constant} \iota(n) \lesssim \mathrm{Var}~T_f(0,P(n)) \lesssim \Cr{c: variance_upper_bound_constant} \iota(n),
\]
and
\[
\Cr{c: variance_lower_bound_constant} \mathbb{E}T_f(0,P(n)) \lesssim \mathrm{Var}~T_f(0,P(n)) \lesssim \Cr{c: variance_upper_bound_constant} \mathbb{E}T_f(0,P(n)).
\]
\end{thm}

The second result is a CLT for $T_f(0,P(n))$ under assumptions A1-A3.
\begin{thm}\label{thm: updated_KZ_CLT}
Suppose that \eqref{eq: gap_var_condition_1} and \eqref{eq: gap_var_condition_2} hold. If $f$ and $(r_i)$ satisfy \eqref{eq: f_conditions}, \eqref{eq: r_i_increasing}, and assumptions A1-A3 for some $\Cr{c: r_i_assumption_1}>0$, $\Cr{c: r_i_assumption_2} \in [1,\infty)$, and $\Cr{c: r_i_assumption_3} : [0,\infty) \to (0,\infty)$, then we have
\[
\frac{T_f(0,P(n)) - \mathbb{E}T_f(0,P(n))}{\sqrt{\mathrm{Var}~T_f(0,P(n))}} \Rightarrow N(0,1) \text{ as } n \to \infty,
\]
where $N(0,1)$ denotes the standard normal distribution.
\end{thm}

\begin{rem}
We verify assumptions A1-A3 of Thm.~\ref{thm: updated_KZ_CLT} and Thm.~\ref{thm: updated_KZ_variance} for the function $f$ from \eqref{eq: our_f} in Sec.~\ref{sec: variance_proof} to deduce Thm.~\ref{thm: main_variance}. The method of Sec.~\ref{sec: variance_proof} can also be used to verify A1-A3 for a wide range of choices of $f$, for example (a) $f(u) = u^a$ with $a < 1$ and $F(0)=p_c(f) = 1/2$, (b) $f(u) = (\log (1+u))^a$ with $a > 1$ and $F(0) = p_c(f) = 1/2$, and (c) $f(u) = (\log (1+u))^a$ with $a < 1$ and $1/2 \leq F(0) < 1 = p_c(f)$.
\end{rem}

\subsection{Outline of the paper}
In Sec.~\ref{sec: subcritical}, we give asymptotics for connectivity probabilities in subcritical bond percolation in general dimensions. These are used throughout the rest of the paper. Most are consequences of the study of Campanino-Chayes-Chayes \cite{CCC91} in '91 but we also use them to improve ``sponge crossing dimensions'' results of Grimmett from '81 in Cor.~\ref{cor: sponge}. In Sec.~\ref{sec: mean_proof} we prove Thm.~\ref{thm: main_mean}. In Sec.~\ref{sec: general_f_proof}, we prove Thm.~\ref{thm: updated_KZ_variance} and Thm.~\ref{thm: updated_KZ_CLT} on general wedge graphs and in Sec.~\ref{sec: variance_proof}, we use the results of Sec.~\ref{sec: general_f_proof} to prove Thm.~\ref{thm: main_variance}.

\section{Subcritical percolation estimates}\label{sec: subcritical}

Here we will use some of the results of the Campanino-Chayes-Chayes paper \cite{CCC91} to estimate subcritical box-crossing probabilities in bond percolation in general dimension $d \geq 2$. As usual, we write $\mathbb{E}^d$ for the nearest neighbor edges of $\mathbb{Z}^d$ and consider the probability space $\{0,1\}^{\mathbb{E}^d}$ with the cylinder sigma-algebra. Using $\omega$ for a typical outcome and letting $p  \in [0,1]$, we denote by $\mathbb{P}_p$ the probability measure under which the random variables $\omega(e)$ are i.i.d.~with
\[
\mathbb{P}_p(\omega(e) = 1) = p = 1 - \mathbb{P}_p(\omega(e) = 0).
\]
Write $\mathbb{E}_p$ for expectation relative to $\mathbb{P}_p$. For an edge $e$, we say that $e$ is \underline{open} (in $\omega$) if $\omega(e)=1$ and $e$ is \underline{closed} otherwise. (Observe that $e$ is open when $\omega(e)=1$, whereas in Sec.~\ref{sec: KZ_main_section}, $e$ is open when $t_e=0$.) A \underline{path} from $x$ to $y$ is an alternating sequence of vertices and edges $x_0,e_0,x_1, \dots, x_{k-1},e_{k-1},x_k$, where $x_0=x$, $x_k = y$, and each $e_i$ has endpoints $x_i$ and $x_{i+1}$. A path is said to be open if all its edges are open. If there is an open path from $x$ to $y$, we say that $x$ and $y$ are connected by an open path and we write $x \leftrightarrow y$. The \underline{open cluster} of $x$ is the set
\[
C(x) = \{y \in \mathbb{Z}^d : x \leftrightarrow y\}.
\]
If $S \subset \mathbb{Z}^d$ is a set of vertices, we will also write $x \leftrightarrow S$ if there is an open path that starts at $x$ and ends at some vertex of $S$. As usual, there is a critical threshold $p_c \in (0,1)$ defined as $p_c = \sup\{p \in [0,1] : \mathbb{P}_p(\#C(0) = \infty) = 0\}$ and we focus here on the subcritical case $p  \in (0,p_c)$.

\subsection{Results of Campanino-Chayes-Chayes}

Now we can state some of the results of \cite{CCC91}. For an integer $n \geq 0$, define the hyperplane and slab
\begin{equation}\label{eq: S_n_P_n_def}
P(n) = \{x \in \mathbb{Z}^d : x \cdot \mathbf{e}_1 = n\} \text{ and } S(n) = \bigcup_{i=0}^n P(i),
\end{equation}
where $\mathbf{e}_1$ is the first coordinate vector. Let
\[
\mathbb{G}_n(p) = \sum_{x \in P(n)} \mathbb{P}_p(0 \leftrightarrow x).
\]
The first result is \cite[Thm.~6.2(I)]{CCC91}. Together with most of the notes here, it will use the correlation length $\xi$, defined in \eqref{eq: correlation_length_def}. As mentioned in Sec.~\ref{sec: percolation}, it is known that this limit defining $\xi(p)$ exists and that $0 < \xi(p) < \infty$ for $p \in (0,p_c)$.

\begin{prop}[Campanino-Chayes-Chayes] \label{prop: expectation_plane}
Let $p \in (0,p_c)$. There exist $K_2 = K_2(p) \geq 1$ and $\Delta = \Delta(p) >0$ such that
\[
\left| \mathbb{G}_n(p) \exp\left( \frac{n}{\xi(p)}\right) - K_2(p) \right| \leq e^{-\Delta(p)n}.
\]
\end{prop}
The above proposition states that the expectation 
\[
\mathbb{G}_n(p) = \mathbb{E}_p \#\{x \in P(n) : 0 \leftrightarrow x\}
\]
is very close to exponential. We will only need the weaker statement that if $p \in (0,p_c)$, there exist $\Cl[smc]{c: G_n_lower_constant} = \Cr{c: G_n_lower_constant}(p) > 0$ and $\Cl[lgc]{c: G_n_upper_constant} = \Cr{c: G_n_upper_constant}(p) > 0$ such that for all $n$, we have
\begin{equation}\label{eq: weaker_G_n_inequalities}
\Cr{c: G_n_lower_constant} \exp\left( - \frac{n}{\xi(p)}\right) \leq \mathbb{G}_n(p) \leq \Cr{c: G_n_upper_constant} \exp\left( - \frac{n}{\xi(p)}\right).
\end{equation}

Because long open paths in subcritical percolation are quite straight, we expect them to intersect hyperplanes in only a few vertices. Therefore we expect from Prop.~\ref{prop: expectation_plane} that the point-to-plane connection probability, $\mathbb{P}_p(0 \leftrightarrow P(n))$, is also very close to exponential. To state the result of \cite{CCC91} on this, we need more definitions.

For $x \in S(n)$, we write $C(x)||_{S(n)}$ for the open cluster of $x$ restricted to $S(n)$; that is,
\[
C(x)||_{S(n)} = \{y \in \mathbb{Z}^d : \exists \text{ open path from } x \text{ to } y \text{ which uses only vertices in }S(n)\}.
\]
Note that in general, $C(x)||_{S(n)}$ will be smaller than the intersection $C(x) \cap S(n)$. Next we define, for $x \in P(0)$ and $y \in P(n)$, the ``point-to-point cylinder connectivity function'' (we will not need to use this term, but it is the one given in \cite{CCC91}) as
\[
h_{x,y}(p) = \mathbb{P}_p\left(C(x)||_{S(n)} \cap P(0) = \{x\} \text{ and } C(x)||_{S(n)} \cap P(n) = \{y\}\right).
\]
This is the probability that $x \leftrightarrow y$, but the connection occurs within $S(n)$; furthermore, the restricted cluster $C(x)||_{S(n)}$ intersects $P(0)$ only at $\{x\}$ and $P(n)$ only at $\{y\}$. Note that if the event in the probability occurs then, in particular, all edges incident to $x$ in $P(0)$ must be closed, and all edges incident to $y$ in $P(n)$ must be closed. The authors of \cite{CCC91} often refer to these conditions as ``strict cylinder'' boundary conditions. Similar to the definition of $\mathbb{G}_n(p)$, we have the definition
\[
\mathbb{H}_n(p) = \sum_{x \in P(n)} h_{0,x}(p).
\]
There is an asymptotic result for $\mathbb{H}_n(p)$ analogous to Prop.~\ref{prop: expectation_plane} stated in \cite[Thm.~4.4(I)]{CCC91}, but we will only need the simpler \cite[Prop.~3.4(iii)]{CCC91}:
\begin{prop}[Campanino-Chayes-Chayes] \label{prop: strict_connection_plane}
Let $p \in (0,p_c)$. There exists $\beta = \beta(p)>0$ such that for every $n$, we have
\[
\beta(p) \mathbb{G}_n(p) \leq \mathbb{H}_n(p) \leq \mathbb{G}_n(p).
\]
\end{prop}

\subsection{Consequences for connection probabilities}

From Prop.~\ref{prop: expectation_plane} and Prop.~\ref{prop: strict_connection_plane} we can derive some asymptotics for some connection events that are more general than point-to-point. First we have the point-to-plane connection probability.

\begin{prop}\label{prop: point_to_plane_probability}
Let $p \in (0,p_c)$. There exist $\Cl[smc]{c: point_to_plane_lower} = \Cr{c: point_to_plane_lower}(p) > 0$ and $\Cl[lgc]{c: point_to_plane_upper} = \Cr{c: point_to_plane_upper}(p) > 0$ such that for all $n$,
\[
\Cr{c: point_to_plane_lower} \exp\left( - \frac{n}{\xi(p)}\right) \leq \mathbb{P}_p(0 \leftrightarrow P(n)) \leq \Cr{c: point_to_plane_upper} \exp\left( - \frac{n}{\xi(p)} \right).
\]
\end{prop}
\begin{proof}
By \eqref{eq: weaker_G_n_inequalities} and Prop.~\ref{prop: strict_connection_plane}, we have
\begin{align*}
\Cr{c: G_n_lower_constant} \beta(p) \exp\left( - \frac{n}{\xi(p)}\right) \leq \mathbb{H}_n(p) &= \sum_{x \in P(n)} h_{0,x}(p) \\
&\leq \sum_{x \in P(n)} \mathbb{P}_p( 0 \leftrightarrow P(n), C(0)||_{S(n)} \cap P(n) = \{x\}) \\
&\leq \mathbb{P}_p(0 \leftrightarrow P(n)) \\
&\leq \mathbb{G}_n(p) \\
&\leq \Cr{c: G_n_upper_constant} \exp\left( - \frac{n}{\xi(p)}\right).
\end{align*}
\end{proof}

From Prop.~\ref{prop: point_to_plane_probability}, we can quickly find asymptotics for the point-to-box connection probability. Let $B(n) = [-n,n]^d \cap \mathbb{Z}^d$ and write
\[
\partial B(n) = \{x \in B(n) : \exists ~y \in \mathbb{Z}^d \setminus B(n) \text{ with } |x-y|=1\}.
\]
\begin{prop}\label{prop: point_to_box_probability}
Let $p \in (0, p_c)$. There exist $\Cl[smc]{c: point_to_box_lower} = \Cr{c: point_to_box_lower}(p) > 0$ and $\Cl[lgc]{c: point_to_box_upper} = \Cr{c: point_to_box_upper}(p) > 0$ such that for all $n$, we have
\[
\Cr{c: point_to_box_lower} \exp\left( - \frac{n}{\xi(p)} \right) \leq \mathbb{P}_p(0 \leftrightarrow \partial B(n)) \leq \Cr{c: point_to_box_upper} \exp\left( - \frac{n}{\xi(p)} \right).
\]
\end{prop}
\begin{proof}
On the one hand, if $0 \leftrightarrow P(n)$, then $0 \leftrightarrow \partial B(n)$, so from Prop.~\ref{prop: point_to_plane_probability}, we have
\[
\mathbb{P}_p(0 \leftrightarrow \partial B(n)) \geq \mathbb{P}_p( 0 \leftrightarrow P(n)) \geq \Cr{c: point_to_plane_lower} \exp\left( - \frac{n}{\xi(p)}\right).
\]
On the other hand, if $0 \leftrightarrow \partial B(n)$, then $0$ must be connected by an open path to at least one of $2d$ many rotations of $P(n)$ of the form $\{x \in \mathbb{Z}^d : x \cdot \mathbf{e}_i = (-1)^k n\}$ for $k=0,1$ and $i=1, \dots, d$, so by symmetry and a union bound, we obtain
\[
\mathbb{P}_p(0 \leftrightarrow \partial B(n)) \leq 2d \mathbb{P}_p(0 \leftrightarrow P(n)) \leq 2d \Cr{c: point_to_plane_upper} \exp\left( - \frac{n}{\xi(p)}\right).
\]
\end{proof}

In the next proposition, we consider probabilities of box-crossing events. For positive integers $n$ and $h(n)$, define the rectangular box
\[
R_n = \left( [0,n] \times [0,h(n)]^{d-1}\right) \cap \mathbb{Z}^d.
\]
We say that $R_n$ has a left-right open crossing if there is an open path which uses only vertices of $R_n$ and connects a vertex on the left side $\{0\} \times [0,h(n)]^{d-1}$ to a vertex on the right side $\{n\} \times [0,h(n)]^{d-1}$. Write 
\[
\Phi(s) = \frac{s}{s+1} \text{ for } s > 0.
\]
\begin{prop}\label{prop: rectangle_crossing_probability}
Let $p \in (0,p_c)$. There exist $\Cl[smc]{c: rectangle_crossing_lower} = \Cr{c: rectangle_crossing_lower}(p) > 0$ and $\Cl[lgc]{c: rectangle_crossing_upper} = \Cr{c: rectangle_crossing_upper}(p) > 0$ such that for all $n$ and all choices of $h(n) \geq 6n$, we have
\begin{align*}
\Phi\left( \Cr{c: rectangle_crossing_lower} h(n)^{d-1} \exp\left( - \frac{n}{\xi(p)}\right)\right) &\leq \mathbb{P}_p(R_n \text{ has a left-right open crossing}) \\
&\leq \Cr{c: rectangle_crossing_upper} h(n)^{d-1} \exp\left( - \frac{n}{\xi(p)}\right).
\end{align*}
\end{prop}

The results of Prop.~\ref{prop: point_to_box_probability} and Prop.~\ref{prop: rectangle_crossing_probability} improve on those used in \cite[Sec.~11.5]{grimmettbook}. This is the main reason we can show (mentioned in Cor.~\ref{cor: phase_transition}) that the percolation phase transition is continuous on $\mathbb{G}_{a,b}$ if and only if $b \leq a$, rather than just showing discontinuity for $b > 2a$. Concretely, the lower bound in \cite[(11.58)]{grimmettbook} is $\mathbb{P}_p([0,n]^2 \cap \mathbb{Z}^2 \text{ has a left-right open crossing}) \geq (1/4)\mathbb{P}_p(0 \leftrightarrow \partial B(n/2))^2$, which was known to be of order $e^{-(n/\xi(p))+o(1)}$. (Even from Prop.~\ref{prop: point_to_box_probability}, this is of order $e^{-n/\xi(p)}$.) The corresponding lower bound in Prop.~\ref{prop: rectangle_crossing_probability} (for a $6n$ by $n$ rectangle) is $ne^{-n/\xi(1-p)}$. Similarly, the upper bound used in \cite[Sec.~11.5]{grimmettbook} is from \cite[(6.10)]{grimmettbook}, which states that $\mathbb{P}_p(0 \leftrightarrow \partial B(n))\leq Cne^{-n/\xi(p)}$. The corresponding estimate coming from Prop.~\ref{prop: point_to_box_probability} is $\mathbb{P}_p(0 \leftrightarrow \partial B(n)) \leq C e^{-n/\xi(p)}$.

\begin{rem}\label{rem: more_boxes}
We will only use Prop.~\ref{prop: rectangle_crossing_probability} in the case of $h(n) \gg n$, but the result can also be shown for cubes like $B(n)$, or even thin rectangular boxes with certain $h(n) \ll n$. The proof would need to be modified by taking into account that the correlation norm $\xi_p$, defined on $\mathbb{R}^2$ by
\[
\xi_p(x) = \lim_{n \to \infty} - \frac{1}{n} \log \mathbb{P}_p( 0 \leftrightarrow [nx]),
\]
$[nx]$ is the closest integer point to $nx$, is strictly convex \cite[Lem.~4.4]{CI02}. (Using that $\{x : \xi_p(x) \leq 1\}$ has Gaussian curvature bounded away from 0 \cite[Lem.~4.3]{CI02}, it is likely that $h(n)$ could be taken $\gg n^{1/2}$.) This would be used in place of \eqref{eq: confinement_pasta} to confine the connection from $0$ to $P(n)$ inside the rectangular box $\widetilde{R}_n$, which is defined above \eqref{eq: confinement_pasta}.
\end{rem}

\begin{proof}[Proof of Prop.~\ref{prop: rectangle_crossing_probability}]
By translating and using monotonicity of crossing probabilities, it suffices to show both inequalities after replacing $R_n$ by the more convenient
\[
R_n' = ([0,n] \times [-3h(n),3h(n)]^{d-1})\cap \mathbb{Z}^d \text{ with } h(n) \geq n.
\]
The proof of the upper bound follows from a union bound and Prop.~\ref{prop: point_to_plane_probability}:
\begin{align}
\mathbb{P}_p(R_n' \text{ has a left-right open crossing}) &\leq \sum_{\vec{z} \in [-3h(n),3h(n)]^{d-1} \cap \mathbb{Z}^{d-1}} \mathbb{P}((0,\vec{z}) \leftrightarrow P(n)) \nonumber \\
&= (6h(n)+1)^{d-1} \mathbb{P}_p(0 \leftrightarrow P(n)) \nonumber \\
&\leq \Cr{c: point_to_plane_upper} (6h(n)+1)^{d-1} \exp\left( - \frac{n}{\xi(p)}\right). \label{eq: crossing_upper_bound}
\end{align}
Here we have used the notation $(0,\vec{z})$ for $\vec{z} = (z(1), \dots, z(d-1))$ to refer to the point $(0,z(1), \dots, z(d-1))$.

For the lower bound, we begin by applying \eqref{eq: weaker_G_n_inequalities} and Prop.~\ref{prop: strict_connection_plane} to obtain
\begin{align}
&\Cr{c: G_n_lower_constant} \beta(p) \exp\left( - \frac{n}{\xi(p)}\right) \nonumber \\
\leq~& \mathbb{H}_n(p) \nonumber \\
=~& \sum_{x \in P(n)} \mathbb{P}_p(C(0)||_{S(n)} \cap P(0) = \{0\}, C(0)||_{S(n)} \cap P(n) = \{x\}) \nonumber \\
\leq~& \mathbb{P}_p(C(0)||_{S(n)} \cap P(0) = \{0\}, C(0)||_{S(n)} \cap P(n) \neq \emptyset). \label{eq: pasta_macaroni}
\end{align}
In the event $\{C(0)||_{S(n)} \cap P(n) \neq \emptyset \}$ which appears above, there is an open path from $0$ to $P(n)$ that remains in $S(n)$, but we want it to remain in $R_n'$. In fact, we would like it to lie in an even smaller rectangle $\widetilde{R}_n$, which we define as $\widetilde{R}_n = \left( [0,n] \times [-2h(n),2h(n)]^{d-1}\right) \cap \mathbb{Z}^d$. Observe that by symmetry and Prop.~\ref{prop: point_to_plane_probability}, we have
\begin{align}
\mathbb{P}_p(C(0)||_{S(n)} \cap P(n) \neq \emptyset,~ C(0)||_{\widetilde{R}_n} \cap P(n) = \emptyset) &\leq 2(d-1)\mathbb{P}_p(0 \leftrightarrow P(2h(n)))  \nonumber \\
&\leq 2(d-1)\mathbb{P}_p(0 \leftrightarrow P(2n)) \nonumber \\
&\leq 2(d-1)\Cr{c: point_to_plane_upper} \exp\left( - \frac{2n}{\xi(p)}\right). \label{eq: confinement_pasta}
\end{align}
If we combine this with \eqref{eq: pasta_macaroni}, we obtain $\Cl[smc]{c: pasta_macaroni_constant} = \Cr{c: pasta_macaroni_constant}(p) > 0$ such that for all $n$, we have
\[
\Cr{c: pasta_macaroni_constant} \exp\left( - \frac{n}{\xi(p)}\right) \leq \mathbb{P}_p(C(0)||_{S(n)} \cap P(0) = \{0\}, C(0)||_{ \widetilde{R}_n} \cap P(n) \neq \emptyset).
\]
Because $(0,\vec{z}) + \widetilde{R}_n \subset R_n'$ for all $\vec{z} \in [-h(n), \dots, h(n)]^{d-1} \cap  \mathbb{Z}^{d-1}$, the above implies
\begin{equation}\label{eq: almost_rice_a_roni}
(2h(n)+1)^{d-1}\Cr{c: pasta_macaroni_constant} \exp\left( - \frac{n}{\xi(p)}\right) \leq \mathbb{E}\sum_{\vec{z} \in [-h(n), h(n)]^{d-1} \cap  \mathbb{Z}^{d-1}} \mathbf{1}_{D_{\vec{z}}},
\end{equation}
where we have written
\[
D_{\vec{z}} = \left\{ C((0,\vec{z}))||_{S(n)} \cap  P(0) = \{(0,\vec{z})\}, C((0,\vec{z}))||_{R_n'} \cap P(n) \neq \emptyset \right\}.
\]

We can estimate the expectation on the right side of \eqref{eq: almost_rice_a_roni} using the BK inequality \cite[Thm.~2.12]{grimmettbook}. Observe that if $D_{\vec{z}_1}\cap \dots \cap  D_{\vec{z}_k}$ occurs for some distinct $\vec{z}_1, \dots, \vec{z}_k \in [-h(n),h(n)]^{d-1} \cap \mathbb{Z}^{d-1}$, then we can find $k$ many edge-disjoint left-right open crossings of $R_n'$. This is because for each $m$, the second condition in the definition of $D_{\vec{z}_m}$, namely that $C((0,\vec{z}_m))||_{R_n'} \cap P(n) \neq \emptyset$, implies that there is a left-right open crossing $\gamma_m$ of $R_n'$ using only vertices of $C((0,\vec{z}_m))||_{R_n'}$. If, say, $\gamma_m$ and $\gamma_\ell$ were to share an edge, then the sets $C((0,\vec{z}_m))||_{R_n'}$ and $C((0,\vec{z}_\ell))||_{R_n'}$ would share a vertex and would therefore be equal, contradicting the first condition in the definition of $D_{\vec{z}_m}$. So if we let $X_n$ be the maximal number of edge-disjoint left-right open crossings of $R_n'$, we have 
\[
\sum_{\vec{z} \in [-h(n),h(n)]^{d-1} \cap \mathbb{Z}^{d-1}} \mathbf{1}_{D_{\vec{z}}} \leq X_n
\]
and so, from \eqref{eq: almost_rice_a_roni}, we obtain
\begin{equation}\label{eq: crab_cake}
(2h(n)+1)^{d-1}\Cr{c: pasta_macaroni_constant} \exp\left( - \frac{n}{\xi(p)}\right) \leq \mathbb{E}X_n = \sum_{k=1}^\infty \mathbb{P}_p(X_n \geq k).
\end{equation}
The BK inequality implies that $\mathbb{P}_p(X_n \geq k) \leq \mathbb{P}_p(X_n \geq 1)^k$, and so the above display implies
\begin{equation}\label{eq: BK_derivation}
(2h(n)+1)^{d-1}\Cr{c: pasta_macaroni_constant} \exp\left( - \frac{n}{\xi(p)}\right) \leq \frac{\mathbb{P}_p(X_n \geq 1)}{1-  \mathbb{P}_p(X_n \geq 1)}.
\end{equation}
This implies the lower bound and, along with the upper bound in \eqref{eq: crossing_upper_bound}, completes the proof.
\end{proof}

We record here that as a consequence of the proof of Prop.~\ref{prop: rectangle_crossing_probability}, we have the following bounds on the expectation of $X_n$, where
\[
X_n = \text{maximal number of edge-disjoint left-right open crossings of }R_n.
\]
\begin{cor}\label{cor: X_n_bounds_subcritical}
Let $p \in (0,p_c)$. There exist $\Cl[smc]{c: crossing_number_lower_new} = \Cr{c: crossing_number_lower_new}(p) > 0$ and $\Cl[lgc]{c: crossing_number_upper_new} = \Cr{c: crossing_number_upper_new}(p)>0$ such that for all $n$ and all choices of $h(n) \geq 6n$, we have
\[
\Cr{c: crossing_number_lower_new} h(n)^{d-1} \exp\left( - \frac{n}{\xi(p)}\right) \leq \mathbb{E}_p X_n \leq \Cr{c: crossing_number_upper_new} h(n)^{d-1} \exp\left( - \frac{n}{\xi(p)}\right).
\]
\end{cor}
\begin{proof}
The lower bound appears in \eqref{eq: crab_cake} and the upper bound follows from Prop.~\ref{prop: point_to_plane_probability}. Each vertex in $\{0\} \times [0,h(n)]^{d-1}$ can be in at most $2d-1$ many edge-disjoint left-right open crossings of $R_n$. So
\begin{align*}
\mathbb{E}_p X_n &\leq (2d-1) \mathbb{E}_p \#\{v \in \{0\} \times [0,h(n)]^{d-1} : v \leftrightarrow P(n)\} \\
&= (2d-1) \sum_{v \in \{0\} \times [0,h(n)]^{d-1}} \mathbb{P}_p(v \leftrightarrow P(n)) \\
&\leq (2d-1) (h(n)+1)^{d-1} \Cr{c: point_to_plane_upper} \exp\left( - \frac{n}{\xi(p)}\right).
\end{align*}
\end{proof}

Grimmett's ``sponge dimensions'' result from \cite[Thm.~1]{G81} states that in the case $d=2$, the crossing probability for $R_n$ converges to 0 if $h(n)e^{-n/\xi(p)} \to 0$ and converges to 1 if $(h(n)/n)e^{-n/\xi(p)} \to \infty$. An immediate corollary of Prop.~\ref{prop: rectangle_crossing_probability} is the following improvement.
\begin{cor}\label{cor: sponge}
Let $p \in (0,p_c)$ and $(n_k)$ be an increasing sequence of natural numbers. We have
\[
\mathbb{P}_p(R_{n_k} \text{ has a left-right open crossing}) \to \begin{cases}
0 &\quad \text{if } h(n_k)^{d-1} \exp\left( - \frac{n_k}{\xi(p)}\right) \to 0 \\
1 &\quad \text{if } h(n_k)^{d-1} \exp\left( - \frac{n_k}{\xi(p)}\right) \to \infty.
\end{cases}
\]
\end{cor}

In the intermediate case, if
\[
\liminf_{k \to \infty} h(n_k)^{d-1} \exp\left( - \frac{n_k}{\xi(p)}\right) > 0,
\]
then it is immediate from Prop.~\ref{prop: rectangle_crossing_probability} that
\[
\liminf_{k \to \infty} \mathbb{P}_p(R_{n_k} \text{ has a left-right open crossing}) > 0.
\]
In the case $d=2$, we have the following corresponding upper bound, which is a consequence of Prop.~\ref{prop: rectangle_crossing_probability} and \cite[Sec.~4]{DS88}:
\begin{cor}\label{cor: DS88}
Let $p \in (0,p_c)$ and $(n_k)$ be an increasing sequence of natural numbers. In the case $d=2$, if 
\[
\limsup_{k \to \infty} h(n_k) \exp\left( - \frac{n_k}{\xi(p)}\right) < \infty,
\]
then
\[
\limsup_{k \to \infty} \mathbb{P}_p(R_{n_k} \text{ has a left-right open crossing}) < 1.
\]
\end{cor}
\begin{proof}
Define the random variable $D_n$ by
\[
D_n = \min\{m \in \mathbb{N} : \exists~\text{a left-right open crossing of } ([1,n] \times [0,m]) \cap \mathbb{Z}^2\} - 1.
\]
The result of \cite[Thm.~3]{DS88} (stated for the supercritical contact process, but see the end of \cite[Sec.~4]{DS88} for the extension to supercritical bond percolation), along with planar duality, gives that $D_n/\beta_n$ converges as $n \to \infty$ in distribution to a mean-one exponential random variable, where
\[
\beta_n = \inf\left\{t : \mathbb{P}_p(D_n > t) \leq \frac{1}{e}\right\}.
\]
We have 
\begin{equation*}\label{eq: beta_addition}
\mathbb{P}_p(R_{n_k} \text{ has a left-right open crossing}) = \mathbb{P}_p\left( \frac{D_{n_k+1}}{\beta_{n_k+1}} \leq \frac{h(n_k)-1}{\beta_{n_k+1}}\right),
\end{equation*}
so if we prove that
\begin{equation}\label{eq: reduced_liminf_etc_statement}
U:= \limsup_{k \to \infty} \frac{h(n_k)-1}{\beta_{n_k+1}} < \infty,
\end{equation}
then we would have
\[
\limsup_{k \to \infty} \mathbb{P}_p\left( \frac{D_{n_k+1}}{\beta_{n_k+1}} \leq \frac{h(n_k)-1}{\beta_{n_k+1}}\right) \leq \limsup_{k \to \infty} \mathbb{P}_p\left( \frac{D_{n_k+1}}{\beta_{n_k+1}} \leq 2U\right) = 1-e^{-2U} < 1,
\]
which would complete the proof.

To prove \eqref{eq: reduced_liminf_etc_statement}, it suffices to show existence of $\Cl[smc]{c: liminf_etc_lower_constant}>0$
such that for all large $n$, we have
\begin{equation}\label{eq: quantiles_to_prove}
\Cr{c: liminf_etc_lower_constant} h(n) \leq \beta_{n+1}.
\end{equation}
We estimate using Prop.~\ref{prop: rectangle_crossing_probability}:
\begin{align*}
&\mathbb{P}_p(D_{n+1} > \Cr{c: liminf_etc_lower_constant} h(n)) \\
=~& 1 - \mathbb{P}_p\left( ([1,n+1] \times [0,\lfloor\Cr{c: liminf_etc_lower_constant}h(n)\rfloor]) \cap  \mathbb{Z}^2 \text{ has a left-right open crossing}\right) \\
\geq~& 1 - \Cr{c: rectangle_crossing_upper} (\lfloor\Cr{c: liminf_etc_lower_constant}h(n)\rfloor) \exp\left( - \frac{n}{\xi(p)}\right).
\end{align*}
By assumption, there exists $\Cl[lgc]{c: wandering_constant_part_a} > 0$ such that for all large $n$, we have $h(n)e^{-n/\xi(p)} \leq \Cr{c: wandering_constant_part_a}$. This implies for large $n$ that
\[
\mathbb{P}_p(D_{n+1} > \Cr{c: liminf_etc_lower_constant}h(n)) \geq 1 - \Cr{c: rectangle_crossing_upper} \Cr{c: liminf_etc_lower_constant}\Cr{c: wandering_constant_part_a}  > \frac{1}{e},
\]
so long as we choose $\Cr{c: liminf_etc_lower_constant}$ sufficiently small. By definition, this shows \eqref{eq: quantiles_to_prove} and completes the proof.
\end{proof}

\section{Proof of Thm.~\ref{thm: main_mean}}\label{sec: mean_proof}

The conclusions of Thm.~\ref{thm: main_mean} are not difficult to deduce for general weights with the gap condition so long as we prove them in the Bernoulli case. To see why, we couple together the weights $(\tau_e)$ and $(t_e)$ in the following way. Let $(t_e)_{e \in \mathbb{E}^2}$ be a family of i.i.d.~Bernoulli random variables with 
\[
\mathbb{P}(t_e = 0) = p = 1 - \mathbb{P}(t_e = 1).
\]
Next, let $(\tau_e')_{e \in \mathbb{E}^2}$ be an i.i.d.~family of random variables on $(\Omega, \Sigma, \mathbb{P})$, independent of $(t_e)$, satisfying
\[
\mathbb{P}(\tau_e' \in B) = \frac{1}{1 - p} \int_{B \cap (0,\infty)} ~\text{d}F(x) \text{ for Borel } B \subset \mathbb{R}
\]
for an $F$ satisfying \eqref{eq: gap_var_condition_1} and \eqref{eq: finite_mean_tau}. Last, set $\tau_e = t_e \tau_e'$ for all $e$, so that $(\tau_e)$ is i.i.d.~with common distribution function $F$. In addition to the passage time $T_{a,b}$ defined using the weights $(\tau_e)$, we also have the Bernoulli passage time $T_{a,b}^B$ defined using the weights $(t_e)$.

With the above definition, we claim that 
\begin{equation}\label{eq: comparison_claim}
\delta \mathbb{E}T_{a,b}^B(0,P(n)) \leq \mathbb{E}T_{a,b}(0,P(n)) \leq \mathbb{E}[\tau_e \mid \tau_e > 0] \mathbb{E}T_{a,b}^B(0,P(n)).
\end{equation}
Given these inequalities, the results of Thm.~\ref{thm: main_mean} will hold once we prove them for $T_{a,b}^B$.

To show \eqref{eq: comparison_claim}, note that because $\tau_e \geq \delta t_e$, we have $T_{a,b}(0,P(n)) \geq \delta T_{a,b}^B(0,P(n))$, and this gives the first inequality. For the second, let $\gamma_n$ be an optimal (edge self-avoiding) path for $T_{a,b}^B(0,P(n))$, chosen as the first in some deterministic ordering of all finite paths if there are multiple options. Then by independence, we have
\begin{align*}
\mathbb{E}T_{a,b}(0,P(n)) \leq \sum_{f \in \mathbb{E}_{a,b}} \mathbb{E}[\tau_f \mathbf{1}_{\{f \in \gamma_n\}}] &= \sum_{f \in \mathbb{E}_{a,b}} \mathbb{E}\left[ t_f \mathbf{1}_{\{f \in \gamma_n\}} \mathbb{E}[\tau_f' \mid (t_e)] \right]\\
&= \mathbb{E}[\tau_e \mid \tau_e > 0] \mathbb{E}T_{a,b}^B(0,P(n)),
\end{align*}
and this shows \eqref{eq: comparison_claim}.

We now move to the main part of the proof, proving the asymptotics for $\mathbb{E}T_{a,b}^B(0,P(n))$. The subgraph of $\mathbb{G}_{a,b}$ that is relevant for determining $T_{a,b}^B(0,P(n))$ is $\mathbb{G}_{a,b}(n)$, which has vertex set
\[
\mathbb{V}_{a,b}(n) = \{(x_1,x_2) \in \mathbb{V}_{a,b} : 0 \leq x_1 \leq n\}
\]
and edge set
\[
\mathbb{E}_{a,b}(n) = \{\{x,y\} : x,y \in \mathbb{V}_{a,b}(n), |x-y|=1\}.
\]
The graph $\mathbb{G}_{a,b}(n)$ is bounded on the top by a ``highest path'' $P_{a,b}(n)$ that makes only steps right or up. Specifically, it begins at $(0,0) = (0,\lfloor f(0)\rfloor)$, moves right to $(1,\lfloor f(0)\rfloor)$, up (if necessary) to $(1,\lfloor f(1)\rfloor)$, right to $(2, \lfloor f(1) \rfloor)$, up (if necessary) to $(2,\lfloor f(2) \rfloor)$, and so on, until it reaches its final point at $(n,\lfloor f(n) \rfloor)$.  See Fig.~\ref{fig: fig_1} for an illustration of the path $P_{a,b}(n)$.

\begin{figure}[t]
  \centering
  \includegraphics[width=0.6\linewidth, trim={0cm 0cm 0cm 0cm}, clip]{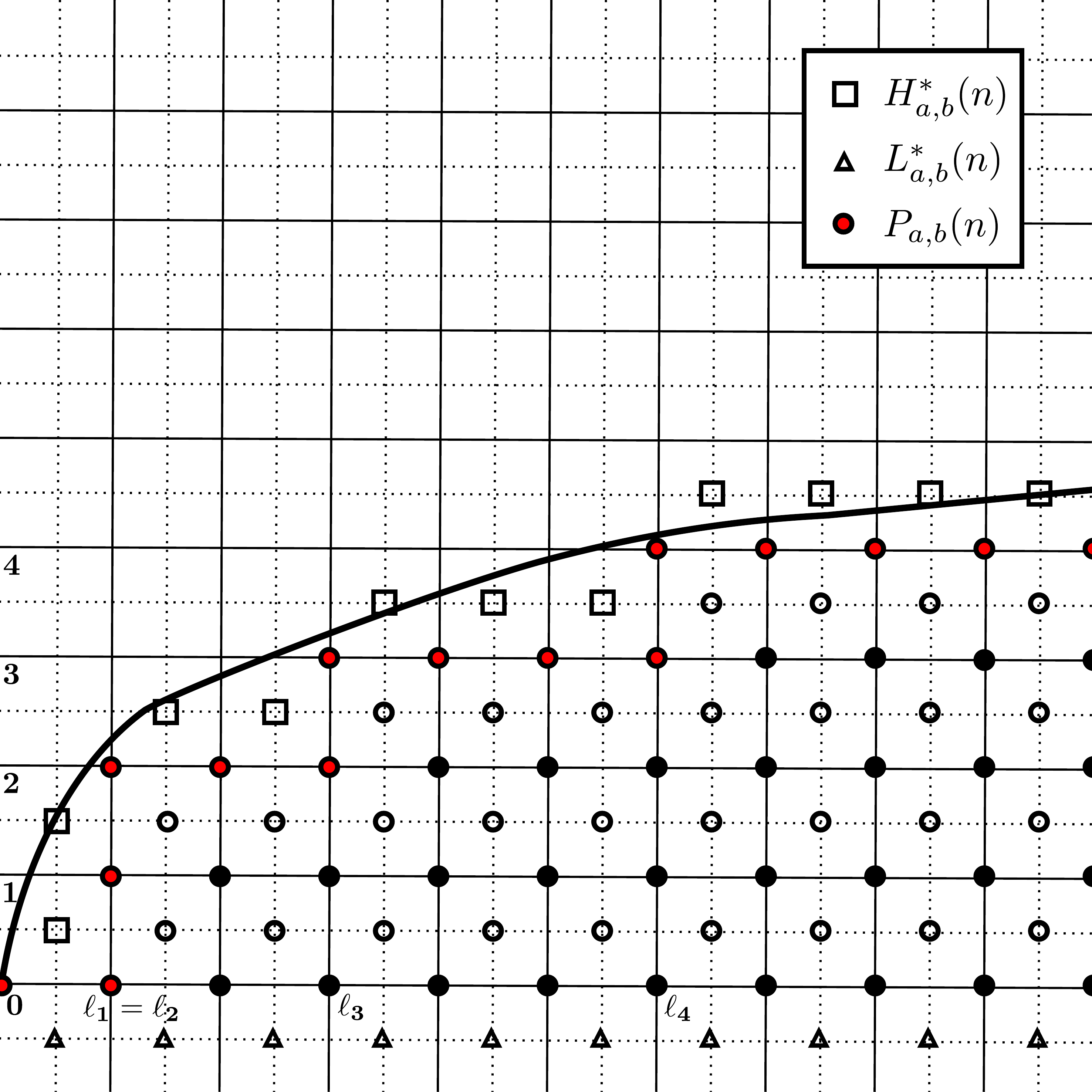}
  \caption{Illustration of $\mathbb{G}_{a,b}(n)$ for $n=10$. The curve is the graph of $f(u) = a \log (1+u) + b \log (1 + \log (1+u))$ and vertices of $\mathbb{G}_{a,b}(n)$ are the solid dots (black and red). The vertices of the dual graph are the non-filled circles, squares, and triangles. Those of the ``highest path'' $P_{a,b}(n)$ are displayed in red, the ``top side'' of $\mathbb{G}_{a,b}^\ast(n)$ is $H_{a,b}^\ast(n)$, whose vertices are displayed as squares, and the ``bottom side'' is $L_{a,b}^\ast(n)$, whose vertices are displayed as triangles. The values listed on the $x$-axis are $\ell_j$, the smallest value of $u$ for which $f(u) \geq j$. Because $f$ increases sharply near 0 in the figure, we have $f(0) = 0$ but $f(1) \geq 2$, so $\ell_1=\ell_2=1$.}
  \label{fig: fig_1}
\end{figure}

\medskip
\noindent
{\bf I. Dual representation.} In the first part of the proof, we describe a standard representation of $T_{a,b}^B(0,P(n))$ in terms of separating sets. We will call an edge $e$ closed if $t_e=1$; otherwise, we call it open. A set $U \subset \mathbb{E}_{a,b}(n)$ is said to be closed if all its edges are closed. We say that $U \subset \mathbb{E}_{a,b}(n)$ is a \underline{separating set} if each path in $\mathbb{G}_{a,b}(n)$ from $0$ to any vertex in $P(n)$ contains an edge in $U$. Define $Y_n$ as
\begin{equation}\label{eq: Y_n}
Y_n = \text{maximal number of disjoint closed separating sets}.
\end{equation}
We claim that
\begin{equation}\label{eq: separating_claim}
T_{a,b}^B(0,P(n)) = Y_n.
\end{equation}
(This claim also holds for Bernoulli FPP on general graphs, with ``separating set'' defined suitably.)

To prove \eqref{eq: separating_claim}, we start with the inequality $\geq$. If $Y_n=0$ then the inequality is trivial; otherwise, let $E_1, \dots, E_{Y_n}$ be a maximal collection of disjoint closed separating sets. If $\gamma$ is a path from $0$ to a vertex in $P(n)$, we can find distinct edges $e_1 \in E_1, \dots, e_{Y_n} \in E_{Y_n}$ that are in $\gamma$ and are closed. Then $T(\gamma) \geq t_{e_1} + \dots + t_{e_{Y_n}} \geq Y_n$. To prove the inequality $\leq$, we must produce $T_{a,b}^B(0,P(n))$ many disjoint closed separating sets. Again if $T_{a,b}^B(0,P(n))=0$ then the inequality is trivial. Otherwise, for $j = 0 , \dots, T_{a,b}^B(0,P(n)) - 1$, define $B(j) = \{v \in \mathbb{V}_{a,b}(n) : T_{a,b}^B(0,v) \leq j\}$ and $E_j = \{\{v,w\} \in \mathbb{E}_{a,b}(n) : v \in B(j), w \notin B(j)\}$. The $E_j$ are disjoint because if $\{v,w\}$ is an edge in $E_{j}$ with $v \in B(j)$ and $w \notin B(j)$, then we have $j < T_{a,b}^B(0,w) \leq T_{a,b}^B(0,v) + t_{\{v,w\}} \leq j+1$, implying that $w \in B(j+1) \setminus B(j)$. This means that for each $j_1 < j_2$, every edge in $E_{j_2}$ has an endpoint in $B(j_2+1) \setminus B(j_2)$, and this therefore cannot be an endpoint of any  edge in $E_{j_1}$. To show they are separating, fix $j=0, \dots, T_{a,b}^B(0,P(n))-1$ and let $\gamma$ be a path from $0$ to a vertex $x \in P(n)$ in $\mathbb{G}_{a,b}(n)$. We have $0 \in B(j)$ and because $x \in P(n)$, we have $T_{a,b}^B(0,x) \geq T_{a,b}^B(0,P(n)) > j$, so $x \notin B(j)$. Therefore there is an edge $\{v,w\}$ in $\gamma$ such that $v \in B(j)$ and $w \notin B(j)$, and hence $E_j$ is separating. Last, if $\{v,w\} \in E_j$ then we must have $t_{\{v,w\}} = 1$ because if $t_{\{v,w\}} = 0$ we would have $T_{a,b}^B(0,v) = T_{a,b}^B(0,w)$, which is false. Therefore $\{E_j\}$ forms a collection of disjoint closed separating sets, and so $T_{a,b}^B(0,P(n)) \leq Y_n$. This proves \eqref{eq: separating_claim}.

Next, we represent $Y_n$ using the dual lattice. Let $(\mathbb{Z}^2)^\ast = \mathbb{Z}^2 + (1/2,1/2)$ be the set of dual vertices and let $(\mathbb{E}^2)^\ast$ be its set of nearest neighbor edges. Each edge $e \in \mathbb{E}^2$ is bisected by a unique edge $e^\ast \in (\mathbb{E}^2)^\ast$ and we assign this dual edge the same weight as its primal edge by defining $t_{e^\ast} = t_e$. We say that the dual edge $e^\ast$ is open if $t_e=0$ and is closed otherwise. The dual graph for $\mathbb{G}_{a,b}(n)$ is $\mathbb{G}_{a,b}^\ast(n)$, which has edge set 
\[
\mathbb{E}_{a,b}^\ast(n) = \{e^\ast : e \in \mathbb{E}_{a,b}(n)\}
\]
and vertex set 
\[
\mathbb{V}_{a,b}^\ast(n) = \{v : v \text{ is an endpoint of an edge of } \mathbb{E}_{a,b}^\ast(n)\}.
\]
The ``top side'' of $\mathbb{G}^\ast_{a,b}(n)$ is the set of dual vertices that are above the highest path $P_{a,b}(n)$:
\[
H_{a,b}^\ast(n) = \left\{ v \in \mathbb{V}_{a,b}^\ast(n) :
\begin{array}{c}
 v \text{ is an endpoint of }e^\ast \text{ for some edge }e \text{ in } P_{a,b}(n) 
\text{ and } v \text{ lies } \\
\text{above }P_{a,b}(n) \text{ in the strip } \{(x_1,x_2) \in \mathbb{R}^2 : 0 \leq x_1 \leq n\}
\end{array}
\right\}
\]
(these vertices are either the top or left endpoints of edges dual to those in $P_{a,b}(n)$), and the ``bottom side'' is the set $L_{a,b}^\ast(n) = \{(1/2,-1/2), \dots, (n-1/2, -1/2)\}$. See Fig.~\ref{fig: fig_1} for an illustration of these definitions.	

We claim that 
\begin{equation}\label{eq: new_Y_n}
Y_n = 
\begin{array}{c}
\text{maximal number of edge-disjoint closed dual } \\
\text{paths in } \mathbb{G}_{a,b}^\ast(n) \text{ from }H_{a,b}^\ast(n) \text{ to } L_{a,b}^\ast(n).
\end{array}
\end{equation}
To prove this, denote the right side by $Y_n'$. We first show that $Y_n \leq Y_n'$ and observe that it is trivially true if $Y_n=0$. Otherwise, let $E_1, \dots, E_{Y_n}$ be a maximal collection of disjoint closed separating sets. Fix $j = 1, \dots, Y_n$, call each edge in $E_j$ ``vacant'' and each edge in $\mathbb{E}_{a,b}(n) \setminus E_j$ ``occupied.'' A dual edge is called occupied if its corresponding primal edge is occupied; otherwise it is vacant. By planar duality (see \cite[p.~55]{BR06}), either there is a path of occupied edges from $0$ to $P(n)$ in $\mathbb{G}_{a,b}(n)$ or there is a dual path of vacant edges from $H_{a,b}^\ast(n)$ to $L_{a,b}^\ast(n)$ in $\mathbb{G}_{a,b}^\ast(n)$. Because $E_j$ is separating, any path of the first type would have to contain a vacant edge, which is impossible. So a dual path of the second type must exist and this shows that $Y_n \leq Y_n'$. The other inequality $Y_n \geq Y_n'$ also follows similarly. Indeed, it is trivial if $Y_n' =0$; otherwise, if $P_1, \dots, P_{Y_n'}$ is a maximal collection of edge-disjoint paths in $\mathbb{G}_{a,b}^\ast(n)$ from $H^\ast_{a,b}(n)$ to $L_{a,b}^\ast(n)$, we may fix $j$ and call all edges of $P_j$ ``vacant'' with all other edges ``occupied.'' A primal edge is called occupied if its corresponding dual edge is occupied; otherwise it is vacant. Planar duality then implies that there is no occupied path in $\mathbb{G}_{a,b}(n)$ from $0$ to $P(n)$, so any path from $0$ to $P(n)$ must contain a vacant edge. Therefore the edges of $P_j$ form a closed separating set, and $Y_n \geq Y_n'$.

\medskip
\noindent
{\bf II. Upper bounds}. Now that we have given our representation for $T_{a,b}^B(0,P(n))$ in \eqref{eq: separating_claim} and \eqref{eq: new_Y_n}, we move to proving the upper bounds. For each $j = 0, \dots, \lfloor f(n) \rfloor$, write
\[
H_{a,b}^\ast(n;j) = \left\{ (x_1, y_1) \in H_{a,b}^\ast(n) : y_1 = j+\frac{1}{2}\right\}.
\]
If we define
\[
Y_{n,j} = \begin{array}{c}
\text{maximal number of edge-disjoint closed dual} \\
\text{paths in } \mathbb{G}_{a,b}^\ast(n) \text{ from } H_{a,b}^\ast(n;j) \text{ to } L_{a,b}^\ast(n),
\end{array}
\]
then from \eqref{eq: separating_claim} and \eqref{eq: new_Y_n}, we obtain
\begin{equation}\label{eq: reduction_upper_bound}
\mathbb{E}T_{a,b}^B(0,P(n)) \leq \sum_{j=0}^{\lfloor f(n) \rfloor} \mathbb{E} Y_{n,j}.
\end{equation}
To bound $\mathbb{E} Y_{n,j}$, we need a good estimate on $\#H_{a,b}^\ast(n,j)$. For this, we introduce the numbers
\begin{equation}\label{eq: ell_j_def}
\ell_j = \min\{d \in \mathbb{Z} : d \geq 0 \text{ and } f(d) \geq j\} = \lceil f^{-1}(j)\rceil,
\end{equation}
for $j \geq 0$. This $\ell_j$ is the $\mathbf{e}_1$-coordinate of the first vertex of the highest path $P_{a,b}(n)$ with $\mathbf{e}_2$-coordinate at least equal to $j$. We claim that
\begin{equation}\label{eq: counting_claim}
\#H_{a,b}^\ast(n;j) = \begin{cases}
\max\{\ell_{j+1} - \ell_j, 1\} &\quad \text{if } j = 0, \dots, \lfloor f(n) \rfloor - 1\\
n - \ell_{\lfloor f(n) \rfloor} &\quad \text{if } j = \lfloor f(n) \rfloor \\
0 & \quad \text{if } j  > \lfloor f(n) \rfloor.
\end{cases}
\end{equation}
To show this, we can describe the motion of $P_{a,b}(n)$ in terms of the $\ell_j$'s.  Specifically, as we follow the path, we start at $0 = (\ell_0,0)$, move right until $(\ell_1,0)$ and then up to $(\ell_1,1)$. Generally, if we are at $(\ell_j,j)$, we move right (if necessary) to $(\ell_{j+1},j)$, and then up to $(\ell_{j+1},j+1)$. We do this until we reach the point $(\ell_{\lfloor f(n) \rfloor}, \lfloor f(n) \rfloor)$, and then move right to $(n,\lfloor f(n) \rfloor)$. The vertices of $H_{a,b}^\ast(n;j)$ are either (a) upper endpoints of vertical (dual) edges which are dual to horizontal edges with both endpoints in $\{(x_1, j) : x_1 \in \mathbb{R}\}$ or (b) left endpoints of horizontal (dual) edges which are dual to vertical edges connecting $\{(x_1,j) : x_1 \in \mathbb{R}\}$ to $\{(x_1,j+1) : x_1 \in \mathbb{R}\}$. If $j = 0, \dots, \lfloor f(n) \rfloor - 1$, the $\ell_{j+1} - \ell_j$ many endpoints of type (a) are
\[
\left( \ell_j + \frac{1}{2}, j+\frac{1}{2}\right), \dots, \left( \ell_{j+1} - \frac{1}{2}, j + \frac{1}{2}\right)
\]
and if $j = \lfloor f(n) \rfloor$, the $n - \ell_{\lfloor f(n) \rfloor}$ many endpoints are
\[
\left( \ell_{\lfloor f(n) \rfloor} + \frac{1}{2}, \lfloor f(n) \rfloor + \frac{1}{2}\right), \dots, \left( n - \frac{1}{2} , \lfloor f(n) \rfloor + \frac{1}{2}\right).
\]
The endpoint of type (b) is $(\ell_{j+1} - 1/2, j+ 1/2)$ if $j = 0, \dots, \lfloor f(n) \rfloor - 1$ (observe that this point has already been counted in the list of type (a) points when $\ell_{j+1} - \ell_j > 0$) and there are no such endpoints when $j = \lfloor f(n) \rfloor$. If $j > \lfloor f(n) \rfloor$, then there are no endpoints of types (a) or (b). Putting these cases together proves \eqref{eq: counting_claim}.

To estimate $\ell_j$, we note that for all large $j$, we have
\[
\frac{e^{\frac{j}{a}}}{\left( 1 + \frac{j}{a}\right)^{\frac{b}{a}}} - 1 \leq f^{-1}(j) \leq \frac{e^{\frac{j}{a}}}{\left( 1 + \frac{j}{a}\right)^{\frac{b}{a}}} \cdot \left( \frac{1}{1 - \frac{b}{a+j} \log \left( 1 + \frac{j}{a}\right)} \right)^{\frac{b}{a}} - 1.
\]
This can be verified by applying $f$ to all three terms. Therefore we have
\begin{equation}\label{eq: ell_j_asymptotic}
\ell_j \sim \frac{e^{\frac{j}{a}}}{\left( \frac{j}{a}\right)^{\frac{b}{a}}} \text{ and } (\ell_{j+1} - \ell_j) \sim (e^{\frac{1}{a}}-1) \frac{e^{\frac{j}{a}}}{\left( \frac{j}{a}\right)^{\frac{b}{a}}},
\end{equation}
where we use the notation $f_1(j) \sim f_2(j)$ to mean that $f_1(j) \lesssim f_2(j) \lesssim f_1(j)$.

\medskip
\noindent
{\bf II(a). Case $p>1/2$.} First we suppose that $p>1/2$. Observe that a vertex in the set $H_{a,b}^\ast(n;j)$ can be in at most two edge-disjoint dual closed paths in $\mathbb{G}_{a,b}^\ast(n)$ and so we have
\[
Y_{n,j} \leq 2\#\{v \in H_{a,b}^\ast(n;j) : v \text{ is connected to } L_{a,b}^\ast(n) \text{ in } \mathbb{G}_{a,b}^\ast(n) \text{ by a closed dual path}\}.
\]
Using this in \eqref{eq: reduction_upper_bound} with a union bound, we obtain
\begin{align*}
\mathbb{E}T_{a,b}^B(0,P(n)) &\leq 2 \sum_{j=0}^{\lfloor f(n) \rfloor} \sum_{v \in H_{a,b}^\ast(n;j)} \mathbb{P} \left(
\begin{array}{c}
v \text{ is connected to }L_{a,b}^\ast(n) \text{ in } \\
\mathbb{G}_{a,b}^\ast(n) \text{ by a closed dual path}
\end{array}
\right)
\\
&\leq 2\sum_{j=0}^{\lfloor f(n) \rfloor} \#H_{a,b}^\ast(n;j) \mathbb{P}_{1-p} ( 0 \leftrightarrow P(j+1)).
\end{align*}
Equation \eqref{eq: counting_claim} and Prop.~\ref{prop: point_to_plane_probability} then imply
\begin{align}
\mathbb{E}T_{a,b}^B(0,P(n)) &\leq 2 \Cr{c: point_to_plane_upper} \sum_{j=0}^{\lfloor f(n) \rfloor - 1}\max\{\ell_{j+1}-\ell_j,1\} \exp\left( - \frac{j+1}{\xi(1-p)}\right) \nonumber \\
&+ 2 \Cr{c: point_to_plane_upper}(n-\ell_{\lfloor f(n) \rfloor}) \exp\left( - \frac{\lfloor f(n) \rfloor + 1}{\xi(1-p)}\right) \nonumber.
\end{align}
Due to \eqref{eq: ell_j_asymptotic}, there exists $J_{a,b}$ such that if $j \geq J_{a,b}$ then $\ell_{j+1}-\ell_j \geq 1$. This implies that
\begin{align}
\mathbb{E} T_{a,b}^B(0,P(n)) &\leq 2 \Cr{c: point_to_plane_upper} J_{a,b} \nonumber \\
&+ 2 \Cr{c: point_to_plane_upper} \sum_{j=0}^{\lfloor f(n) \rfloor - 1} (\ell_{j+1}-\ell_j) \exp\left( - \frac{j+1}{\xi(1-p)}\right) \label{eq: split_1a} \\
&+ 2 \Cr{c: point_to_plane_upper} (n - \ell_{\lfloor f(n) \rfloor}) \exp\left( - \frac{\lfloor f(n)\rfloor + 1}{\xi(1-p)}\right). \label{eq: split_2a}
\end{align}
For the first term, we use \eqref{eq: ell_j_asymptotic} to say that
\begin{equation}\label{eq: diverging_condition}
\text{if } \sum_j \frac{\exp\left( j \left( \frac{1}{a} - \frac{1}{\xi(1-p)}\right) \right)}{ j^{\frac{b}{a}} }\text{ diverges,}
\end{equation}
then
\begin{equation}\label{eq: before_the_cases}
\eqref{eq: split_1a} \sim 2 \Cr{c: point_to_plane_upper} \left( e^{\frac{1}{a}} - 1 \right) a^{\frac{b}{a}} e^{-\frac{1}{\xi(1-p)}} \sum_{j=1}^{\lfloor f(n) \rfloor -1} \frac{\exp\left( j \left( \frac{1}{a} - \frac{1}{\xi(1-p)}\right)\right)}{j^{\frac{b}{a}}}.
\end{equation}
For the second term, if we write
\[
\Delta_{a,b}^{(2)}(n) = \frac{n - \ell_{\lfloor f(n) \rfloor}}{n} \exp\left(  \frac{f(n) - \lfloor f(n) \rfloor -1}{\xi(1-p)} \right),
\]
we have
\begin{align}
\eqref{eq: split_2a} &= 2 \Cr{c: point_to_plane_upper}\Delta_{a,b}^{(2)}(n) \cdot n \exp\left( - \frac{f(n)}{\xi(1-p)}\right) \nonumber \\
&= 2 \Cr{c: point_to_plane_upper}\Delta_{a,b}^{(2)}(n) \cdot n (1+n)^{-\frac{a}{\xi(1-p)}} (1+\log (1+n))^{- \frac{b}{\xi(1-p)}} \nonumber \\
&\sim 2 \Cr{c: point_to_plane_upper}\Delta_{a,b}^{(2)}(n) \frac{n^{1- \frac{a}{\xi(1-p)}} }{(\log n)^{\frac{b}{\xi(1-p)}}}. \label{eq: split_2aa}
\end{align}

If $a< \xi(1-p)$, then \eqref{eq: diverging_condition} holds and we use that for any $\alpha > 0$ and $\beta \geq 0$, one has 
\begin{equation}\label{eq: exponential_sum}
\sum_{\ell=1}^j \frac{e^{\alpha \ell}}{\ell^{\beta}} \sim \frac{1}{e^{\alpha} - 1} \cdot \frac{e^{\alpha(j+1)}}{j^{\beta}}.
\end{equation}
This implies from \eqref{eq: before_the_cases} that if we write
\[
\Delta_{a,b}^{(1)}(n) = \exp\left( - (f(n) - \lfloor f(n) \rfloor) \left( \frac{1}{a} - \frac{1}{\xi(1-p)}\right)\right),
\]
then we have
\begin{align}
\eqref{eq: split_1a} &\sim   2 \Cr{c: point_to_plane_upper} a^{\frac{b}{a}} \cdot \frac{e^{\frac{1}{a}} - 1}{e^{ \frac{1}{a}} - e^{\frac{1}{\xi(1-p)}}} \cdot  \frac{\exp\left( \lfloor f(n) \rfloor \left( \frac{1}{a} - \frac{1}{\xi(1-p)}\right) \right)}{\left( \lfloor f(n)\rfloor - 1 \right)^{\frac{b}{a}} } \nonumber \\
&\sim 2 \Cr{c: point_to_plane_upper} \Delta_{a,b}^{(1)}(n) \cdot \frac{e^{\frac{1}{a}} - 1}{e^{\frac{1}{a}} - e^{\frac{1}{\xi(1-p)}}} \cdot \frac{\exp\left( f(n)\left( \frac{1}{a} - \frac{1}{\xi(1-p)}\right) \right)}{(\log n)^{\frac{b}{a}}} \nonumber \\
&= 2 \Cr{c: point_to_plane_upper}\Delta_{a,b}^{(1)}(n) \cdot \frac{e^{\frac{1}{a}} - 1}{e^{\frac{1}{a}} - e^{\frac{1}{\xi(1-p)}}} \cdot \frac{(1+n)^{1-\frac{a}{\xi(1-p)}} (1+ \log (1+n))^{b\left( \frac{1}{a} - \frac{1}{\xi(1-p)}\right)}}{(\log n)^{\frac{b}{a}}} \nonumber \\
&\sim 2 \Cr{c: point_to_plane_upper}\Delta_{a,b}^{(1)}(n) \left( 1+  \frac{e^{\frac{1}{\xi(1-p)}} - 1}{e^{\frac{1}{a}} - e^{\frac{1}{\xi(1-p)}}}\right)  \frac{n^{1-\frac{a}{\xi(1-p)}}}{(\log n)^{\frac{b}{\xi(1-p)}}}. \label{eq: my_second_part}
\end{align}

In summary, when $p>1/2$ and $a < \xi(1-p)$, we have from \eqref{eq: split_2aa} and \eqref{eq: my_second_part} that 
\begin{equation}\label{eq: to_use_later_on}
\eqref{eq: split_1a} + \eqref{eq: split_2a} 
\sim 2 \Cr{c: point_to_plane_upper} \left( \Delta_{a,b}^{(1)}(n) \left( 1+  \frac{e^{\frac{1}{\xi(1-p)}} - 1}{e^{\frac{1}{a}} - e^{\frac{1}{\xi(1-p)}}}\right) +  \Delta_{a,b}^{(2)}(n)\right) \frac{n^{1-\frac{a}{\xi(1-p)}}}{(\log n)^{\frac{b}{\xi(1-p)}}}. 
\end{equation}
Using $|\Delta_{a,b}^{(1)}(n)| \leq 1$, $|\Delta_{a,b}^{(2)}(n)| \leq 1$, and
\[
\frac{e^{\frac{1}{\xi(1-p)}}-1}{e^{\frac{1}{a}}-e^{\frac{1}{\xi(1-p)}}} = \frac{1 - e^{-\frac{1}{\xi(1-p)}}}{e^{\frac{1}{a}-\frac{1}{\xi(1-p)}}-1} \leq \frac{1}{e^{\frac{1}{a}-\frac{1}{\xi(1-p)}}-1} \leq \frac{1}{\frac{1}{a}- \frac{1}{\xi(1-p)}} \leq \frac{\xi(1-p)^2}{\xi(1-p)-a},
\]
we find
\[
\mathbb{E}T_{a,b}^B(0,P(n)) \lesssim 2 \Cr{c: point_to_plane_upper} \left( 2 + \frac{\xi(1-p)^2}{\xi(1-p)-a}\right) \frac{n^{1-\frac{a}{\xi(1-p)}}}{(\log n)^{\frac{b}{\xi(1-p)}}}.
\]
The prefactor is bounded above by $C_p / (\xi(1-p)-a)$ for a $p$-dependent constant $C_p$, so this implies the upper bound in the theorem in this case.

If, instead, we have $b<a = \xi(1-p)$, then \eqref{eq: diverging_condition} still holds and we can use $\sum_{\ell=1}^j \ell^{-\alpha} \sim j^{1-\alpha}/(1-\alpha)$ for $\alpha <1$ in \eqref{eq: before_the_cases} to produce
\begin{align*}
\eqref{eq: split_1a} \sim 2 \Cr{c: point_to_plane_upper} \left( e^{\frac{1}{a}} - 1\right) a^{\frac{b}{a}} e^{-\frac{1}{\xi(1-p)}} \sum_{j=1}^{\lfloor f(n) \rfloor - 1} \frac{1}{j^{\frac{b}{a}}} &\sim 2 \Cr{c: point_to_plane_upper} \frac{1-e^{-\frac{1}{\xi(1-p)}}}{1 - \frac{b}{\xi(1-p)}} a^{\frac{b}{a}} \cdot \left( \lfloor f(n) \rfloor - 1 \right)^{1 - \frac{b}{a}} \\
&\sim 2 \Cr{c: point_to_plane_upper} \xi(1-p) \frac{1 - e^{-\frac{1}{\xi(1-p)}}}{1 - \frac{b}{\xi(1-p)}} \left( \log n \right)^{1 - \frac{b}{\xi(1-p)}}.
\end{align*}
Also, from \eqref{eq: split_2aa}, we have $\eqref{eq: split_2a} \to 0$ as $n \to \infty$. Therefore if $p>1/2$ and $b < a = \xi(1-p)$, then
\begin{equation}\label{eq: to_use_later_on_2}
\eqref{eq: split_1a} + \eqref{eq: split_2a}  \sim 2 \Cr{c: point_to_plane_upper}\xi(1-p) \frac{1 - e^{-\frac{1}{\xi(1-p)}}}{1 - \frac{b}{\xi(1-p)}} \left( \log n \right)^{1 - \frac{b}{\xi(1-p)}}.
\end{equation}
This implies the upper bound in the theorem in this case.

Next, if $a=b=\xi(1-p)$, then \eqref{eq: diverging_condition} holds and \eqref{eq: before_the_cases} implies that
\[
\eqref{eq: split_1a} \sim 2 \Cr{c: point_to_plane_upper} a \left( 1 - e^{-\frac{1}{\xi(1-p)}}\right) \sum_{j=1}^{\lfloor f(n) \rfloor - 1} \frac{1}{j} \sim 2\Cr{c: point_to_plane_upper} a \left( 1 - e^{-\frac{1}{\xi(1-p)}}\right) \log \log n.
\]
From \eqref{eq: split_2aa}, we have $\eqref{eq: split_2a} \to 0$ as $n \to \infty$. Therefore if $p>1/2$ and $a=b=\xi(1-p)$, then we have
\begin{equation}\label{eq: to_use_later_on_3}
\eqref{eq: split_1a} + \eqref{eq: split_2a} \sim 2\Cr{c: point_to_plane_upper} a \left( 1 - e^{-\frac{1}{\xi(1-p)}}\right) \log \log n
\end{equation}
and this is the upper bound in the theorem in this case.

Last, if $a > \xi(1-p)$, or if $a = \xi(1-p)$ but $b > \xi(1-p)$, then \eqref{eq: diverging_condition} fails. Because $\ell_{j+1}-\ell_j \sim (e^{1/a}-1)e^{j/a}/(j/a)^{b/a}$, this immediately implies that \eqref{eq: split_1a} converges. Also $\eqref{eq: split_2a} \to 0$ as $n \to \infty$, so $\sup_n \mathbb{E}T_{a,b}^B(0,P(n))<\infty$, as claimed.

\medskip
\noindent
{\bf II(b). Case $p=1/2$.} Next we suppose that $p=1/2$ and return to the estimate \eqref{eq: reduction_upper_bound} for $\mathbb{E}T_{a,b}^B(0,P(n))$ in terms of the variables $Y_{n,j}$. Unlike in the case $p>1/2$, we cannot just bound $Y_{n,j}$ by the number of vertices in $H^\ast_{a,b}(n;j)$ that are connected to $L_{a,b}^\ast(n)$. When $p=1/2$, many such vertices will correspond to the same closed dual path because large critical clusters touch $H^\ast_{a,b}(n;j)$ many times, but also have cut-edges. To deal with this, we use some standard critical percolation constructions.

To stay close to the results in Sec.~\ref{sec: subcritical}, we will temporarily switch to open paths rather than closed dual ones (this does not matter because $p=1/2$). For integers $n$ and $h(n)$, let 
\begin{equation}\label{eq: U_h_n_def}
U(h(n),n) = \begin{array}{c}
\text{maximal number of edge-disjoint open paths } \\
\text{from  }\left( \{0\} \times [0,h(n)]\right) \cap \mathbb{Z}^2 \text{ to } P(n).
\end{array}
\end{equation}
These open paths are allowed to use any edges in $\mathbb{E}^2$. We will show that there exists $\Cl[lgc]{c: critical_upper} > 0$ such that for all $n$ and all choices of $h(n) \geq n$, we have
\begin{equation}\label{eq: critical_upper}
\mathbb{E} U(h(n),n) \leq \Cr{c: critical_upper} \frac{h(n)}{n}.
\end{equation}
To prove this, we first consider the case $h(n) = n$. Any open path from $(\{0\} \times [0,n])\cap \mathbb{Z}^2$ to $P(n)$ must contain a left-right crossing of $B_1$ or $B_2$, or contain a top-down crossing of $B_3$ or $B_4$, where the $B_i$ are defined as follows:
\[
B_1 = \left( [0,n] \times [-n,2n]\right) \cap \mathbb{Z}^2, ~B_2 = B_1 - n \mathbf{e}_1,
\]
and
\[
B_3 = \left( [-2n,n] \times [n,2n] \right) \cap \mathbb{Z}^2, ~B_4 = B_3-2n\mathbf{e}_2.
\]
By symmetry, we obtain
\[
\mathbb{E} U(n,n) \leq 4 \mathbb{E} \left( \text{maximal number of edge-disjoint left-right open crossings of }B_1\right).
\]
By the RSW theorem (see \cite[Eq.~(11.74),(11.76)]{grimmettbook}), there exists $\Cl[smc]{c: RSW1}>0$ such that for all $n$, we have 
\[
\mathbb{P}(B_1\text{ has a left-right open crossing}) \leq 1-\Cr{c: RSW1}.
\] 
As in \eqref{eq: BK_derivation}, we apply to BK inequality to obtain
\[
\mathbb{E} \left( \text{maximal number of edge-disjoint left-right open crossings of }B_1\right) \leq \frac{1-\Cr{c: RSW1}}{\Cr{c: RSW1}},
\]
and so $\mathbb{E} U(n,n) \leq 4 (1-\Cr{c: RSW1})/\Cr{c: RSW1}$. This shows \eqref{eq: critical_upper} in the case $h(n)=n$. If $h(n) > n$, we simply write $U_n(j)$ for the maximal number of edge-disjoint open paths from $(\{0\} \times [(j-1)n, jn]) \cap \mathbb{Z}^2$ to $P(n)$ and observe that by symmetry, we have
\[
\mathbb{E} U(h(n),n) \leq \sum_{j=1}^{\left\lfloor \frac{h(n)}{n} \right\rfloor +1} \mathbb{E} U_n(j) \leq \left(  \frac{h(n)}{n} +1 \right) \mathbb{E} U(n,n) \leq \frac{8(1-\Cr{c: RSW1})}{\Cr{c: RSW1}} \cdot \frac{h(n)}{n},
\]
proving \eqref{eq: critical_upper}.

We now return to \eqref{eq: reduction_upper_bound} and to using closed dual paths.  Because $p=1/2$, we have 
\[
\mathbb{E}Y_{n,j} \leq \mathbb{E} U(\#H_{a,b}^\ast(n;j), j+1),
\]
and so using \eqref{eq: counting_claim}, we find
\[
\mathbb{E}T_{a,b}^B(0,P(n)) \leq \sum_{j=0}^{\lfloor f(n) \rfloor - 1} \mathbb{E} U(\max\{\ell_{j+1}-\ell_j, 1\}, j+1) + \mathbb{E} U(n - \ell_{\lfloor f(n) \rfloor}, \lfloor f(n) \rfloor + 1).
\]
The second summand is bounded by $\mathbb{E}U(n,\lfloor f(n) \rfloor + 1)$. Because $\ell_{j+1} - \ell_j \geq j+1$ for $j$ sufficiently large (depending on $a$ and $b$) and $n \geq \lfloor f(n) \rfloor + 1$ for $n$ sufficiently large, \eqref{eq: critical_upper} implies that for some $K_{a,b}>0$, we have
\[
\mathbb{E}T_{a,b}^B(0,P(n)) \leq K_{a,b} + \sum_{j=0}^{\lfloor f(n) \rfloor -1} \Cr{c: critical_upper} \frac{\ell_{j+1}-\ell_j}{j+1} + \Cr{c: critical_upper} \frac{n}{\lfloor f(n) \rfloor + 1},
\]
if $n$ is large enough. Using \eqref{eq: ell_j_asymptotic} and the fact that $\sum_j e^{j/a}/(j/a)^{b/a+1}$ diverges for all $a > 0$ and $b \geq 0$, we get
\begin{align*}
\mathbb{E}T_{a,b}^B(0,P(n)) &\lesssim \Cr{c: critical_upper} \left( e^{\frac{1}{a}} -1\right) \sum_{j=1}^{\lfloor f(n) \rfloor -1} \frac{e^{\frac{j}{a}}}{\left( \frac{j}{a}\right)^{\frac{b}{a}} (j+1)} + \Cr{c: critical_upper} \frac{n}{a \log n} \\
&\lesssim \Cr{c: critical_upper}\left( e^{\frac{1}{a}} -1 \right) a^{\frac{b}{a}} \sum_{j=1}^{\lfloor f(n) \rfloor -1} \frac{e^{\frac{j}{a}}}{j^{\frac{b}{a}+1}} + \Cr{c: critical_upper} \frac{n}{a \log n}
\end{align*}
Again \eqref{eq: exponential_sum} implies
\begin{align*}
\mathbb{E}T_{a,b}^B(0,P(n)) &\lesssim \Cr{c: critical_upper} a^{\frac{b}{a}} \cdot \frac{e^{\frac{\lfloor f(n) \rfloor }{a}}}{\left( \lfloor f(n) \rfloor -1 \right)^{\frac{b}{a}+1}} + \Cr{c: critical_upper} \frac{n}{a \log n} \\
&\lesssim \Cr{c: critical_upper}  \frac{(1+n) (1+ \log (1+n))^{\frac{b}{a}}}{a\left( \log n \right)^{\frac{b}{a}+1}} + \Cr{c: critical_upper} \frac{n}{a \log n} \\
&\lesssim 2\Cr{c: critical_upper}\frac{n}{a \log n}.
\end{align*}
This completes the proof of the upper bound in the case $p=1/2$.

\medskip
\noindent
{\bf III. Lower bounds.} For lower bounds, we look at top-down dual closed crossings of rectangles.

\medskip
\noindent
{\bf III(a). Case $p>1/2$.} In this case, we use rectangles that are contained in $\mathbb{G}_{a,b}^\ast(n)$, defining $R_{a,b}^\ast(n;j)$ to have top side equal to $H_{a,b}^\ast(n;j)$: for any $j$ such that $\ell_{j+1}>\ell_j$, we set
\[
R_{a,b}^\ast(n;j) = \left( \left[ \ell_j + \frac{1}{2}, \ell_{j+1}-\frac{1}{2}\right] \times \left[ - \frac{1}{2}, j + \frac{1}{2}\right] \right) \cap (\mathbb{Z}^2)^\ast.
\]
There exists $J_{a,b}$ such that 
\begin{equation}\label{eq: j_a_b_def}
\text{for } j \geq J_{a,b} \text{ we have }\ell_{j+1} - \ell_j - 1 \geq 6(j+1),
\end{equation}
and so, in particular, for $j = J_{a,b}, \dots, \lfloor f(n) \rfloor - 1$, the set $R_{a,b}^\ast(n;j)$ is defined. The final rectangle is
\[
R_{a,b}^\ast(n;\lfloor f(n) \rfloor) = \left( \left[\ell_{\lfloor f(n) \rfloor} + \frac{1}{2}, n- \frac{1}{2}\right] \times \left[ - \frac{1}{2}, \lfloor f(n) \rfloor + \frac{1}{2} \right] \right) \cap (\mathbb{Z}^2)^\ast.
\]
For $j = J_{a,b}, \dots, \lfloor f(n) \rfloor$, let
\[
X_{n,j} = \text{maximal number of edge-disjoint closed dual top-down crossings of }R_{a,b}^\ast(n;j).
\]
Because any top-down crossing of $R_{a,b}^\ast(n;j)$ starts in $H_{a,b}^\ast(n)$ and ends in $L_{a,b}^\ast(n)$, and the $R_{a,b}^\ast(n;j)$ are disjoint, we have from \eqref{eq: separating_claim} and \eqref{eq: new_Y_n} that
\[
\mathbb{E}T_{a,b}^B(0,P(n)) \geq \sum_{j=J_{a,b}}^{\lfloor f(n) \rfloor} \mathbb{E} X_{n,j}.
\]
Because of \eqref{eq: j_a_b_def}, we may apply Cor.~\ref{cor: X_n_bounds_subcritical} to the rectangle $R_{a,b}^\ast(n;j)$ to get
\[
\mathbb{E} X_{n,j} \geq \Cr{c: crossing_number_lower_new}(\ell_{j+1}-\ell_j - 1) \exp\left( - \frac{j+1}{\xi(1-p)}\right) \text{ for } j = J_{a,b}, \dots, \lfloor f(n) \rfloor -1,
\]
and for $j = \lfloor f(n) \rfloor$, if $n - \ell_{\lfloor f(n) \rfloor} -1 \geq 6(\lfloor f(n) \rfloor + 1)$, we have
\[
\mathbb{E} X_{n, \lfloor f(n) \rfloor} \geq \Cr{c: crossing_number_lower_new} (n - \ell_{\lfloor f(n) \rfloor}) \exp\left( - \frac{\lfloor f(n) \rfloor + 1}{\xi(1-p)}\right).
\]
On the other hand, if $n - \ell_{\lfloor f(n) \rfloor} \leq 6(\lfloor f(n) \rfloor + 1)$, the right side of the above display is bounded above by $6\Cr{c: crossing_number_lower_new} (\lfloor f(n) \rfloor + 1) e^{-(\lfloor f(n) \rfloor + 1)/\xi(1-p)}$, which tends to zero as $n \to \infty$. Putting these together, we obtain a constant $K_{a,b,p}$ depending on $a,b$, and $p$ such that
\begin{align}
\mathbb{E}T_{a,b}^B(0,P(n)) &\geq \sum_{j=J_{a,b}}^{\lfloor f(n)\rfloor -1} \Cr{c: crossing_number_lower_new}(\ell_{j+1}-\ell_j-1) \exp\left( -\frac{j+1}{\xi(1-p)}\right) \nonumber \\
&+ \Cr{c: crossing_number_lower_new} (n - \ell_{\lfloor f(n) \rfloor}) \exp\left( - \frac{\lfloor f(n) \rfloor + 1}{\xi(1-p)}\right) \nonumber \\&- 6\Cr{c: crossing_number_lower_new} (\lfloor f(n) \rfloor + 1) e^{-(\lfloor f(n) \rfloor + 1)/\xi(1-p)} \nonumber \\
&\geq -K_{a,b,p}  \nonumber \\
&+ \Cr{c: crossing_number_lower_new}\sum_{j=0}^{\lfloor f(n) \rfloor-1} (\ell_{j+1}-\ell_j) \exp\left( - \frac{j+1}{\xi(1-p)}\right) \nonumber \\
&+ \Cr{c: crossing_number_lower_new} (n - \ell_{\lfloor f(n) \rfloor}) \exp\left( - \frac{\lfloor f(n) \rfloor + 1}{\xi(1-p)}\right) \nonumber \\
&= -K_{a,b,p} + \frac{\Cr{c: crossing_number_lower_new}}{2 \Cr{c: point_to_plane_upper}} \cdot \left( \eqref{eq: split_1a} + \eqref{eq: split_2a} \right) \label{eq: to_go_from_here}.
\end{align}

If $p > 1/2$ and $b < a = \xi(1-p)$, then the lower bound in the theorem follows directly from \eqref{eq: to_use_later_on_2} and \eqref{eq: to_go_from_here}. Similarly, if  $p>1/2$ and $a = b = \xi(1-p)$, then the lower bound in the theorem follows directly from \eqref{eq: to_use_later_on_3} and \eqref{eq: to_go_from_here}. We may therefore suppose that $a< \xi(1-p)$ and may use \eqref{eq: to_use_later_on} in \eqref{eq: to_go_from_here} to find
\[
\mathbb{E}T_{a,b}^B(0,P(n)) \gtrsim \Cr{c: crossing_number_lower_new} \left( \Delta_{a,b}^{(1)}(n) \left( 1+  \frac{e^{\frac{1}{\xi(1-p)}} - 1}{e^{\frac{1}{a}} - e^{\frac{1}{\xi(1-p)}}}\right) +  \Delta_{a,b}^{(2)}(n)\right) \frac{n^{1-\frac{a}{\xi(1-p)}}}{(\log n)^{\frac{b}{\xi(1-p)}}}. 
\]
To finish the lower bound in the case $p>1/2$ and $a< \xi(1-p)$, we may show that for some constant $\Cl[smc]{c: my_favorite_constant} = \Cr{c: my_favorite_constant}(p) > 0$, we have
\begin{equation}\label{eq: gettin_there}
\Delta_{a,b}^{(1)}(n)\left( 1 + \frac{e^{\frac{1}{\xi(1-p)}}-1}{e^{\frac{1}{a}} - e^{\frac{1}{\xi(1-p)}}}\right) + \Delta_{a,b}^{(2)}(n) \gtrsim \frac{\Cr{c: my_favorite_constant}}{\xi(1-p)-a}.
\end{equation}

If $n \geq 2 \ell_{\lfloor f(n) \rfloor}$, then
\[
\Delta_{a,b}^{(2)}(n) \geq \frac{1}{2} \exp\left( \frac{f(n) - \lfloor f(n) \rfloor -1}{\xi(1-p)}\right) \geq \frac{1}{2} \exp\left( - \frac{1}{\xi(1-p)}\right).
\]
On the other hand, if $n \leq 2 \ell_{\lfloor f(n) \rfloor}$, then we can show that $\Delta_{a,b}^{(1)}(n) \gtrsim 1/2$ by using the following facts. First, we have 
\begin{equation}\label{eq: cadbury_chocolate_1}
\lfloor f(n) \rfloor - f(\ell_{\lfloor f(n) \rfloor}) \to 0 \text{ as } n \to \infty.
\end{equation}
This follows because $f(\ell_j) \geq j$ for $j \geq 0$ and $f(\ell_j-1) < j$ when $\ell_j \geq 1$, and so for large $n$, we have
\[
\lfloor f(n) \rfloor \leq f(\ell_{\lfloor f(n) \rfloor}) < f(\ell_{\lfloor f(n) \rfloor}) - f(\ell_{\lfloor f(n) \rfloor}-1) + \lfloor f(n) \rfloor.
\]
Because of our choice of $f$, we have $f(\ell_{\lfloor f(n)\rfloor}) - f(\ell_{\lfloor f(n) \rfloor}-1) \to 0$ as $n \to \infty$, and so this gives \eqref{eq: cadbury_chocolate_1}.

Next, because $n \lesssim e^{\frac{1}{a}} \ell_{\lfloor f(n) \rfloor}$, we have $\ell_{\lfloor f(n) \rfloor} \leq n \leq C_{a,b} \ell_{\lfloor f(n) \rfloor}$ for some constant $C_{a,b}$ depending on $a$ and $b$, and so
\begin{align*}
&| f(n) - f(\ell_{\lfloor f(n) \rfloor}) - \left( a \log n - a \log \ell_{\lfloor f(n) \rfloor}\right) | \\
\leq~& a \bigg|  \log \left( \frac{1+n}{n} \right) - \log \left( \frac{1+ \ell_{\lfloor f(n) \rfloor}}{\ell_{\lfloor f(n) \rfloor}}\right)\bigg| + b\bigg| \log\left( \frac{1 + \log (1+n)}{1 + \log (1+ \ell_{\lfloor f(n) \rfloor})}\right) \bigg| \\
\leq~& a \bigg|  \log \left( \frac{1+n}{n} \right) - \log \left( \frac{1+ \ell_{\lfloor f(n) \rfloor}}{\ell_{\lfloor f(n) \rfloor}}\right)\bigg| + b\bigg| \log\left( \frac{1 + \log (1+C_{a,b}\ell_{\lfloor f(n) \rfloor})}{1 + \log (1+ \ell_{\lfloor f(n) \rfloor})}\right) \bigg| \\
\to~& 0 \text{ as } n \to \infty.
\end{align*}
This, along with \eqref{eq: cadbury_chocolate_1}, implies that
\begin{align*}
\Delta_{a,b}^{(1)}(n) &\sim \exp\left( - (f(n) - f(\ell_{\lfloor f(n) \rfloor}))\left( \frac{1}{a} - \frac{1}{\xi(1-p)}\right) \right) \\
&\sim \exp\left( - (a \log n - a \log \ell_{\lfloor f(n) \rfloor}) \left( \frac{1}{a} - \frac{1}{\xi(1-p)}\right)\right) \\
&= \exp\left( - \left( \log \frac{n}{\ell_{\lfloor f(n) \rfloor}} \right) \left( 1 - \frac{a}{\xi(1-p)}\right) \right) \\
&\geq \frac{1}{2}.
\end{align*}
Therefore $\Delta_{a,b}^{(1)}(n) + \Delta_{a,b}^{(2)}(n) \gtrsim e^{-1/\xi(1-p)}/2$ and so
\begin{align*}
&\Delta_{a,b}^{(1)}(n)\left( 1 + \frac{e^{\frac{1}{\xi(1-p)}}-1}{e^{\frac{1}{a}} - e^{\frac{1}{\xi(1-p)}}}\right) + \Delta_{a,b}^{(2)}(n) \\
\gtrsim~& \Delta_{a,b}^{(1)}(n)\left( \frac{e^{\frac{1}{\xi(1-p)}}-1}{e^{\frac{1}{a}} - e^{\frac{1}{\xi(1-p)}}}\right) + \frac{1}{2}\exp\left( - \frac{1}{\xi(1-p)}\right) \\
\geq~& \exp\left( - \left( \frac{1}{a} - \frac{1}{\xi(1-p)}\right) \right)  \left( \frac{e^{\frac{1}{\xi(1-p)}}-1}{e^{\frac{1}{a}} - e^{\frac{1}{\xi(1-p)}}}\right) + \frac{1}{2}\exp\left( - \frac{1}{\xi(1-p)}\right).
\end{align*}
If $a < \xi(1-p)/2$, then we drop the first term to see that this expression is $\geq e^{-1/\xi(1-p)}/2 \geq c_p  / (\xi(1-p) - a)$, with $c_p = e^{-1/\xi(1-p)} \xi(1-p)/4$. On the other hand, if $a \geq \xi(1-p)/2$, then we drop the second term, giving a lower bound for the expression of 
\[
\exp\left( - \frac{2}{\xi(1-p)}\right) \left( \exp\left( \frac{1}{\xi(1-p)} \right) - 1\right) \cdot \frac{1}{e^{\frac{1}{a} - \frac{1}{\xi(1-p)}}-1}
\]
which is bounded below by $c_p/(\xi(1-p)-a)$ for a (more complicated) constant $c_p$ depending only on $p$, assuming that $a$ remains bigger than $\xi(1-p)/2$. These two cases imply \eqref{eq: gettin_there} and complete the proof of the lower bound when $p>1/2$.

\medskip
\noindent
{\bf III(b). Case $p = 1/2$.} Last, we prove the lower bound in the case $p=1/2$ and show, like in \eqref{eq: critical_upper}, that if 
\[
V(h(n),n) = \begin{array}{c}
\text{maximal number of edge-disjoint left-right open} \\
\text{crossings of } ([0,n] \times [0,h(n)]) \cap \mathbb{Z}^2,
\end{array}
\]
then for some constant $\Cl[smc]{c: critical_lower_constant}>0$, we have
\begin{equation}\label{eq: critical_lower}
\mathbb{E} V(h(n),n) \geq \Cr{c: critical_lower_constant} \frac{h(n)}{n}
\end{equation}
for all choices of $n$ and $h(n) \geq n$. To do this, we make a packing by squares. Set $B_1 = ([0,n] \times [0,n])\cap \mathbb{Z}^2$, $B_2 = ([0,n] \times [n+1, 2n+1]) \cap \mathbb{Z}^2$, and in general, for $j \geq 1$, set
\[
B_j = ( [0,n] \times [(j-1)(n+1), j(n+1)-1]) \cap \mathbb{Z}^2.
\]
The final square $B_j$ which fits inside of our rectangle has index $j_0 =  \lfloor (h(n)+1)/(n+1) \rfloor$. Using \cite[p.~316]{grimmettbook}, we have $\mathbb{P}(B_1 \text{ has a left-right open crossing}) \geq 1/2$, and so \eqref{eq: critical_lower} follows:
\[
\mathbb{E} V(h(n),n) \geq \sum_{j=1}^{j_0} \mathbb{P}( B_j \text{ has a left-right open crossing}) \geq \frac{j_0}{2} \geq \frac{1}{8} \cdot \frac{h(n)}{n}.
\]

To use \eqref{eq: critical_lower}, we observe that if we define the dual rectangle 
\[
R_{a,b}^\ast(n) = \left( \left[ \frac{1}{2}, n - \frac{1}{2} \right] \times \left[ - \frac{1}{2}, \lfloor f(n) \rfloor + \frac{1}{2}\right] \right) \cap (\mathbb{Z}^2)^\ast,
\]
then every top-down dual crossing of $R_{a,b}^\ast(n)$ contains a dual path from $H_{a,b}^\ast(n)$ to $L_{a,b}^\ast(n)$. Therefore $Y_n$, defined back in \eqref{eq: new_Y_n}, satisfies
\[
Y_n \geq \text{maximal number of edge-disjoint top-down closed dual crossings of } R_{a,b}^\ast(n).
\]
Because $n-1 \geq \lfloor f(n) \rfloor + 1$ for all large $n$, we can use \eqref{eq: critical_lower} to conclude that
\[
\mathbb{E} T_{a,b}^B(0,P(n)) \geq \mathbb{E} V(n-1, \lfloor f(n) \rfloor + 1) \gtrsim \Cr{c: critical_lower_constant} \frac{n - 1}{\lfloor f(n) \rfloor +1} \gtrsim \Cr{c: critical_lower_constant} \frac{n}{a \log n},
\]
completing the proof.

\section{Martingale construction for general wedges}\label{sec: general_f_proof}

The proofs of both Thm.~\ref{thm: updated_KZ_variance} and Thm.~\ref{thm: updated_KZ_CLT} use the martingale construction of Kesten-Zhang \cite{KZ97}, replacing their open circuits in annuli with open top-down crossings of the sets $R_i$. It will be more convenient for us to have the $R_i$ separated by blank space, so we will work only with the even indices, which correspond to the sets $R_{2i}$ for $i \geq 0$. To avoid endless confusion, we set $r_i' = r_{2i}$ and $R_i' = R_{2i}$ for $i \geq 0$. From here on, we use the coupling of $(\tau_e)$ and $(t_e)$ defined at the beginning of Sec.~\ref{sec: mean_proof}, and we assume A1-A3 for some $\Cr{c: r_i_assumption_1} > 0$, $\Cr{c: r_i_assumption_2} \in [1,\infty)$, and $\Cr{c: r_i_assumption_3} : [0,\infty) \to (0,\infty)$. Therefore we may choose an integer $i^\ast \geq 1$ such that
\begin{equation}\label{eq: new_A1}
\mathbb{P}(\exists~\text{top-down open crossing of }R_i) \geq \frac{\Cr{c: r_i_assumption_1}}{2} \text{ for } i \geq 2i^\ast,
\end{equation}
and
\begin{equation}\label{eq: new_A2_pro}
\mathbb{E}T_f^B(0,P(r_{i+1}) - \mathbb{E}T_f^B(0,P(r_i)) \leq 2\Cr{c: r_i_assumption_2} \text{ for } i \geq i^\ast.
\end{equation}

The next few sections develop the martingale construction, and the proofs of Thm.~\ref{thm: updated_KZ_variance} and Thm.~\ref{thm: updated_KZ_CLT} will be given in Sec.~\ref{sec: KZ_proofs}.

\subsection{Line-to-line bound}\label{sec: line_to_line}
Before we give the main definitions, we prove a tail bound on the passage time between lines $P(r_i')$ and $P(r_j')$ for $i < j$ that will be applied throughout. We first observe that \eqref{eq: new_A2_pro} implies that for integers $i,j$ with $i^\ast \leq i < j$, we have
\begin{align}
\mathbb{E}T_f^B(P(r_i'),P(r_j')) &\leq \mathbb{E}\left[ T_f^B(0,P(r_j')) - T_f^B(0,P(r_i'))\right]  \nonumber \\
&=\sum_{\ell=2i}^{2j-1} \left[  \mathbb{E}T_f^B(0,P(r_{\ell+1})) - \mathbb{E}T_f^B(0,P(r_\ell))  \right] \nonumber \\
&\leq 4\Cr{c: r_i_assumption_2}(j-i). \label{eq: new_A2}
\end{align}
(In fact, \eqref{eq: new_A2} is a weaker version of A2 that suffices for our purposes.) The following corresponds to \cite[Eq.~(2.50)]{KZ97}.

\begin{lem}\label{lem: plane_to_plane_estimate}
There exist $\Cl[lgc]{c: haydns_constant_1} \geq 1$ and $\Cl[smc]{c: haydns_constant_2}>0$ depending only on $\Cr{c: r_i_assumption_2}$ such that for all integers $i,j$ with $i^\ast \leq i < j$ and $y \geq 0$, we have
\[
\mathbb{P}(T_f(P(r_i'),P(r_j')) \geq y) \leq \Cr{c: haydns_constant_1} \exp\left( - \Cr{c: haydns_constant_2} \frac{y}{j-i}\right) + \frac{\Cr{c: haydns_constant_1}}{y^\frac{\eta}{2}}.
\]
\end{lem}
\begin{proof}
Let $z \geq 0$ be an integer and write
\begin{align}
\mathbb{P}(T_f(P(r_i'),P(r_j')) \geq y) &\leq \mathbb{P}(T_f^B(P(r_i'),P(r_j')) \geq z) \label{eq: plane_to_plane_term_1}\\
&+ \mathbb{P}(T_f^B(P(r_i'),P(r_j')) < z, T_f(P(r_i'),P(r_j')) \geq y). \label{eq: plane_to_plane_term_2}
\end{align}
For \eqref{eq: plane_to_plane_term_1}, we first use \eqref{eq: new_A2} and Markov's inequality to obtain
\begin{equation}\label{eq: easy_markov}
\mathbb{P}(T_f^B(P(r_i'),P(r_j')) \geq 8\Cr{c: r_i_assumption_2}(j-i)) \leq  \frac{1}{2}.
\end{equation}
Next observe that, similarly to \eqref{eq: Y_n}, we have
\begin{equation}\label{eq: more_separating_equivalence}
T_f^B(P(r_i'),P(r_j')) = \begin{array}{c}
\text{maximal number of disjoint closed sets in }\mathbb{E}_f 
\text{ separating } \\
\{r_i'\} \times \{0, \dots, \lfloor f(r_i') \rfloor\} \text{ from } \{r_j'\} \times \{0, \dots, \lfloor f(r_j')\rfloor\}.
\end{array}
\end{equation}
As in \eqref{eq: separating_claim}, these separating sets are defined as edge sets in $\mathbb{E}_f$ such that each path in $\mathbb{G}_f$ starting in $\{r_i'\} \times \{0, \dots, \lfloor f(r_i')\rfloor\}$ and ending in $\{r_j'\} \times \{0, \dots, \lfloor f(r_j') \rfloor\}$ must contain an edge in the separating set. If $m \geq 0$ is an integer and $T_f^B(P(r_i'),P(r_j')) \geq m\lceil 8 \Cr{c: r_i_assumption_2}(j-i)\rceil$, then we can find $m$ disjoint edge sets, each of which contains at least $\lceil 8\Cr{c: r_i_assumption_2}(j-i)\rceil$ many disjoint separating sets. By the BK inequality and \eqref{eq: easy_markov}, 
\begin{align*}
\mathbb{P}(T_f^B(P(r_i'),P(r_j')) \geq m \lceil 8\Cr{c: r_i_assumption_2}(j-i)\rceil ) &\leq \mathbb{P}(\exists~\lceil 8\Cr{c: r_i_assumption_2}(j-i)\rceil \text{ many disjoint separating sets})^m\\
&\leq 2^{-m},
\end{align*}
and so
\begin{align}
\text{RHS of } \eqref{eq: plane_to_plane_term_1} &= \mathbb{P}\left( T_f^B(P(r_i'),P(r_j')) \geq \frac{z}{\lceil 8 \Cr{c: r_i_assumption_2}(j-i)\rceil} \lceil 8 \Cr{c: r_i_assumption_2}(j-i)\rceil \right) \nonumber \\
&\leq 2^{-\lfloor \frac{z}{\lceil 8 \Cr{c: r_i_assumption_2}(j-i) \rceil} \rfloor} \nonumber \\
&\leq \Cl[lgc]{c: trainor_constant_1} \exp\left( - \Cl[smc]{c: trainor_constant_2} \frac{z}{j-i} \right), \label{eq: z_bound}
\end{align}
for some $\Cr{c: trainor_constant_1}, \Cr{c: trainor_constant_2} > 0$ depending only on $\Cr{c: r_i_assumption_2}$.

As for \eqref{eq: plane_to_plane_term_2}, we write
\begin{equation}\label{eq: durretts_delight}
\eqref{eq: plane_to_plane_term_2} = \sum_{k=0}^{z-1}\mathbb{E}\left[ \mathbb{P}\left( T_f(P(r_i'),P(r_j')) \geq y \mid (t_e)\right) \mathbf{1}_{\{T_f^B(P(r_i'),P(r_j')) = k\}}\right],
\end{equation}
where $\mathbb{P}(\cdot \mid (t_e))$ represents the conditional probability given the Bernoulli weights $(t_e)$. On the event $\{T_f^B(P(r_i'),P(r_j')) = k\}$, we can choose an optimal path for $T_f^B(P(r_i'),P(r_j'))$ (in some $(t_e)$-measurable way, for example by picking the first such optimal path in a deterministic ordering of all finite paths in $\mathbb{Z}^2$) and write $e_1, \dots, e_k$ for the $k$ many edges listed in order along the path that satisfy $t_{e_i} = 1$ for all $i$. If $T_f(P(r_i'),P(r_j')) \geq y$, then it must be that $\tau_{e_1}' + \dots + \tau_{e_k}' = \tau_{e_1} + \dots + \tau_{e_k} \geq y$ (recall the coupling $\tau_e = t_e \tau_e'$ from below \eqref{eq: gap_var_condition_2}), so
\begin{align*}
\eqref{eq: plane_to_plane_term_2} &\leq \sum_{k=0}^{z-1} \sum_{f_1, \dots, f_k} \mathbb{E}\left[ \mathbb{P}\left( \tau_{f_1}' + \dots + \tau_{f_k}' \geq y \mid (t_e) \right) \mathbf{1}_{\{T_f^B(P(r_i'),P(r_j')) = k\}} \mathbf{1}_{\{e_i = f_i ~\forall i\}} \right] \\
&\leq \sum_{k=0}^{z-1} \sum_{f_1, \dots, f_k} \mathbb{P}(\tau_1' + \dots + \tau_k' \geq y) \mathbb{P}(T_f^B(P(r_i'),P(r_j')) = k, e_i = f_i ~\forall i) \\
&\leq \mathbb{P}(\tau_1' + \dots + \tau_z' \geq y),
\end{align*}
where the second inequality uses independence of $(\tau_e')$ from $(t_e)$, and the variables $\tau_1', \dots, \tau_z'$ are i.i.d., with the same common distribution as $(\tau_e')$. So long as 
\begin{equation}\label{eq: y_assumption_for_now}
y \geq 2z \mathbb{E}\tau_1',
\end{equation}
we have
\begin{align*}
\mathbb{P}(\tau_1' + \dots + \tau_z' \geq y) &\leq \mathbb{P}\left(|\tau_1' - \mathbb{E}\tau_1' + \dots + \tau_z' - \mathbb{E}\tau_z'| \geq \frac{y}{2}\right) \\
&\leq 2^\eta y^{-\eta} \mathbb{E}|\tau_1' + \dots + \tau_z' - \mathbb{E}[\tau_1' + \dots + \tau_z'] |^\eta.
\end{align*}
By applying \cite[Thm.~I.5.1(iii)]{G88} and noting that $\mathbb{E}(\tau_k')^\eta < \infty$ by \eqref{eq: gap_var_condition_2} (and that $\eta > 4 \geq 2$), we obtain an upper bound of
\[
\eqref{eq: plane_to_plane_term_2} \leq 2^\eta y^{-\eta} \Cl[lgc]{c: guts_constant} z^{\frac{\eta}{2}} \mathbb{E}|\tau_1'|^\eta \leq 2^\eta y^{-\eta} \Cr{c: guts_constant} \left( \frac{y}{2 \mathbb{E}\tau_1'} \right)^{\frac{\eta}{2}} \mathbb{E}|\tau_1'|^\eta,
\]
for some universal $\Cr{c: guts_constant}$ depending only on $\eta$. In the second inequality, we used \eqref{eq: y_assumption_for_now}. In total, this gives $\Cl[lgc]{c: guts_constant_2}$ depending only on $\eta$ and the distribution of $\tau_e$ such that for all $i,j$ with $0 \leq i < j$ and all $y$ satisfying \eqref{eq: y_assumption_for_now}, we have 
\[
\eqref{eq: plane_to_plane_term_2} \leq \frac{\Cr{c: guts_constant_2}}{y^\frac{\eta}{2}}.
\]
If we then set $z = \lfloor y/(2\mathbb{E}\tau_1') \rfloor$, then the previous display holds, and we can combine it with \eqref{eq: z_bound} to deduce that
\[
\mathbb{P}(T_f(P(r_i'),P(r_j'))(\omega) \geq y) \leq \Cr{c: trainor_constant_1} \exp\left( - \Cr{c: trainor_constant_2} \frac{\lfloor \frac{y}{2 \mathbb{E}\tau_1'} \rfloor}{j-i} \right) + \frac{\Cr{c: guts_constant_2}}{y^\frac{\eta}{2}},
\]
which implies the statement of the lemma.
\end{proof}

The proof of Lem.~\ref{lem: plane_to_plane_estimate} requires $\int x^2 ~\text{d}F(x) < \infty$ to apply \cite[Thm.~I.5.1(iii)]{G88}, and this holds because of \eqref{eq: gap_var_condition_2}. We remark that the bound in the lemma can be improved under the assumption that 
\[
f(m) \geq 3 \text{ for }m=r_i', \dots, r_j'\text{ and } \int x^\eta ~\text{d}F(x)<\infty \text{ for some }\eta> 2/3.
\]
In this case, one can show that given $s\in [2,3\eta)$, there exist $\Cr{c: haydns_constant_1} \geq 1$ and $\Cr{c: haydns_constant_2}>0$ depending only on $\Cr{c: r_i_assumption_2}$ such that for all integers $i,j$ with $i^\ast \leq i < j$ and $y \geq 0$, we have
\begin{equation}\label{eq: plane_to_plane_improved}
\mathbb{P}(T_f(P(r_i'),P(r_j')) \geq y) \leq \Cr{c: haydns_constant_1} \exp\left( - \Cr{c: haydns_constant_2} \frac{y}{j-i}\right) + \frac{\Cr{c: haydns_constant_1}}{y^\frac{s}{2}}.
\end{equation}
In the rest of this subsection, we outline the modifications needed to obtain this stronger inequality. This result is not used in the paper except to justify Rem.~\ref{rem: moment_condition} about the optimal moment condition for Thm.~\ref{thm: main_variance}. The proof is the same as that of Lem.~\ref{lem: plane_to_plane_estimate} until we reach \eqref{eq: durretts_delight}, and then differs by providing a better upper bound for the conditional probability appearing in the sum of that equation. 

To do this, for any edge $e \in \mathbb{E}_f$, we define two edge sets $S_k(e)$ for $k=1,2$ that surround $e$ as follows. First taking $e = \{(0,0),(1,0)\}$, we define $S_1(e)$ as the set of edges with both endpoints in the (topological) boundary of $[0,1]\times[-1,1]$, and $S_2(e)$ as the set of edges with both endpoints in the (topological) boundary of the rectangle $[-1,2]\times[-2,2]$. If $e \in \mathbb{E}_f$ is another horizontal edge, we define $S_k(e)$ by translating the definition of $S_k(\{(0,0),(1,0)\})$ in such a way that $(0,0)$ is mapped to the left endpoint of $e$. In case $e$ is a vertical edge, we rotate these definitions by $\pi/2$.

For $e \in \mathbb{E}_f$, let $m_e^{(0)} = \tau_e', m_e^{(k)} = \sum_{f \in S_k(e)} \tau_f'$, for $k=1,2$, and $m_e = \min_{k=0,1,2} m_e^{(k)}$. 
We claim that
\begin{equation}\label{eq: many_choices_claim}
T_f(P(r_i'),P(r_j'))  \leq \left( \sum_{\ell=1}^k m_{e_\ell} \right) \text{ on the event } \{T_f^B(P(r_i'),P(r_j')) = k\},
\end{equation}
where, like after \eqref{eq: durretts_delight}, we have chosen an optimal path for $T_f^B(P(r_i'),P(r_j'))$ in some $(t_e)$-measurable way and listed $e_1, \dots, e_k$ as the $k$ many edges in order along along the path that satisfy $t_{e_i}=1$ for all $i$. The strategy to show \eqref{eq: many_choices_claim} is to choose $i_1, \dots, i_k \in \{0,1,2\}$ and construct a path $\Gamma_0$ in $\mathbb{G}_f$ that connects $P(r_i')$ to $P(r_j')$ but satisfies 
\begin{equation}\label{eq: many_choices_claim_2}
T_f(\Gamma_0) \leq \sum_{\ell=1}^k m_{e_\ell}^{(i_\ell)}.
\end{equation}
The path $\Gamma_0$ is constructed by following $\Gamma$ but, because the sets $S_k(e)$ surround $e$ for $k=1,2$, taking a detour through $S_{i_\ell}(e_\ell)$ for $\ell=1, \dots, k$ instead of taking the edge $e_\ell$ whenever $i_\ell \neq 0$. This construction appears to require several cases because $f$ is general, but it is always possible assuming that $f(m) \geq 3$ for all $m=r_i', \dots, r_j'$. (If $f(m) \leq 2$ then the set $S_e(k)$ may separate $P(r_i')$ from $P(r_j')$ and intersections of $\Gamma$ with the set of endpoints of edges in $S_e(k)$ may not be connectable through $S_e(k)$ using only edges of $\mathbb{G}_f$.) Because \eqref{eq: many_choices_claim_2} holds for all choices of $i_1, \dots, i_k \in \{0,1,2\}$, this implies \eqref{eq: many_choices_claim}.

Given \eqref{eq: many_choices_claim}, we can continue from \eqref{eq: durretts_delight} by writing
\begin{align*}
\eqref{eq: plane_to_plane_term_2} &\leq \sum_{k=0}^{z-1} \sum_{f_1, \dots, f_k} \mathbb{E}\left[ \mathbb{P}\left( m_{f_1} + \dots + m_{f_k} \geq y  \mid (t_e) \right) \mathbf{1}_{\{T_f^B(P(r_i'),P(r_j')) = k\}} \mathbf{1}_{\{e_i = f_i ~\forall i\}}\right].
\end{align*}
We observe that $\mathbb{E}\left( m_e^{(k)}\right)^{\eta} < \infty$ for $k=0,1,2$ and furthermore that the $m_e^{(k)}$ are independent as $k$ varies and $e$ is fixed. Therefore for $s<3\eta$, we can apply Markov's inequality and independence to obtain for some $\mathsf{C}>0$
\begin{align*}
\mathbb{E}m_e^s = \int_0^\infty s y^{s-1} \mathbb{P}(m_e \geq y)~\text{d}y &= \int_0^\infty s y^{s-1} \prod_{k=0}^3 \mathbb{P}\left( m_e^{(k)} \geq y\right)~\text{d}y \\
&\leq \mathsf{C} \int_0^\infty sy^{s-1} \cdot \frac{\text{d}y}{y^{3\eta}} < \infty.
\end{align*}
The $m_e$ are not independent, but they are only finitely dependent, so if $s \in [2,3\eta)$, we can still apply a version of \cite[Thm.~I.5.1(iii)]{G88} to obtain for $k \leq z$ and some constants $\mathsf{C}$
\begin{align*}
\mathbb{P}(m_{f_1} + \dots + m_{f_k} \geq y \mid (t_e)) &= \mathbb{P}(m_{f_1} + \dots + m_{f_k} \geq y) \\
&\leq 2^s y^{-s} \mathbb{E}|m_{f_1} + \dots + m_{f_k} - \mathbb{E}[m_{f_1} + \dots + m_{f_k}]|^s \\
&\leq 2^s y^{-s} \mathsf{C} z^{\frac{s}{2}} \mathbb{E}m_e^s \\
&\leq \frac{\mathsf{C}}{y^{\frac{s}{2}}},
\end{align*}
assuming that $y \geq 2z\mathbb{E}m_e$ (compare to \eqref{eq: y_assumption_for_now}). We can then follow to the end of the proof of Lem.~\ref{lem: plane_to_plane_estimate} to obtain \eqref{eq: plane_to_plane_improved}.

\subsection{Martingale definitions}
For $i \geq 0$, define
\[
m(i) = \min\{j \geq i : R_{j}' \text{ has a top-down open crossing}\}.
\]
We observe that if $R_{i}'$ has a top-down open crossing, it also has a top-down open crossing that is vertex self-avoiding, and when viewed as a plane curve, any such crossing is simple. To every simple curve $\Gamma$ connecting $\{(x_1,x_2) \in \mathbb{R}^2 : x_2 = f(x_1)\}$ to $\{(x_1,x_2) \in \mathbb{R}^2 : x_1 \geq 0, x_2 = 0\}$ within the region $\{(x_1,x_2) \in \mathbb{R}^2 : 0 \leq x_2 \leq f(x_1)\}$ we can associate an interior
\[
\text{int }\Gamma = \text{ bounded component of }\{(x_1,x_2) \in \mathbb{R}^2 : 0 \leq x_2 \leq f(x_1)\} \setminus \Gamma
\]
and an exterior
\[
\text{ext }\Gamma = \text{ unbounded component of }\{(x_1,x_2) \in \mathbb{R}^2 : 0 \leq x_2 \leq f(x_1)\} \setminus \Gamma.
\]
We say that a path $\Gamma$ is the \underline{leftmost top-down open crossing} of $R_i'$ if it is a vertex self-avoiding top-down open crossing of $R_i'$ and, when viewed as a plane curve, $\text{int }\Gamma$ is minimal among all vertex self-avoiding top-down open crossings of $R_i'$.

\begin{figure}[t]
  \centering
  \includegraphics[width=0.8\linewidth, trim={0cm 0cm 0cm 0cm}, clip]{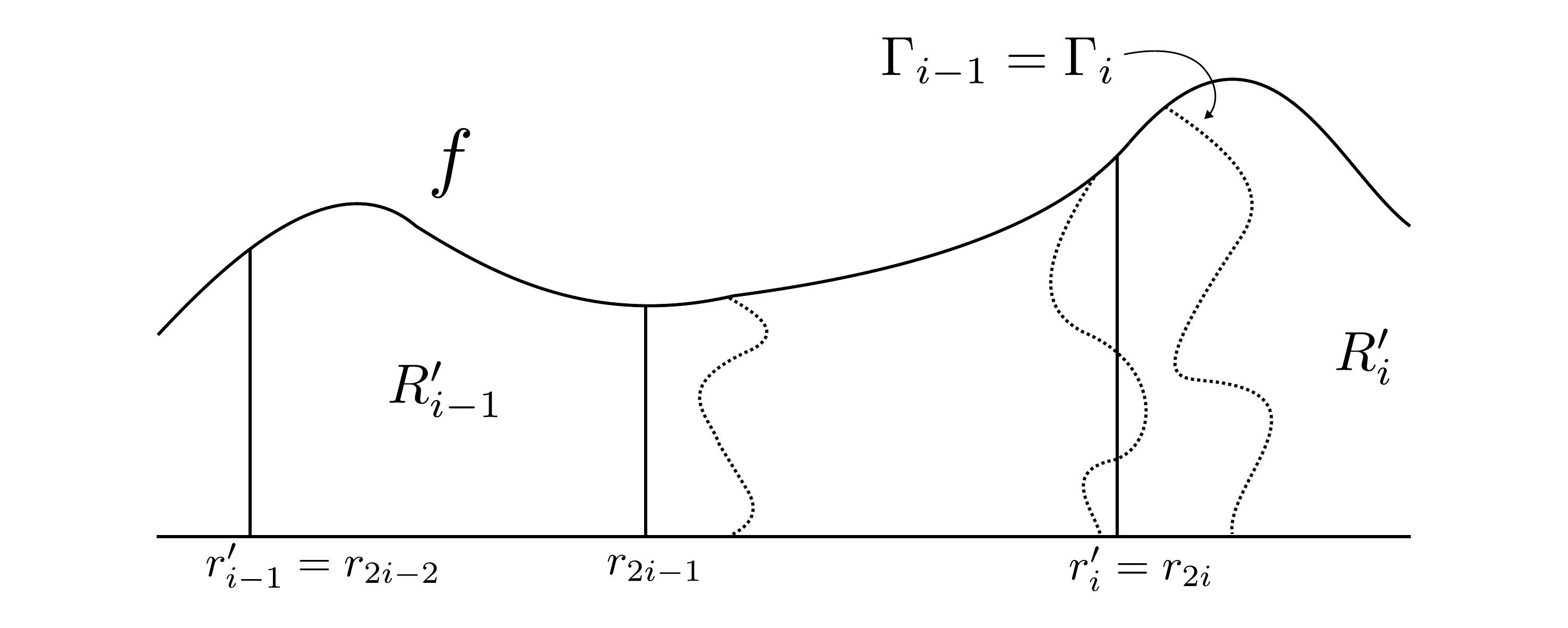}
  \caption{Illustration of the definitions involved in the martingale construction. In the figure, the top solid curve is the graph of $f$ and the region below the curve and between $r_{i-1}'=r_{2i-2}$ and $r_{2i-1}$ is $R_{i-1}'$. The dotted curves connecting the top solid curve to the bottom are open paths. There is no top-down open crossing of $R_{i-1}'$, so $m(i-1) > i-1$. The region to the right of $r_i'$ below the graph of $f$ is $R_i'$, and it has a top-down open crossing, so $m(i-1) = m(i)=i$ and the leftmost such crossing is labeled as $\Gamma_{i-1}$, which is also equal to $\Gamma_i$. The middle two open paths connecting the top solid curve to the bottom are not contained in any $R_j'$, so they do not affect the values of $m(i-1),m(i), \Gamma_{i-1},\Gamma_i$.}
  \label{fig: fig_2}
\end{figure}

By assumption A1, we have $i \leq m(i) < \infty$ almost surely for each $i \geq 0$. We can therefore define 
\[
\Gamma_i = \text{ leftmost top-down open crossing of } R_{m(i)}' \text{ for } i \geq 0.
\]
(See Fig.~\ref{fig: fig_2} for an illustration of $\Gamma_i$.) Next we define the sigma-algebra $\Sigma_i$ that contains the information about (a) what $\Gamma_i$ is and (b) what the weights $(\tau_e)$ are on and to the left of $\Gamma_i$. For any vertex self-avoiding top-down crossing $\Gamma$ of $R_j'$ for some $j \geq i$, any finite subset $E$ of $\mathbb{E}_f$ all whose endpoints are in $\Gamma \cup \text{int } \Gamma$, and any Borel set $B \subset \mathbb{R}^{\# E}$, define the event $\Xi_i(\Gamma, E, B)$ as
\[
\Xi_i(\Gamma, E, B) = \{\Gamma_i = \Gamma\} \cap \{ (t_e)_{e \in E} \in B\}.
\]
Last, we set $\Sigma_{-1} = \{ \emptyset, \Omega\}$, and
\[
\Sigma_i = \sigma\left(\{\Xi_i(\Gamma, E, B)\}_{\Gamma, E, B}\right) \text{ for } i \geq 0,
\]
as $\Gamma, E, B$ range over the items satisfying the conditions above. Regarding these definitions, one can check that
\[
\Sigma_i \subset \Sigma_{i+1} \text{ for } i \geq -1 
\]
and
\[
T_f(0,\Gamma_i) \text{ is } \Sigma_i \text{-measurable for } i \geq 0,
\]
so long as we set $T_f(0,\Gamma_{i}) = \infty$ on the probability zero set $\{m(i) = \infty\}$ on which $\Gamma_i$ is not defined.

Regarding these definitions, we have the following preliminary bounds. First, there exists $\Cl[smc]{c: circuit_appearance_constant} > 0$ depending only on $\Cr{c: r_i_assumption_1}$ such that
\begin{equation}\label{eq: 2.28}
\mathbb{P}(m(i) \geq i + t) \leq e^{-\Cr{c: circuit_appearance_constant}t} \text{ for all integers } i \geq i^\ast \text{ and } t \geq 0. 
\end{equation}
To show this, we directly use \eqref{eq: new_A1} and independence. If $m(i) \geq i+t$, then there can be no top-down open crossing of any of the regions $R_i', R_{i+1}', \dots, R_{i+t-1}'$, so by independence, we have
\[
\mathbb{P}(m(i) \geq i+t) = \mathbb{P}\left( \bigcap_{t'=0}^{t-1} \{\exists \text{ top-down open crossing of }R_{i+t'}'\}^c \right) \leq \left(1- \frac{\Cr{c: r_i_assumption_1}}{2}\right)^t,
\]
and this implies \eqref{eq: 2.28}.

Second, Lem.~\ref{lem: plane_to_plane_estimate} and \eqref{eq: 2.28} imply that
\begin{equation}\label{eq: finite_mean}
\mathbb{E}T_f(0,\Gamma_i)^2 < \infty \text{ for all } i \geq 0.
\end{equation}
Indeed, for $x \geq 0$, we have
\[
\mathbb{P}(T_f(0,\Gamma_i) \geq x) \leq \mathbb{P}(m(i) \geq i + \lfloor \sqrt{x} \rfloor) + \mathbb{P}(T_f(P(r_0'),P(r_{i+\lfloor \sqrt{x}\rfloor}')) \geq x).
\]
If we write $\Xi$ for the sum of $\tau_e$ over all $e \in \mathbb{E}_f$ with an endpoint in $S(r_{i^\ast}')$ (where $S(n)$ was defined in \eqref{eq: S_n_P_n_def}), we have 
\begin{equation}\label{eq: more_bandaids}
T_f(P(r_0'),P(r_{i+\lfloor \sqrt{x}\rfloor}')) \leq \Xi + T_f(P(r_{i^\ast}'),P(r_{i+\lfloor \sqrt{x}\rfloor}')) \text{ for } i \geq 0 \text{ and } x \geq (i^\ast)^2.
\end{equation}
Therefore for $x \geq (i^\ast +1)^2$, we can use Lem.~\ref{lem: plane_to_plane_estimate} and \eqref{eq: 2.28} to obtain
\begin{align}
\mathbb{P}(T_f(0,\Gamma_i) \geq x) &\leq e^{-\Cr{c: circuit_appearance_constant}\lfloor \sqrt{x}\rfloor} + \mathbb{P}\left(\Xi \geq \frac{x}{2}\right) + \mathbb{P}\left( T_f(P(r_{i^\ast}'),P(r_{i+\lfloor \sqrt{x}\rfloor}')) \geq \frac{x}{2}\right) \nonumber \\ 
&\leq e^{-\Cr{c: circuit_appearance_constant}\lfloor \sqrt{x}\rfloor} + \mathbb{P}\left(\Xi \geq \frac{x}{2}\right) + \Cr{c: haydns_constant_1} \exp\left( - \Cr{c: haydns_constant_2} \frac{x}{2(i+\lfloor \sqrt{x}\rfloor - i^\ast)}\right) + \Cr{c: haydns_constant_1}\left( \frac{2}{x}\right)^\frac{\eta}{2}. \label{eq: carpool_constant}
\end{align}
Multiplying by $x$, integrating over $x$, and noting that \eqref{eq: gap_var_condition_2} implies $\mathbb{E}\Xi^2 < \infty$ gives \eqref{eq: finite_mean}.

Most of the proof consists of proving variance bounds and a CLT for $T_f(0,\Gamma_{i_0})$ as $i_0 \to \infty$, and then using these results to prove similar ones for $T_f(0,P(n))$ with $n \geq 0$. For this purpose, we let $i_0 \geq 0$ and decompose
\begin{align*}
T_f(0,\Gamma_{i_0}) - \mathbb{E}T_f(0, \Gamma_{i_0}) &= \sum_{i=0}^{i_0} \left( \mathbb{E}\left[ T_f(0, \Gamma_{i_0}) \mid \Sigma_i\right] - \mathbb{E} \left[ T_f(0, \Gamma_{i_0}) \mid \Sigma_{i-1}\right] \right) \\
&:= \sum_{i=0}^{i_0} \Delta_{i,i_0}.
\end{align*}
Because of \eqref{eq: finite_mean}, we have $\mathrm{Var}~T_f(0,\Gamma_{i_0}) = \sum_{i=0}^{i_0}\mathbb{E}\Delta_{i,i_0}^2$.

\subsection{Representations for $\Delta_{i,i_0}$}
To study the variables $\Delta_{i,i_0}$, we give the following representation that is analogous to \cite[Lem.~1]{KZ97}. For its statement, we introduce a second copy $(\Omega', \Sigma', \mathbb{P}')$ of the original probability space, writing a typical element of $\Omega$ as $\omega$ and a typical element of $\Omega'$ as $\omega'$. Until the end of Sec.~\ref{sec: Delta_bounds}, we will include $\omega$ or $\omega'$ in the previous notation, so $m(i,\omega)$ or $\Gamma_i(\omega)$, for example, means $m(i)$ or $\Gamma_i$ evaluated in the outcome $\omega$.

Our formula will differ from that in \cite{KZ97} because of the appearance of events $E_{i,i_0}$ and $E_{i-1,i_0}$, where
\[
E_{i,i_0} = \{m(i,\omega) < i_0\}
\]
It seems that these $E_i$ are needed also in the statements of \cite[Lem.~1]{KZ97} and \cite[Lem.~2]{KZ97}, but their absence does not provide a major problem in the later arguments of that reference. As in \cite{KZ97}, the formula also involves terms like $\mathbb{E}'T_f(\Gamma_i(\omega), \Gamma_{i_0}(\omega'))(\omega')$ which can be thought of as follows. We first fix the outcome $\omega$, and therefore the path $\Gamma_i(\omega)$. Next, we choose $\omega'$ and therefore the path $\Gamma_{i_0}(\omega')$. Last, we evaluate the passage time from the set $\Gamma_i(\omega)$ to the set $\Gamma_{i_0}(\omega')$ in the edge-weights associated to the outcome $\omega'$, and integrate over $\omega'$. The term is therefore a function only of $\omega$.

\begin{lem}\label{lem: representation_1} For every $i_0 \geq 0$, and $i = 0, \dots, i_0$, we have
\[
\Delta_{i,i_0}(\omega) = \begin{cases}
T_f(\Gamma_{i-1}(\omega), \Gamma_i(\omega))(\omega) + \mathbf{1}_{E_{i,i_0}} \mathbb{E}'T_f(\Gamma_i(\omega), \Gamma_{i_0}(\omega'))(\omega')  \\
\hspace{1.55in} - \mathbf{1}_{E_{i-1,i_0}} \mathbb{E}'T_f(\Gamma_{i-1}(\omega), \Gamma_{i_0}(\omega'))(\omega') & \text{ if } i = 1, \dots, i_0, \\
T_f(0,\Gamma_0(\omega))(\omega) + \mathbf{1}_{E_{0,i_0}} \mathbb{E}'T_f(\Gamma_0(\omega),\Gamma_{i_0}(\omega'))(\omega') \\
\hspace{1.13in} - \mathbb{E}'T_f(0,\Gamma_{i_0}(\omega'))(\omega') & \text{ if } i = 0.
\end{cases}
\]
Also if $g$ is a deterministic function such that $g(i) \geq i$ for all $i \geq 0$ and if
\[
0 \leq i_1 < \dots < i_K \leq i_0 \text{ and } i_{k+1} \geq g(i_k) + 2 \text{ for all }k,
\]
then the variables $(\Delta_{i_k,i_0}(\omega) \mathbf{1}_{\{m(i_k,\omega)\leq g(i_k)\}})_{1 \leq k \leq K}$ are independent. 
\end{lem}
\begin{proof}

We first prove that for $i_0 \geq 0$ and $j=0, \dots, i_0$, we have
\begin{equation}\label{eq: conditional_expectation_formula}
\mathbb{E}[T_f(\Gamma_j(\omega), \Gamma_{i_0}(\omega))(\omega) \mid \Sigma_j] = \mathbf{1}_{E_{j,i_0}} \mathbb{E}' T_f(\Gamma_j(\omega),\Gamma_{i_0}(\omega'))(\omega').
\end{equation}
The right side is $\Sigma_j$-measurable, since it depends on $\omega$ only through $\Gamma_j(\omega)$, so we must show that for any relevant $\Gamma, E, B$, we have
\begin{equation}\label{eq: conditional_expectation_definition}
\mathbb{E}T_f(\Gamma_j(\omega), \Gamma_{i_0}(\omega))(\omega) \mathbf{1}_{\Xi_j(\Gamma,E,B)} = \mathbb{E} \left[ \mathbf{1}_{E_{j,i_0}} \mathbb{E}' T_f(\Gamma_j(\omega),\Gamma_{i_0}(\omega'))(\omega') \mathbf{1}_{\Xi_j(\Gamma,E,B)} \right].
\end{equation}

If $\Gamma$ has all vertices $u$ with $u \cdot \mathbf{e}_1 \geq r_{i_0}'$, then
\[
\Xi_j(\Gamma, E, B) \subset \{\Gamma_j(\omega) = \Gamma\} \subset E_{j,i_0}^c
\]
and $\Xi_j(\Gamma, E, B) \subset \{\Gamma_j(\omega) = \Gamma_{i_0}(\omega)\}$. This implies that both sides of \eqref{eq: conditional_expectation_definition} are zero and so \eqref{eq: conditional_expectation_definition} holds. Otherwise, all vertices $u$ of $\Gamma$ have $u \cdot \mathbf{e}_1 < r_{i_0}'$ (there is no other possibility since $\Gamma$ is a top-down crossing of some $R_{i}'$), and the left side of \eqref{eq: conditional_expectation_definition} is
\[
\mathbb{E}T_f(\Gamma, \Gamma_{i_0}(\omega))(\omega) \mathbf{1}_{\Xi_j(\Gamma,E,B)} = \sum_{\Delta} \mathbb{E}T_f(\Gamma,\Delta)(\omega) \mathbf{1}_{\Xi_j(\Gamma,E,B)} \mathbf{1}_{\{\Gamma_{i_0}(\omega) = \Delta\}},
\]
where the sum is over all vertex self-avoiding top-down crossings $\Delta$ of $R_\ell'$ for some $\ell \geq i_0$. The event $\Xi_j(\Gamma,E,B)$ depends on values of $\tau_e$ for $e$ with both endpoints on $\Gamma$ or in its interior. On the other hand, $T_f(\Gamma,\Delta)$ depends on $\tau_e$ for $e$ with at least one endpoint in the exterior of $\Gamma$. The event $\{\Gamma_{i_0}(\omega) = \Delta\}$ depends on $\tau_e$ for $e$ with both endpoints in $\{u : u \cdot \mathbf{e}_1 \geq r_{i_0}'\}$, which is a subset of $\text{ext } \Gamma$. By independence, the above equals
\begin{align*}
\sum_\Delta \mathbb{E}T_f(\Gamma, \Delta)(\omega) \mathbf{1}_{\{\Gamma_{i_0}(\omega) = \Delta\}} \mathbb{E}\mathbf{1}_{\Xi_j(\Gamma,E,B)} &= \mathbb{E}'T_f(\Gamma, \Gamma_{i_0}(\omega'))(\omega') \mathbb{E}\mathbf{1}_{\Xi_j(\Gamma,E,B)} \\
&=\mathbb{E}\left[ \mathbf{1}_{\Xi_j(\Gamma,E,B)} \mathbb{E}'T_f(\Gamma_j(\omega),\Gamma_{i_0}(\omega'))(\omega') \right]
\end{align*}
In this case, $\Xi_j(\Gamma,E,B) \subset E_{j,i_0}$, so this equals the right side of \eqref{eq: conditional_expectation_definition}. We conclude that \eqref{eq: conditional_expectation_formula} holds.

To finish the proof of the equation in Lem.~\ref{lem: representation_1}, we use \eqref{eq: conditional_expectation_formula} to write for $i_0 \geq 0$ and $i = 1, \dots, i_0$
\begin{align*}
\mathbb{E}[T_f(0,\Gamma_{i_0}(\omega))(\omega) \mid \Sigma_i] &= \mathbb{E}[T_f(0,\Gamma_{i-1}(\omega))(\omega) \mid \Sigma_i] + \mathbb{E}[T_f(\Gamma_{i-1}(\omega), \Gamma_i(\omega))(\omega) \mid \Sigma_i] \\
&+ \mathbb{E}[T_f(\Gamma_i(\omega),\Gamma_{i_0}(\omega))(\omega) \mid \Sigma_i] \\
&= T_f(0,\Gamma_{i-1}(\omega))(\omega) + T_f(\Gamma_{i-1}(\omega), \Gamma_i(\omega))(\omega) \\
&+ \mathbf{1}_{E_{i,i_0}} \mathbb{E}' T_f(\Gamma_i(\omega),\Gamma_{i_0}(\omega'))(\omega').
\end{align*}
For $i_0 \geq 0$ and $i=0$, we similarly write
\[
\mathbb{E}[T_f(0,\Gamma_{i_0}(\omega))(\omega) \mid \Sigma_0] =  T_f(0,\Gamma_0(\omega))(\omega) + \mathbf{1}_{E_{0,i_0}} \mathbb{E}'T_f(\Gamma_0(\omega),\Gamma_{i_0}(\omega'))(\omega').
\]
In the case $i \geq 1$, we use the previous two displays with $i,i-1$ to prove the formula for Lem.~\ref{lem: representation_1}. In the case $i=0$ we use the second display along with the fact that $\mathbb{E}[T_f(0,\Gamma_{i_0}(\omega))(\omega) \mid \Sigma_{-1}] = \mathbb{E}'T_f(0,\Gamma_{i_0}(\omega'))(\omega')$ to complete the proof.

For the independence claim, we observe that for $i_0 \geq 0$ and $i = 1, \dots, i_0$, 
\begin{equation}\label{eq: independence_pre_claim}
\Delta_{i,i_0}(\omega) \mathbf{1}_{\{m(i,\omega) \leq g(i)\}}\text{ depends only on } \tau_e \text{ for }e\in \mathcal{E}_i,
\end{equation}
where
\[
\mathcal{E}_i = \{e : e \text{ has both endpoints in } \{u : r_{i-1}' \leq u \cdot \mathbf{e}_1 \leq r_{2g(i)+1}\}\}.
\]
To see why, observe that $\{m(i,\omega) \leq g(i)\}$ depends only on $\tau_e$ for $e \in \mathcal{E}_i$ and if this event occurs, then both paths $\Gamma_{i-1}(\omega)$ and $\Gamma_i(\omega)$ have all their edges (and all edges in the region between them) in $\mathcal{E}_i$. Therefore $T_f(\Gamma_{i-1}(\omega),\Gamma_i(\omega))(\omega) \mathbf{1}_{\{m(i,\omega) \leq g(i)\}}$ depends only on $\tau_e$ for $e \in \mathcal{E}_i$. The same is true for the variables $\mathbf{1}_{\{m(i,\omega) \leq g(i)\}}\mathbb{E}'T_f(\Gamma_{i-1}(\omega),\Gamma_{i_0}(\omega'))(\omega')$ and $\mathbf{1}_{\{m(i,\omega) \leq g(i)\}}\mathbb{E}'T_f(\Gamma_i(\omega),\Gamma_{i_0}(\omega'))(\omega')$. Last, the events $E_{i,i_0}$ and $E_{i-1,i_0}$ are determined by the values of $m(i,\omega)$ and $m(i-1,\omega)$, and when $\{m(i,\omega) \leq g(i)\}$ occurs, these are determined by $\tau_e$ for $e \in \mathcal{E}_i$. We conclude that \eqref{eq: independence_pre_claim} holds. In the case $i_0 \geq 0$ and $i=0$, \eqref{eq: independence_pre_claim} also holds, by similar reasoning, so long as we set
\[
\mathcal{E}_0 = \{e : e \text{ has both endpoints in } \{u : 0 \leq u \cdot \mathbf{e}_1 \leq r_{2g(0)+1}\}\}.
\]

The independence claim follows directly from \eqref{eq: independence_pre_claim}. Indeed, if $0 \leq i_0 < \dots < i_K \leq i_0$ and $i_{k+1} \geq g(i_k) + 2$ for all $k$, then the edge sets $\mathcal{E}_{i_k}$ for $k = 1, \dots, K$ are disjoint. Therefore $(\Delta_{i,i_0}(\omega)\mathbf{1}_{\{m(i_k,\omega) \leq g(i_k)\}})_{1 \leq k \leq K}$ are independent.
\end{proof}

Next, following \cite[Lem.~2]{KZ97}, we give an alternative representation for $\Delta_{i,i_0}(\omega)$.
\begin{lem}\label{lem: representation_2}
Define
\[
\lambda(i,\omega,\omega') = m(m(i,\omega)+1,\omega').
\]
Then for every $i_0 \geq 0$ and $i=1, \dots, i_0$, we have
\begin{align*}
\Delta_{i,i_0}(\omega) &= T_f(\Gamma_{i-1}(\omega),\Gamma_i(\omega))(\omega) \\
&+ \mathbf{1}_{E_{i-1,i_0} \cap E_{i,i_0}} \left[ \mathbb{E}'T_f(\Gamma_i(\omega),\Gamma_{\lambda(i,\omega,\omega')}(\omega'))(\omega') - \mathbb{E}'T_f(\Gamma_{i-1}(\omega),\Gamma_{\lambda(i,\omega,\omega')}(\omega'))(\omega') \right] \\
&- \mathbf{1}_{E_{i-1,i_0}\setminus E_{i,i_0}} \mathbb{E}'T_f(\Gamma_{i-1}(\omega), \Gamma_{i_0}(\omega'))(\omega').
\end{align*}
In the case $i=0$, for every $i_0 \geq 0$, we have
\begin{align*}
\Delta_{0,i_0}(\omega) &= T_f(0,\Gamma_0(\omega))(\omega) \\
&+ \mathbf{1}_{E_{0,i_0}} \left[ \mathbb{E}'T_f(\Gamma_0(\omega), \Gamma_{\lambda(0,\omega,\omega')}(\omega'))(\omega') - \mathbb{E}'T_f(0,\Gamma_{\lambda(0,\omega,\omega')}(\omega'))(\omega') \right]\\
&-\mathbf{1}_{E_{0,i_0}^c} \mathbb{E}'T_f(0,\Gamma_{i_0}(\omega'))(\omega').
\end{align*}
\end{lem}
\begin{proof}
First take $i_0 \geq 0$ and $i = 1, \dots, i_0$. From the formula for $\Delta_{i,i_0}(\omega)$ in Lem.~\ref{lem: representation_1}, if $\omega \in E_{i-1,i_0} \setminus E_{i,i_0}$, then $\mathbf{1}_{E_{i,i_0}} = 0$ but $\mathbf{1}_{E_{i-1,i_0}} = 1$ and the formula reduces to
\[
T_f(\Gamma_{i-1}(\omega),\Gamma_i(\omega))(\omega) - \mathbb{E}'T_f(\Gamma_{i-1}(\omega),\Gamma_{i_0}(\omega'))(\omega'),
\]
which coincides with the above. On the other hand, if $\omega \in E_{i-1,i_0}^c \cap E_{i,i_0}^c$, then both indicators are zero, and the formula gives $T_f(\Gamma_{i-1}(\omega),\Gamma_i(\omega))(\omega)$. Last, if $\omega \in E_{i-1,i_0} \cap E_{i,i_0}$, the formula reads
\begin{equation}\label{eq: its_a_reader}
T_f(\Gamma_{i-1}(\omega), \Gamma_i(\omega))(\omega) +  \mathbb{E}'T_f(\Gamma_i(\omega), \Gamma_{i_0}(\omega'))(\omega') - \mathbb{E}'T_f(\Gamma_{i-1}(\omega),\Gamma_{i_0}(\omega'))(\omega').
\end{equation}
In this case, we have $m(i-1,\omega) \leq m(i,\omega) < i_0$ and therefore $m(i,\omega)+1 \leq i_0$. This implies that $\lambda(i,\omega,\omega') = m(m(i,\omega)+1,\omega') \leq m(i_0,\omega')$ and so in total
\[
m(i-1,\omega) \leq m(i,\omega) < \lambda(i,\omega,\omega') \leq m(i_0,\omega'),
\]
giving
\[
\text{int } \Gamma_{i-1}(\omega) \subset \text{int } \Gamma_i(\omega) \subset \text{int } \Gamma_{\lambda(i,\omega,\omega')}(\omega') \subset \text{int } \Gamma_{i_0}(\omega').
\]
As a result, we have
\[
\mathbb{E}'T_f(\Gamma_i(\omega),\Gamma_{i_0}(\omega'))(\omega') = \mathbb{E}'T_f(\Gamma_i(\omega),\Gamma_{\lambda(i,\omega,\omega')}(\omega'))(\omega') + \mathbb{E}'T_f(\Gamma_{\lambda(i,\omega,\omega')}(\omega'), \Gamma_{i_0}(\omega'))(\omega')
\]
and
\[
\mathbb{E}'T_f(\Gamma_{i-1}(\omega),\Gamma_{i_0}(\omega'))(\omega') = \mathbb{E}'T_f(\Gamma_{i-1}(\omega),\Gamma_{\lambda(i,\omega,\omega')}(\omega'))(\omega') + \mathbb{E}'T_f(\Gamma_{\lambda(i,\omega,\omega')}(\omega'), \Gamma_{i_0}(\omega'))(\omega').
\]
Using the previous two displays in \eqref{eq: its_a_reader} completes the proof in the case $i \geq 1$.

If $i_0 \geq 0$ and $i=0$, we argue similarly. If $\omega \in E_{0,i_0}^c$, then the formula in Lem.~\ref{lem: representation_1} reduces to
\[
T_f(0,\Gamma_0(\omega))(\omega) - \mathbb{E}'T_f(0,\Gamma_{i_0}(\omega'))(\omega').
\]
If instead $\omega \in E_{0,i_0}$, then the formula in Lem.~\ref{lem: representation_1} becomes
\begin{equation}\label{eq: its_a_reader_2}
T_f(0,\Gamma_0(\omega))(\omega) + \mathbb{E}'T_f(\Gamma_0(\omega), \Gamma_{i_0}(\omega'))(\omega') - \mathbb{E}'T_f(0,\Gamma_{i_0}(\omega'))(\omega').
\end{equation}
As above, we have
\[
\text{int } \Gamma_0(\omega) \subset \text{int } \Gamma_{\lambda(0,\omega,\omega')}(\omega') \subset \text{int } \Gamma_{i_0}(\omega'),
\]
and so
\[
\mathbb{E}'T_f(\Gamma_0(\omega),\Gamma_{i_0}(\omega'))(\omega') = \mathbb{E}'T_f(\Gamma_0(\omega),\Gamma_{\lambda(0,\omega,\omega')}(\omega'))(\omega') + \mathbb{E}'T_f(\Gamma_{\lambda(0,\omega,\omega')}(\omega'),\Gamma_{i_0}(\omega'))(\omega')
\]
and
\[
\mathbb{E}'T_f(0,\Gamma_{i_0}(\omega'))(\omega') = \mathbb{E}'T_f(0,\Gamma_{\lambda(0,\omega,\omega')}(\omega'))(\omega') + \mathbb{E}'T_f(\Gamma_{\lambda(0,\omega,\omega')}(\omega'),\Gamma_{i_0}(\omega'))(\omega').
\]
Placing the previous two displays in \eqref{eq: its_a_reader_2} completes the proof of the case $i=0$.
\end{proof}

\subsection{Estimates on $\Delta_{i,i_0}$}\label{sec: Delta_bounds}

The work in this subsection corresponds to \cite[Eqs.~(2.29),(2.32)]{KZ97} and serves to estimate the tail of the distribution of $\Delta_{i,i_0}$ its moments. First we bound the tail of the distribution of $\Delta_{i,i_0}$ as in \cite[Eq.~(2.29)]{KZ97}. We will use the following immediate consequence of \eqref{eq: 2.28}: for each fixed $\omega$, we have
\begin{equation}\label{eq: 2.28b}
\mathbb{P}'(\lambda(i,\omega,\omega') \geq m(i,\omega) + 1 + t) \leq e^{-\Cr{c: circuit_appearance_constant}t} \text{ for all integers } i \geq i^\ast \text{ and } t \geq 0.
\end{equation}
\begin{prop}\label{prop: Delta_concentration_prop}
There exists $\Cl[lgc]{c: Delta_concentration_constant} > 0$ depending only on $\Cr{c: r_i_assumption_1}$ and $\Cr{c: r_i_assumption_2}$ such that for all $i_0 \geq i^\ast+1$ and $i=i^\ast+1, \dots, i_0$, we have
\[
\mathbb{P}(|\Delta_{i,i_0}(\omega)| \geq x) \leq \frac{\Cr{c: Delta_concentration_constant}}{x^{\frac{\eta}{2}}} \text{ for all } x \geq 0.
\]
\end{prop}
\begin{proof}
From Lem.~\ref{lem: representation_2}, we have for $i_0 \geq 0$ and $i = 1, \dots, i_0$,
\begin{align*}
|\Delta_{i,i_0}(\omega)| &\leq T_f(\Gamma_{i-1}(\omega),\Gamma_i(\omega))(\omega) + \mathbf{1}_{E_{i-1,i_0} \cap E_{i,i_0}} \mathbb{E}'T_f(\Gamma_{i-1}(\omega),\Gamma_{\lambda(i,\omega,\omega')}(\omega'))(\omega') \\
&+ \mathbf{1}_{E_{i-1,i_0} \setminus E_{i,i_0}} \mathbb{E}'T_f(\Gamma_{i-1}(\omega),\Gamma_{i_0}(\omega'))(\omega').
\end{align*}
For $\omega \in E_{i-1,i_0} \setminus E_{i,i_0}$, we have $m(i-1,\omega) < i_0 \leq m(i,\omega)$, and so $m(i_0,\omega') \leq m(m(i,\omega)+1,\omega') = \lambda(i,\omega,\omega')$. This implies that $\text{int } \Gamma_{i-1}(\omega) \subset \text{int } \Gamma_{i_0}(\omega') \subset \text{int } \Gamma_{\lambda(i,\omega,\omega')}(\omega')$ and so $T_f(\Gamma_{i-1}(\omega), \Gamma_{i_0}(\omega'))(\omega') \leq T_f(\Gamma_{i-1}(\omega),\Gamma_{\lambda(i,\omega,\omega')}(\omega'))(\omega')$. Placing this in the display gives
\begin{align}
|\Delta_{i,i_0}(\omega)| &\leq T_f(\Gamma_{i-1}(\omega),\Gamma_i(\omega))(\omega) + \mathbf{1}_{E_{i-1,i_0} \cap E_{i,i_0}} \mathbb{E}'T_f(\Gamma_{i-1}(\omega),\Gamma_{\lambda(i,\omega,\omega')}(\omega'))(\omega') \nonumber \\
&+ \mathbf{1}_{E_{i-1,i_0} \setminus E_{i,i_0}} \mathbb{E}'T_f(\Gamma_{i-1}(\omega),\Gamma_{\lambda(i,\omega,\omega')}(\omega'))(\omega') \nonumber \\
&\leq T_f(\Gamma_{i-1}(\omega),\Gamma_i(\omega))(\omega) +  \mathbb{E}'T_f(\Gamma_{i-1}(\omega),\Gamma_{\lambda(i,\omega,\omega')}(\omega'))(\omega'). \label{eq: Delta_inequality_1}
\end{align}

We begin with the second term $\mathbb{E}'T_f(\Gamma_{i-1}(\omega), \Gamma_{\lambda(i,\omega,\omega')}(\omega'))(\omega')$ of \eqref{eq: Delta_inequality_1} and use \eqref{eq: 2.28b} to estimate, for $i \geq i^\ast+1$, $y \geq 0$, integer $t \geq 0$, and fixed $\omega$, 
\begin{align}
\mathbb{P}'(T_f(\Gamma_{i-1}(\omega), \Gamma_{\lambda(i,\omega,\omega')}(\omega'))(\omega') \geq y) &\leq \mathbb{P}'(\lambda(i,\omega,\omega') \geq m(i,\omega) + 1 + t) \nonumber \\
&+ \mathbb{P}'(T_f(P(r_{i-1}'),P(r_{m(i,\omega)+1+t}'))(\omega') \geq y) \nonumber \\
&\leq e^{-\Cr{c: circuit_appearance_constant}t} + \mathbb{P}'(T_f(P(r_{i-1}'),P(r_{m(i,\omega)+1+t}'))(\omega') \geq y). \label{eq: vegetarian_supreme_1}
\end{align}
We now use Lem.~\ref{lem: plane_to_plane_estimate} to obtain the bound
\[
e^{-\Cr{c: circuit_appearance_constant}t} + \Cr{c: haydns_constant_1} \exp\left( - \Cr{c: haydns_constant_2} \frac{y}{m(i,\omega)-i + 2 + t}\right) + \frac{\Cr{c: haydns_constant_1}}{y^\frac{\eta}{2}}.
\]
Restricting to $y \geq 1$, set $t = \lfloor \sqrt{y} \rfloor$ to produce
\begin{align*}
&\mathbb{P}'(T_f(\Gamma_{i-1}(\omega), \Gamma_{\lambda(i,\omega,\omega')}(\omega'))(\omega') \geq y)  \\
\leq~& e^{-\Cr{c: circuit_appearance_constant}\lfloor \sqrt{y} \rfloor} + \Cr{c: haydns_constant_1} \exp\left( - \Cr{c: haydns_constant_2} \frac{y}{m(i,\omega)-i + 2 + \lfloor \sqrt{y} \rfloor }\right) + \frac{\Cr{c: haydns_constant_1}}{y^\frac{\eta}{2}}  \\
\leq~& e^{-\frac{\Cr{c: circuit_appearance_constant}}{2} \sqrt{y} } + \Cr{c: haydns_constant_1} \exp\left( - \Cr{c: haydns_constant_2} \frac{y}{m(i,\omega)-i + 2 +  \sqrt{y}  }\right) + \frac{\Cr{c: haydns_constant_1}}{y^\frac{\eta}{2}}  \\
\leq~& 2\Cr{c: haydns_constant_1} \exp\left( - \Cl[smc]{c: haydns_constant_3} \frac{y}{m(i,\omega) - i + 2 + \sqrt{y}}\right) + \frac{\Cr{c: haydns_constant_1}}{y^{\frac{\eta}{2}}} 
\end{align*}
for $\Cr{c: haydns_constant_3} = \min\{\Cr{c: circuit_appearance_constant}/2, \Cr{c: haydns_constant_2}\}$. As $y \geq 1$, the above implies
\begin{equation}\label{eq: vegetarian_supreme_2}
\mathbb{P}'(T_f(\Gamma_{i-1}(\omega), \Gamma_{\lambda(i,\omega,\omega')}(\omega'))(\omega') \geq y) \leq 2\Cr{c: haydns_constant_1} \exp\left( - \frac{\Cr{c: haydns_constant_3}}{3} \cdot  \frac{y}{m(i,\omega) - i  + \sqrt{y}}\right) + \frac{\Cr{c: haydns_constant_1}}{y^{\frac{\eta}{2}}}.
\end{equation}

By integrating this inequality, we obtain $\Cl[lgc]{c: haydns_constant_4}$ depending only on $\Cr{c: r_i_assumption_1}$ and $\Cr{c: r_i_assumption_2}$ such that for all $\omega$ and $i \geq i^\ast+1$, we have
\begin{equation}\label{eq: my_favorite_mean_bound}
\mathbb{E}'T_f(\Gamma_{i-1}(\omega), \Gamma_{\lambda(i,\omega,\omega')}(\omega'))(\omega') \leq \Cr{c: haydns_constant_4} (m(i,\omega) - i + 1).
\end{equation}
So from \eqref{eq: Delta_inequality_1}, we find for $i_0 \geq i^\ast+1$, $i=i^\ast+1, \dots, i_0$, $x \geq 0$, and integers $t \in [0,\lfloor x/(2\Cr{c: haydns_constant_4})\rfloor]$ that
\begin{align}
\mathbb{P}(|\Delta_{i,i_0}(\omega)| \geq x) &\leq \mathbb{P}(T_f(\Gamma_{i-1}(\omega), \Gamma_i(\omega))(\omega) + \Cr{c: haydns_constant_4} (m(i,\omega)-i+1) \geq x) \nonumber \\
&\leq \mathbb{P}\left(T_f(\Gamma_{i-1}(\omega),\Gamma_i(\omega))(\omega) + \Cr{c: haydns_constant_4} t \geq x, m(i,\omega) < i+t\right) \nonumber \\
&+ \mathbb{P}(m(i,\omega) \geq i+t) \nonumber \\
&\leq \mathbb{P}\left(T_f(P(r_{i-1}'),P(r_{i+t}'))(\omega) \geq \frac{x}{2} \right) + e^{-\Cr{c: circuit_appearance_constant}t}, \label{eq: vegetarian_supreme_3}
\end{align}
where the last line used \eqref{eq: 2.28}. By applying Lem.~\ref{lem: plane_to_plane_estimate}, we get
\[
\mathbb{P}(|\Delta_{i,i_0}(\omega)| \geq x) \leq \Cr{c: haydns_constant_1} \exp\left( - \Cr{c: haydns_constant_2} \frac{x}{2(1+t)}\right) + \frac{2^{\frac{\eta}{2}} \Cr{c: haydns_constant_1}}{x^{\frac{\eta}{2}}} + e^{-\Cr{c: circuit_appearance_constant}t}.
\]
Setting $t = \lfloor \sqrt{x} \rfloor$ completes the proof of  Prop.~\ref{prop: Delta_concentration_prop}.
\end{proof}

\begin{rem}
In addition to \eqref{eq: Delta_inequality_1} we have 
\begin{equation}\label{eq: Delta_inequality_2}
|\Delta_{0,i_0}(\omega)| \leq T_f(0,\Gamma_0(\omega))(\omega) + \mathbb{E}'T_f(0,\Gamma_{\lambda(0,\omega,\omega')}(\omega'))(\omega'),
\end{equation}
which is proved in the same way. We will use the following consequence of these inequalities later: we have
\begin{equation}\label{eq: my_favorite_bandaid}
\sup_{i_0 \geq j_0} \mathbb{E}\Delta_{i,i_0}^2 < \infty \text{ for any } j_0 \geq 1 \text{ and } i=0, \dots, j_0.
\end{equation}
This holds because if $j_0 \geq 1$ and $i = 0, \dots, j_0$, we can use \eqref{eq: Delta_inequality_1} and \eqref{eq: Delta_inequality_2} to deduce that
\begin{align*}
\mathbb{E} \Delta_{i,i_0}^2 &\leq \mathbb{E} \left( T_f(0,\Gamma_{j_0}(\omega))(\omega) + \mathbb{E}'T_f(0,\Gamma_{\lambda(j_0,\omega,\omega')}(\omega'))(\omega') \right)^2 \\
&\leq 2 \mathbb{E}T_f(0,\Gamma_{j_0}(\omega))(\omega)^2 + 2\mathbb{E}\mathbb{E}'T_f(0,\Gamma_{\lambda(j_0,\omega,\omega')}(\omega'))(\omega')^2 
\end{align*}
The first term is finite by \eqref{eq: finite_mean}. We rewrite the second term and apply H\"older's inequality with exponents $(\eta+4)/4$ and $(\eta+4)/\eta$ to obtain
\begin{align*}
&2\sum_{\ell = j_0+1}^\infty \mathbb{E}\mathbb{E}' T_f(0,\Gamma_\ell(\omega'))(\omega')^2\mathbf{1}_{\{\lambda(j_0,\omega,\omega') = \ell\}} \\
\leq~& 2\sum_{\ell=j_0+1}^\infty \left( \mathbb{E}\mathbb{E}' T_f(0,\Gamma_\ell(\omega'))(\omega')^\frac{\eta+4}{2} \right)^{\frac{4}{\eta+4}} \left( \mathbb{E}\mathbb{E}' \mathbf{1}_{\{\lambda(j_0,\omega,\omega')=\ell\}}\right)^{\frac{\eta}{\eta+4}}.
\end{align*}
By integrating \eqref{eq: carpool_constant}, we find that for any $s < \eta/2$, there exists $\Cl[lgc]{c: carpool_constant_2} = \Cr{c: carpool_constant_2}(s)$ such that $\mathbb{E}T_f(0,\Gamma_\ell)^s \leq \Cr{c: carpool_constant_2}(1+\ell)^s$ for all $\ell \geq 0$. Applying this with $s = (\eta+4)/2$ provides a constant $\Cl[lgc]{c: carpool_constant_3}$ such that 
\begin{equation}\label{eq: carpool_sum}
\mathbb{E}\mathbb{E}' T_f(0,\Gamma_{\lambda(j_0,\omega,\omega')}(\omega'))(\omega')^2 \leq \Cr{c: carpool_constant_3} \sum_{\ell=j_0+1} (1+\ell)^2 \left( \mathbb{E}\mathbb{E}' \mathbf{1}_{\{\lambda(j_0,\omega,\omega')=\ell\}}\right)^{\frac{\eta}{\eta+4}}.
\end{equation}
Because of \eqref{eq: 2.28} and \eqref{eq: 2.28b}, the quantity
\[
\mathbb{E}\mathbb{E}' \mathbf{1}_{\{\lambda(j_0,\omega,\omega')=\ell\}} \leq \mathbb{E}\mathbb{E}' \left[ \mathbf{1}_{\left\{m(j_0,\omega) \geq \frac{\ell}{2}\right\}} + \mathbf{1}_{\left\{ \lambda(j_0,\omega,\omega') \geq m(j_0,\omega) + \frac{\ell}{2}\right\}}\right]
\]
decays exponentially in $\ell$, so the sum on the right of \eqref{eq: carpool_sum} is finite, and this gives \eqref{eq: my_favorite_bandaid}.
\end{rem}

We finish this section with an estimate on the second moment of $\Delta_{i,i_0}$ that uses Prop.~\ref{prop: Delta_concentration_prop} and Assumption A3. 
Taking $\Cr{c: haydns_constant_4}$ from \eqref{eq: my_favorite_mean_bound}, we may choose a positive integer $i^\ast$ such that not only do \eqref{eq: new_A1} and \eqref{eq: new_A2_pro} hold, but also (using A3 with $M=2\Cr{c: haydns_constant_4}/\delta$) we have 
\begin{equation}\label{eq: new_A3}
\mathbb{P}\left( T_f^B(P(r_i),P(r_{i+1})) \geq \frac{2\Cr{c: haydns_constant_4}}{\delta} \right) \geq \frac{\Cr{c: r_i_assumption_3}}{2} \text{ for } i \geq i^\ast.
\end{equation}

\begin{prop}\label{prop: second_moment_prop}
There exist $\Cl[lgc]{c: Delta_second_moment_upper}$ depending only on $\Cr{c: r_i_assumption_1}$ and $\Cr{c: r_i_assumption_2}$, and $\Cl[smc]{c: Delta_second_moment_lower} > 0$ depending only on $\Cr{c: r_i_assumption_1}$ and $\Cr{c: r_i_assumption_3}$ such that
\begin{enumerate}
\item  for all $i_0 \geq i^\ast+1$ and $i= i^\ast+1, \dots, i_0$, we have $\mathbb{E}\Delta_{i,i_0}^2(\omega) \leq \Cr{c: Delta_second_moment_upper}$, and
\item for all $i_0 \geq i^\ast+1$ and $i = i^\ast+1, \dots, i_0$, we have $\mathbb{E}\Delta_{i,i_0}^2(\omega) \geq \Cr{c: Delta_second_moment_lower}$.
\end{enumerate}
\end{prop}
\begin{proof}
Item 1 follows directly from integrating the bound in Prop.~\ref{prop: Delta_concentration_prop}. For item 2, we use Assumption A3 in the form of \eqref{eq: new_A3}. From Lem.~\ref{lem: representation_2}, we have for $i_0 \geq 0$ and $i=1, \dots, i_0$
\begin{align*}
\Delta_{i,i_0}(\omega) &\geq T_f(\Gamma_{i-1}(\omega),\Gamma_i(\omega))(\omega) - \mathbf{1}_{E_{i-1,i_0} \cap E_{i,i_0}} \mathbb{E}'T_f(\Gamma_{i-1}(\omega),\Gamma_{\lambda(i,\omega,\omega')}(\omega'))(\omega') \\
&- \mathbf{1}_{E_{i-1,i_0} \setminus  E_{i,i_0}} \mathbb{E}'T_f(\Gamma_{i-1}(\omega),\Gamma_{i_0}(\omega'))(\omega').
\end{align*}
The same argument as for \eqref{eq: Delta_inequality_1} produces 
\[
\Delta_{i,i_0}(\omega) \geq T_f(\Gamma_{i-1}(\omega),\Gamma_i(\omega))(\omega) -  \mathbb{E}'T_f(\Gamma_{i-1}(\omega),\Gamma_{\lambda(i,\omega,\omega')}(\omega'))(\omega'),
\]
and by using \eqref{eq: my_favorite_mean_bound}, we find 
\[
\Delta_{i,i_0}(\omega) \geq T_f(\Gamma_{i-1}(\omega),\Gamma_i(\omega))(\omega) - \Cr{c: haydns_constant_4} (m(i,\omega) - i + 1)
\]
for $i_0 \geq i^\ast+1$ and $i = i^\ast+1, \dots, i_0$. 

To estimate $\Delta_{i,i_0}(\omega)$ from below, we note that if $m(i,\omega) = i$ and $T_f(\Gamma_{i-1}(\omega),\Gamma_i(\omega))(\omega) \geq 2 \Cr{c: haydns_constant_4}$, then we have $\Delta_{i,i_0}(\omega) \geq \Cr{c: haydns_constant_4}$. If, in addition, $m(i-1,\omega) = i-1$, then any path in $\mathbb{G}_f$ from $\Gamma_{i-1}(\omega)$ to $\Gamma_i(\omega)$ must contain a subpath that connects $P(r_{2i-1})$ to $P(r_{2i})$. Therefore using independence and \eqref{eq: new_A1}, we find
\begin{align*}
\mathbb{P}(\Delta_{i,i_0}(\omega) \geq \Cr{c: haydns_constant_4}) &\geq \mathbb{P}(m(i-1,\omega) = i-1, m(i,\omega) = i, T_f(P(r_{2i-1}),P(r_{2i})) \geq 2\Cr{c: haydns_constant_4}) \\
&= \mathbb{P}(m(i-1,\omega)=i-1)\mathbb{P}(m(i,\omega) = i) \mathbb{P}(T_f(P(r_{2i-1}),P(r_{2i})) \geq 2\Cr{c: haydns_constant_4}) \\
&\geq \frac{\Cr{c: r_i_assumption_1}^2}{4} \mathbb{P}(T_f(P(r_{2i-1}),P(r_{2i})) \geq 2\Cr{c: haydns_constant_4}).
\end{align*}
Here we have used that for $j \geq 0$, $m(j,\omega) = j$ exactly when $R_{2j}$ has a top-down open crossing. If $\gamma$ is an optimal path for $T_f(P(r_{2i-1}),P(r_{2i}))$, then we have $T_f(\gamma) = \sum_{e \in \gamma} \tau_e \geq \delta \sum_{e \in \gamma} t_e \geq \delta T_f^B(P(r_{2i-1}),P(r_{2i}))$, and we conclude that
\[
\mathbb{P}(\Delta_{i,i_0}(\omega) \geq \Cr{c: haydns_constant_4}) \geq \frac{\Cr{c: r_i_assumption_1}^2}{4} \mathbb{P}\left( T_f^B(P(r_{2i-1}),P(r_{2i})) \geq \frac{2\Cr{c: haydns_constant_4}}{\delta} \right).
\]
Along with \eqref{eq: new_A3}, this implies the lower bound in the proposition.
\end{proof}

\subsection{Preliminary CLT}
We can now follow \cite[Lem.~4]{KZ97} to prove a CLT for $T_f(0,\Gamma_{i_0}(\omega))(\omega)$. We will no longer include $\omega$ in all the notation, since we do not need the second probability space $(\Omega', \Sigma', \mathbb{P}')$.
\begin{prop}\label{prop: preliminary_CLT}
As $i_0 \to \infty$,
\[
\frac{T_f(0,\Gamma_{i_0}) - \mathbb{E}T_f(0,\Gamma_{i_0})}{\sqrt{\mathrm{Var}~T_f(0,\Gamma_{i_0})}} \Rightarrow N(0,1).
\]
\end{prop}
\begin{proof}
As in \cite{KZ97}, we directly apply McLeish's martingale CLT from \cite[Thm.~2.3]{M74}. Defining
\[
X_{i,i_0} = \frac{\Delta_{i,i_0}}{\sqrt{\sum_{i=0}^{i_0}\mathbb{E} \Delta_{i,i_0}^2}},
\]
McLeish's theorem states that it suffices to verify the following conditions:
\begin{enumerate}
\item[C1.] $\sup_{i_0 \geq 0} \mathbb{E}\left( \max_{i=0, \dots, i_0} X_{i,i_0} \right)^2 < \infty$,
\item[C2.] $\max_{i=0, \dots, i_0} |X_{i,i_0}| \to 0$ in probability as $i_0 \to \infty$, and
\item[C3.] $\sum_{i=0}^{i_0} X_{i,i_0}^2 \to 1$ in probability as $i_0 \to \infty$.
\end{enumerate}

To verify C2, we apply Prop.~\ref{prop: Delta_concentration_prop} and item 2 of Prop.~\ref{prop: second_moment_prop}, along with a union bound, to find for $\epsilon>0$
\begin{equation}\label{eq: gala_apple_a}
\mathbb{P}\left( \max_{i=i^\ast+1, \dots, i_0} |X_{i,i_0}| \geq \epsilon \right) \leq \mathbb{P}\left( \max_{i=i^\ast+1, \dots, i_0} |\Delta_{i,i_0}| \geq \epsilon \sqrt{(i_0-i^\ast)\Cr{c: Delta_second_moment_lower}} \right) \leq \frac{\Cr{c: Delta_concentration_constant} (i_0-i^\ast)}{\epsilon^{\frac{\eta}{2}} \left( ( i_0-i^\ast) \Cr{c: Delta_second_moment_lower} \right)^{\frac{\eta}{4}}}.
\end{equation}
This converges to zero as $i_0 \to \infty$. Furthermore, by \eqref{eq: my_favorite_bandaid} and item 2 of Prop.~\ref{prop: second_moment_prop}, we have
\begin{equation}\label{eq: gala_apple_b}
\mathbb{E} \left( \max_{i=0, \dots, i^\ast} X_{i,i_0}\right)^2 \leq \frac{\mathbb{E}\left( \sum_{i=0}^{i^\ast} |\Delta_{i,i_0}|\right)^2}{\sum_{i=0}^{i_0}\mathbb{E}\Delta_{i,i_0}^2} \leq (i^\ast+1) \frac{\sum_{i=0}^{i^\ast} \mathbb{E}\Delta_{i,i_0}^2}{\sum_{i=0}^{i_0}\mathbb{E}\Delta_{i,i_0}^2} \to 0 \text{ as } i_0 \to \infty.
\end{equation}
Displays \eqref{eq: gala_apple_a} and \eqref{eq: gala_apple_b} imply that C2 holds. If $i_0 \geq i^\ast+1$, then \eqref{eq: gala_apple_a} implies that $\mathbb{P}(\max_{i=i^\ast+1, \dots, i_0}|X_{i,i_0}| \geq \epsilon) \leq \Cl[lgc]{c: suitcase_constant}/\epsilon^{\eta/2}$ for $\Cr{c: suitcase_constant} = \Cr{c: Delta_concentration_constant}/\Cr{c: Delta_second_moment_lower}^{\eta/4}$. Because $\eta > 4$, this and \eqref{eq: gala_apple_b} are enough to conclude C1.

For C3, in view of item 2 of Prop.~\ref{prop: second_moment_prop}, we may prove that
\[
\frac{1}{i_0} \left[ \sum_{i=0}^{i_0} \Delta_{i,i_0}^2 - \sum_{i=0}^{i_0} \mathbb{E}\Delta_{i,i_0}^2 \right] \to 0 \text{ in probability as }i_0 \to \infty,
\]
and by \eqref{eq: my_favorite_bandaid}, it suffices to show that
\begin{equation}\label{eq: to_prove_C3}
\frac{1}{i_0} \left[ \sum_{i=i^\ast+1}^{i_0} \Delta_{i,i_0}^2 - \sum_{i=i^\ast+1}^{i_0} \mathbb{E}\Delta_{i,i_0}^2 \right] \to 0 \text{ in probability as }i_0 \to \infty.
\end{equation}

This is a type of weak law of large numbers proved by a truncation argument. Defining the truncation
\[
\widetilde{\Delta}_{i,i_0} = \Delta_{i,i_0} \mathbf{1}_{\left\{m(i) \leq i + \frac{3}{\Cr{c: circuit_appearance_constant}} \log i_0 \text{ and } |\Delta_{i,i_0}| \leq i_0^{\frac{2}{\eta}} \log i_0\right\}},
\]
we see from \eqref{eq: 2.28} and Prop.~\ref{prop: Delta_concentration_prop} that
\begin{align}
&\mathbb{P}(\Delta_{i,i_0} \neq \widetilde{\Delta}_{i,i_0} \text{ for some } i = i^\ast+1, \dots, i_0) \nonumber \\
\leq~&\sum_{i=i^\ast+1}^{i_0} \left[ \mathbb{P}\left(m(i) > i + \frac{3}{\Cr{c: circuit_appearance_constant}} \log i_0\right) + \mathbb{P}\left( |\Delta_{i,i_0}| > i_0^{\frac{2}{\eta}} \log i_0\right) \right] \nonumber \\
\leq~& \sum_{i=i^\ast+1}^{i_0} \left( i_0^{-3} + \frac{\Cr{c: Delta_concentration_constant}}{i_0 \left( \log i_0\right)^{\frac{\eta}{2}}} \right) \to 0 \text{ as } i_0 \to \infty. \label{eq: go_ahead_and_substitute}
\end{align}
Furthermore, using H\"older's inequality with exponents $(4+\eta)/8$ and $(4+\eta)/(\eta-4)$, we find
\begin{align*}
&\sum_{i=i^\ast+1}^{i_0} \left( \mathbb{E}\Delta_{i,i_0}^2 - \mathbb{E} \widetilde{\Delta}_{i,i_0}^2 \right) \\
\leq~& \sum_{i=i^\ast+1}^{i_0} \mathbb{E}\Delta_{i,i_0}^2 \mathbf{1}_{\left\{m(i) > i + \frac{3}{\Cr{c: circuit_appearance_constant}} \log i_0 \text{ or } |\Delta_{i,i_0}| > i_0^{\frac{2}{\eta}} \log i_0\right\}} \\
\leq~&\sum_{i=i^\ast+1}^{i_0} \left( \mathbb{E} \Delta_{i,i_0}^{1+ \frac{\eta}{4}} \right)^{\frac{8}{4+\eta}} \left( \mathbb{P}\left(m(i) > i + \frac{3}{\Cr{c: circuit_appearance_constant}} \log i_0\right) + \mathbb{P}\left(|\Delta_{i,i_0}| > i_0^{\frac{2}{\eta}} \log i_0\right) \right)^{\frac{\eta-4}{4+\eta}}.
\end{align*}
Prop.~\ref{prop: Delta_concentration_prop} and the fact that $\eta > 4$ imply that $\mathbb{E}\Delta_{i,i_0}^{1+\eta/4}$ is bounded for $i^\ast+1 \leq i \leq i_0$, so we can apply \eqref{eq: 2.28} to find $\Cl[lgc]{c: weird_bird_constant} > 0$ such that 
\[
\frac{1}{i_0}\sum_{i=i^\ast+1}^{i_0} \left( \mathbb{E}\Delta_{i,i_0}^2 - \mathbb{E} \widetilde{\Delta}_{i,i_0}^2 \right) \leq \frac{\Cr{c: weird_bird_constant}}{i_0} \sum_{i=i^\ast+1}^{i_0}\left( i_0^{-3} + \frac{\Cr{c: Delta_concentration_constant}}{i_0 \left( \log i_0\right)^{\frac{\eta}{2}}} \right)^{\frac{\eta-4}{4+\eta}} \to 0 \text{ as } i_0 \to \infty.
\] 
This display, along with \eqref{eq: go_ahead_and_substitute}, implies that \eqref{eq: to_prove_C3} will follow as long as we show that
\begin{equation}\label{eq: final_to_prove_C3}
\frac{1}{i_0} \sum_{i=i^\ast + 1}^{i_0}\left[  \widetilde{\Delta}_{i,i_0}^2 - \mathbb{E}\widetilde{\Delta}_{i,i_0}^2 \right] \to 0 \text{ in probability as }i_0 \to \infty.
\end{equation}

To show \eqref{eq: final_to_prove_C3}, we observe that $\widetilde{\Delta}_{i,i_0}$ and $\widetilde{\Delta}_{j,i_0}$ are independent so long as $|i-j| \geq (3/\Cr{c: circuit_appearance_constant}) \log i_0 + 2$ due to the independence claim of Lem.~\ref{lem: representation_1} and the fact that 
\[
\widetilde{\Delta}_{i,i_0} = \Delta_{i,i_0} \mathbf{1}_{\left\{m(i) \leq i + \frac{3}{\Cr{c: circuit_appearance_constant}} \log i_0\right\}} \mathbf{1}_{\left\{ |\Delta_{i,i_0}|\mathbf{1}_{\left\{m(i) \leq i + \frac{3}{\Cr{c: circuit_appearance_constant}} \log i_0\right\}} \leq i_0^{\frac{2}{\eta}} \log i_0\right\}}.
\]
Therefore
\begin{equation}\label{eq: calculator_bound}
\mathrm{Var}\left( \sum_{i=i^\ast+1}^{i_0} \left( \widetilde{\Delta}_{i,i_0}^2 - \mathbb{E}\widetilde{\Delta}_{i,i_0}^2 \right) \right) \leq 2 \sum_{i=i^\ast+1}^{i_0} \sum_{j : i \leq j \leq \frac{3}{\Cr{c: circuit_appearance_constant}} \log i_0 + 2} \sqrt{\mathrm{Var}\left( \widetilde{\Delta}_{i,i_0}^2\right) \mathrm{Var}\left( \widetilde{\Delta}_{j,i_0}^2\right)}.
\end{equation}
Note that there exists $\Cl[lgc]{c: iguanas_favorite_constant} > 0$ such that for $i_0 \geq i^\ast +1$ and $i = i^\ast+1, \dots, i_0$,
\begin{align*}
\mathrm{Var}\left( \widetilde{\Delta}_{i,i_0}^2\right) \leq \mathbb{E} \widetilde{\Delta}_{i,i_0}^4(\omega) &\leq 4 \int_0^{i_0^{\frac{2}{\eta}} \log i_0} x^3 \mathbb{P}\left( |\Delta_{i,i_0}| \geq x\right)~\text{d}x \\
&\leq 4 \int_0^{i_0^{\frac{2}{\eta}} \log i_0} x^3 \min\left\{ 1,\frac{\Cr{c: Delta_concentration_constant}}{x^\frac{\eta}{2}} \right\}~\text{d}x \\
&\leq \Cr{c: iguanas_favorite_constant} \left( 1+ i_0^{\frac{8-\eta}{\eta}} \left( \log i_0 \right)^{\frac{8-\eta}{2}}\right).
\end{align*}
%By \eqref{eq: my_favorite_bandaid}, this bound holds for all $i_0 \geq i^\ast+1$ and $i=0, \dots, i_0$ if we simply increase the constant $\Cr{c: iguanas_favorite_constant}$. 
So \eqref{eq: calculator_bound} implies that
\begin{align*}
&\frac{1}{i_0^2} \mathrm{Var}\left( \sum_{i=i^\ast+1}^{i_0} \left( \widetilde{\Delta}_{i,i_0}^2 - \mathbb{E}\widetilde{\Delta}_{i,i_0}^2 \right) \right)  \\
\leq~& \frac{ 2\Cr{c: iguanas_favorite_constant} (i_0-i^\ast)}{i_0^2} \left( \frac{3}{\Cr{c: circuit_appearance_constant}} \log i_0 + 3\right)\left( 1+ i_0^{\frac{8-\eta}{\eta}} \left( \log i_0 \right)^{\frac{8-\eta}{2}}\right) \to 0 \text{ as } i_0 \to \infty.
\end{align*}
This is sufficient, along with Chebyshev's inequality, to conclude \eqref{eq: final_to_prove_C3}, and complete the proof of Prop.~\ref{prop: preliminary_CLT}.
\end{proof}

\subsection{Proofs of Thm.~\ref{thm: updated_KZ_variance} and Thm.~\ref{thm: updated_KZ_CLT}}\label{sec: KZ_proofs}
For $n \geq 1$ recall the definition of $\iota(n)$ in \eqref{eq: iota_def} and note that 
\[
r_{\iota(n)-1}' < n \leq r_{\iota(n)}'.
\]
We will first relate $T_f(0,P(n))$ to $T_f(0,\Gamma_{\iota(n)})$, since we already have results for the latter quantity. Specifically, we will prove that there exists $\Cl[lgc]{c: babys_first_playset_constant}$ such that for all $n$ with $\iota(n) \geq i^\ast+1$, we have
\begin{equation}\label{eq: to_prove_completing_the_proof}
\mathbb{P}(|T_f(0,P(n)) - T_f(0,\Gamma_{\iota(n)})| \geq x) \leq \frac{\Cr{c: babys_first_playset_constant}}{x^\frac{\eta}{2}} \text{ for } x \geq 0.
\end{equation}

To show \eqref{eq: to_prove_completing_the_proof}, we begin by noting that
\begin{equation}\label{eq: easy_upper_bound}
T_f(0,P(n)) \leq T_f(0,\Gamma_{\iota(n)}).
\end{equation}
Furthermore, if for some integers $k,t \geq 1$, we have $m(\iota(n)-k) < \iota(n)-1$ and $m(\iota(n)) < \iota(n) + t$, then each path from $0$ to $P(n)$ in $\mathbb{G}_f$ must intersect $\Gamma_{\iota(n)-k}$ and we also have $\text{int}~\Gamma_{\iota(n)} \subset S(r_{\iota(n)+t}')$, so
\begin{align*}
T_f(0,P(n)) \geq T_f(0,\Gamma_{\iota(n)-k}) &= T_f(0,\Gamma_{\iota(n)}) - T_f(\Gamma_{\iota(n)-k}, \Gamma_{\iota(n)}) \\
&\geq T_f(0,\Gamma_{\iota(n)}) - T_f(P(r_{\iota(n)-k}'),P(r_{\iota(n)+t}')).
\end{align*}
The last two displays imply that if there exist such $k,t$, then
\[
|T_f(0,P(n)) - T_f(0,\Gamma_{\iota(n)})| \leq T_f(P(r_{\iota(n)-k}'),P(r_{\iota(n)+t}')),
\]
so for all $x \geq 0$, $k = 1, \dots, \iota(n) - i^\ast$, $t \geq 0$, and $n$ such that $\iota(n) \geq i^\ast+1$, we have 
\begin{align*}
\mathbb{P}(|T_f(0,P(n)) - T_f(0,\Gamma_{\iota(n)})| \geq x) &\leq \mathbb{P}(m(\iota(n)-k) \geq \iota(n)-1) + \mathbb{P}(m(\iota(n)) \geq \iota(n) + t) \\
&+ \mathbb{P}(T_f(P(r_{\iota(n)-k}'),P(r_{\iota(n)+t}')) \geq x) \\
&\leq e^{-\Cr{c: circuit_appearance_constant}(k-1)} + e^{-\Cr{c: circuit_appearance_constant}t} + \Cr{c: haydns_constant_1} \exp\left( - \Cr{c: haydns_constant_2} \frac{x}{t+k}\right) + \frac{\Cr{c: haydns_constant_1}}{x^\frac{\eta}{2}}.
\end{align*}
Here we have used \eqref{eq: 2.28} and Lem.~\ref{lem: plane_to_plane_estimate}. Now set $t = k = \lfloor \sqrt{x}\rfloor$, provided that $x \in [1, (\iota(n)-i^\ast+1)^2)$ (so that the condition on $k \in \{1, \dots, \iota(n)-i^\ast\}$ holds) to obtain
\[
\mathbb{P}(|T_f(0,P(n)) - T_f(0,\Gamma_{\iota(n)})| \geq x)
\leq e^{-\Cr{c: circuit_appearance_constant}(\lfloor \sqrt{x}\rfloor -1)} + e^{-\Cr{c: circuit_appearance_constant}\lfloor \sqrt{x} \rfloor } + \Cr{c: haydns_constant_1} \exp\left( - \Cr{c: haydns_constant_2} \frac{x}{2 \lfloor \sqrt{x}\rfloor}\right) + \frac{\Cr{c: haydns_constant_1}}{x^\frac{\eta}{2}}.
\]
Therefore there exists $\Cl[lgc]{c: ravels_constant}$ such that for $n$ with $\iota(n) \geq i^\ast+1$, we have
\begin{equation}\label{eq: babys_first_playset}
\mathbb{P}(|T_f(0,P(n)) - T_f(0,\Gamma_{\iota(n)})| \geq x)
\leq \frac{\Cr{c: ravels_constant}}{x^{\frac{\eta}{2}}} \text{ for } x \in [0,(\iota(n)-i^\ast+1)^2).
\end{equation}

When $\iota(n) \geq i^\ast+1$ but $x \geq (\iota(n)-i^\ast+1)^2$, we instead only use \eqref{eq: easy_upper_bound} to estimate for integer $t \geq 1$
\begin{align*}
\mathbb{P}(|T_f(0,P(n)) - T_f(0,\Gamma_{\iota(n)})| \geq x) &\leq \mathbb{P}(T_f(0, \Gamma_{\iota(n)}) \geq x) \\
&\leq \mathbb{P}(m(\iota(n)) \geq \iota(n) + t) + \mathbb{P}(T_f(P(r_0'),P(r_{\iota(n)+t}')) \geq x) \\
&\leq e^{-\Cr{c: circuit_appearance_constant} t} + \mathbb{P}\left(T_f(P(r_{i^\ast}'),P(r_{\iota(n)+t}')) \geq \frac{x}{2}\right) + \mathbb{P}\left( \Xi \geq \frac{x}{2}\right) \\
&\leq e^{-\Cr{c: circuit_appearance_constant} t} + \Cr{c: haydns_constant_1} \exp\left( - \Cr{c: haydns_constant_2} \frac{x}{2(\iota(n)-i^\ast +t)}\right) + \frac{\Cr{c: haydns_constant_1}2^{\frac{\eta}{2}}}{x^\frac{\eta}{2}} \\
&+ \mathbb{P}\left( \Xi \geq \frac{x}{2}\right),
\end{align*}
due again to \eqref{eq: 2.28}, Lem.~\ref{lem: plane_to_plane_estimate}, and \eqref{eq: more_bandaids}. Setting $t = \lfloor \sqrt{x} \rfloor \geq \iota(n)-i^\ast+1$ gives the upper bound
\[
e^{-\Cr{c: circuit_appearance_constant} \lfloor \sqrt{x} \rfloor} + \Cr{c: haydns_constant_1} \exp\left( - \Cr{c: haydns_constant_2} \frac{x}{4\lfloor \sqrt{x} \rfloor -2} \right) + \frac{\Cr{c: haydns_constant_1}2^\frac{\eta}{2}}{x^{\frac{\eta}{2}}} + \mathbb{P}\left( \Xi \geq  \frac{x}{2}\right).
\]
Because of \eqref{eq: gap_var_condition_2}, we have $\mathbb{P}(\Xi \geq x/2) \leq \Cl[lgc]{c: bottle_top} / x^\eta$ for some $\Cr{c: bottle_top}$, so if we combine this with \eqref{eq: babys_first_playset}, we conclude \eqref{eq: to_prove_completing_the_proof}.

By integrating \eqref{eq: to_prove_completing_the_proof}, we obtain a constant $\Cl[lgc]{c: second_moment_comparison}>0$ such that for all $n$ with $\iota(n) \geq i^\ast+1$, we have
\[
\mathbb{E}(T_f(0,P(n)) - T_f(0,\Gamma_{\iota(n)}))^2 \leq \Cr{c: second_moment_comparison},
\]
and by monotonicity of norms,
\begin{equation}\label{eq: first_moment_comparison}
\mathbb{E}|T_f(0,P(n)) - T_f(0,\Gamma_{\iota(n)})| \leq \sqrt{\Cr{c: second_moment_comparison}}.
\end{equation}
Furthermore, the triangle inequality implies
\[
\left| \sqrt{\mathrm{Var}~T_f(0,P(n))} - \sqrt{\mathrm{Var}~T_f(0,\Gamma_{\iota(n)})}\right| \leq \sqrt{\mathrm{Var}~(T_f(0,P(n)) - T_f(0,\Gamma_{\iota(n)}))}  \leq \sqrt{\Cr{c: second_moment_comparison}},
\]
and so
\begin{align*}
\left| \frac{\mathrm{Var}~T_f(0,P(n))}{\mathrm{Var}~T_f(0,\Gamma_{\iota(n)})} - 1 \right| &\leq \sqrt{\Cr{c: second_moment_comparison}} \frac{ \sqrt{\mathrm{Var}~T_f(0,P(n))} + \sqrt{\mathrm{Var}~T_f(0,\Gamma_{\iota(n)})}}{\mathrm{Var}~T_f(0,\Gamma_{\iota(n)})} \\
&\leq \sqrt{\Cr{c: second_moment_comparison}} \frac{\sqrt{\Cr{c: second_moment_comparison}} + 2\sqrt{\mathrm{Var}~T_f(0,\Gamma_{\iota(n)})}}{\mathrm{Var}~T_f(0,\Gamma_{\iota(n)})}.
\end{align*}
Because $\mathrm{Var}~T_f(0,\Gamma_{\iota(n)}) = \sum_{i=0}^{i_0} \mathbb{E}\Delta_{i,\iota(n)}^2 \to \infty$ as $n \to \infty$ by item 2 of Prop.~\ref{prop: second_moment_prop}, we conclude that
\begin{equation}\label{eq: variance_comparison}
\frac{\mathrm{Var}~T_f(0,P(n))}{\mathrm{Var}~T_f(0,\Gamma_{\iota(n)})} \to 1 \text{ as } n \to \infty.
\end{equation}

Now we can complete the proofs of Thm.~\ref{thm: updated_KZ_variance} and Thm.~\ref{thm: updated_KZ_CLT}. For the first, we use \eqref{eq: my_favorite_bandaid}, item 1 of Prop.~\ref{prop: second_moment_prop}, and \eqref{eq: variance_comparison} to obtain
\begin{align}
\mathrm{Var}~T_f(0,P(n)) \lesssim \mathrm{Var}~T_f(0,\Gamma_{\iota(n)}) &= \sum_{i=0}^{i^\ast} \mathbb{E}\Delta_{i,\iota(n)}^2 + \sum_{i=i^\ast+1}^{\iota(n)} \mathbb{E}\Delta_{i,\iota(n)}^2 \nonumber \\
&\leq \sum_{i=0}^{i^\ast} \mathbb{E}\Delta_{i,\iota(n)}^2 + \Cr{c: Delta_second_moment_upper}(\iota(n)-i^\ast) \lesssim \Cr{c: Delta_second_moment_upper} \iota(n). \label{eq: variance_upper_bound_near_end}
\end{align}
Similarly, we use item 2 of Prop.~\ref{prop: second_moment_prop} and \eqref{eq: variance_comparison} to obtain
\begin{equation}\label{eq: lower_bound_variance}
\mathrm{Var}~T_f(0,P(n)) \gtrsim \mathrm{Var}~T_f(0,\Gamma_{\iota(n)}) = \sum_{i=0}^{\iota(n)} \mathbb{E}\Delta_{i,\iota(n)}^2 \geq \Cr{c: Delta_second_moment_lower} (\iota(n) - i^\ast) \gtrsim \Cr{c: Delta_second_moment_lower} \iota(n).
\end{equation}
This shows the first part of Thm.~\ref{thm: updated_KZ_variance}. 

For the second part, we observe that because $\iota(n) = \iota(r'_{\iota(n)})$, we have from \eqref{eq: first_moment_comparison} for all $n$ such that $\iota(n) \geq i^\ast+1$ that
\begin{align}
\mathbb{E}|T_f(0,P(r'_{\iota(n)})) - T_f(0,P(n))| &\leq \mathbb{E}|T_f(0,P(r'_{\iota(n)})) - T_f(0,\Gamma_{\iota(n)})| \nonumber \\
&+ \mathbb{E}|T_f(0,P(n)) - T_f(0,\Gamma_{\iota(n)})| \nonumber \\
&\leq 2 \sqrt{\Cr{c: second_moment_comparison}}. \label{eq: new_first_moment_comparison}
\end{align}

By assumption A2 and both \eqref{eq: comparison_claim} and \eqref{eq: lower_bound_variance}, we obtain
\begin{align*}
\mathbb{E}T_f(0,P(r'_{\iota(n)})) &\leq \mathbb{E}[\tau_e \mid \tau_e > 0]~ \mathbb{E}T_f^B(0,P(r'_{\iota(n)})) \\
&= \mathbb{E}[\tau_e \mid \tau_e > 0]\sum_{i=0}^{2\iota(n)-1} \mathbb{E}\left[ T_f^B(0,P(r_{i+1})) - T_f^B(0,P(r_i))\right] \\
&\lesssim 2\mathbb{E}[\tau_e \mid \tau_e > 0]~\Cr{c: r_i_assumption_2} \iota(n) \\
&\lesssim \frac{2\mathbb{E}[\tau_e \mid \tau_e > 0]~\Cr{c: r_i_assumption_2}}{\Cr{c: Delta_second_moment_lower}} \mathrm{Var}~T_f(0,P(n)).
\end{align*}
Similarly, by assumption A3 and both \eqref{eq: comparison_claim} and \eqref{eq: variance_upper_bound_near_end}, we find
\begin{align*}
\mathbb{E}T_f(0,P(r'_{\iota(n)})) &\geq \delta \mathbb{E} T_f^B(0,P(r'_{\iota(n)})) \\
&\geq  \delta \sum_{i=0}^{\iota(n)-1} \mathbb{E}T_f^B(P(r'_i),P(r'_{i+1})) \gtrsim \delta \Cr{c: r_i_assumption_3}(1) \iota(n) \gtrsim \frac{\delta \Cr{c: r_i_assumption_3}(1)}{\Cr{c: Delta_second_moment_upper}} \mathrm{Var}~T_f(0,P(n)).
\end{align*}
The previous three displays imply the second part of Thm.~\ref{thm: updated_KZ_variance}.

For Thm.~\ref{thm: updated_KZ_CLT}, we write
\begin{align}
\frac{T_f(0,P(n)) - \mathbb{E}T_f(0,P(n))}{\sqrt{\mathrm{Var}~T_f(0,P(n))}} &= \frac{T_f(0,P(n)) - T_f(0,\Gamma_{\iota(n)})}{\sqrt{\mathrm{Var}~T_f(0,P(n))}} \label{eq: bach_term_1}\\
&- \frac{\mathbb{E}\left( T_f(0,P(n)) - T_f(0,\Gamma_{\iota(n)})\right)}{\sqrt{\mathrm{Var}~T_f(0,P(n))}} \label{eq: bach_term_2}\\
&+ \sqrt{\frac{\mathrm{Var}~T_f(0,\Gamma_{\iota(n)})}{\mathrm{Var}~T_f(0,P(n))}} \cdot \frac{T_f(0,\Gamma_{\iota(n)}) - \mathbb{E}T_f(0,\Gamma_{\iota(n)})}{\sqrt{\mathrm{Var}~T_f(0,\Gamma_{\iota(n)})}}. \label{eq: bach_term_3}
\end{align}
By \eqref{eq: first_moment_comparison} and \eqref{eq: lower_bound_variance}, the term \eqref{eq: bach_term_1} converges to 0 in probability and \eqref{eq: bach_term_2} converges to 0 as $n \to \infty$. By Prop.~\ref{prop: preliminary_CLT} and \eqref{eq: variance_comparison}, the term \eqref{eq: bach_term_3} converges to a standard normal in distribution. This completes the proof.

\section{Proof of Thm.~\ref{thm: main_variance}}\label{sec: variance_proof}

To prove Thm.~\ref{thm: main_variance}, we will apply Thm.~\ref{thm: updated_KZ_variance} and Thm.~\ref{thm: updated_KZ_CLT}. We start with weights $(\tau_e)_{e \in \mathbb{E}^2}$ satisfying \eqref{eq: gap_var_condition_1} and \eqref{eq: gap_var_condition_2}. It suffices to show that there exist $\Cr{c: r_i_assumption_1}>0, \Cr{c: r_i_assumption_2} \in [1,\infty),$ and $\Cr{c: r_i_assumption_3}: [0,\infty) \to (0,\infty)$ such that the following holds. For any $a> 0$ and $b \geq 0$ such that $\mathbb{E}T_{a,b}(0,P(n)) \to \infty$ as $n \to \infty$, there exists a sequence $(r_i)_{i=0}^\infty = (r_i(a,b))_{i=0}^\infty$ of integers such that
\begin{equation}\label{eq: my_final_playset_to_show}
\text{A1-A3 hold for }(r_i) \text{ and } \Cr{c: r_i_assumption_1},\Cr{c: r_i_assumption_2}, \Cr{c: r_i_assumption_3},
\end{equation}
when we set $f(u) = a \log (1+u) + b \log(1 + \log(1+u))$. The construction depends on the value of $p = F(0)$.

\bigskip
\noindent
{\bf Case 1: $p=1/2$}. Recall the definition of $\ell_j$ from \eqref{eq: ell_j_def}. We define a double sequence of integers $(r_i^{(j)} : j \geq 0, i = 0, \dots, q^{(j)})$ as follows. First, we can choose $j_0 = j_0(a,b)$ such that
\[
\ell_{j+1} - \ell_j \geq j > 0 \text{ for all } j \geq j_0.
\]
For $j < j_0$, we define 
\[
q^{(j)} = 1,~r_0^{(j)} = \ell_j, \text{ and }r_1^{(j)} = \ell_{j+1}.
\]

For $j \geq j_0$, we split the region between $\ell_j$ and $\ell_{j+1}$ into rectangles of bounded aspect ratio. Observe that if $n,d$ are integers with $1 \leq d\leq n$, then there exist $q := \lfloor n/d \rfloor$ integers $d_1 , \dots, d_q$ satisfying
\begin{equation}\label{eq: basic_number_theory}
n = d_1 + \dots + d_q \text{ and } d \leq d_i \leq 2d-1 \text{ for all } i.
\end{equation}
This follows easily from the division algorithm, as we can write $n = dq + r$ for some $r \in \{0, \dots, d-1\}$ and then set $d_1 = d+r$ and $d_2 = \dots = d_q = d$. We apply this fact with $n = \ell_{j+1}-\ell_j$ and $d = j$ to obtain integers $q^{(j)} = \lfloor (\ell_{j+1}-\ell_j)/j\rfloor$ and $d_1^{(j)}, \dots, d_{q^{(j)}}^{(j)}$ with
\begin{equation}\label{eq: basic_number_theory_applied_once}
\ell_{j+1} - \ell_j = d_1^{(j)} + \dots + d_{q^{(j)}}^{(j)} \text{ and } j \leq d_i^{(j)} \leq 2j-1 \text{ for all } i,
\end{equation}
and then define
\begin{equation}\label{eq: basic_number_theory_consequence}
r_0^{(j)} = \ell_j, r_1^{(j)} = \ell_j + d_1^{(j)}, \dots, r_{q^{(j)}}^{(j)} = \ell_j + d_1^{(j)} + \dots + d_{q^{(j)}}^{(j)} = \ell_{j+1}.
\end{equation}

Having defined the double sequence $(r_i^{(j)})$, we now consider the set of values $\{r_i^{(j)} : j \geq 0, i = 0, \dots, q^{(j)}\}$ and enumerate it in order as $0 = r_0 < r_1 < \dots$. This is our final definition of $(r_i) = (r_i(a,b))$ in the case $p=1/2$. We observe that with this definition, we have
\begin{equation}\label{eq: main_inequality_p_1/2}
\lfloor f(r_i) \rfloor \leq r_{i+1} - r_i  \leq 2 \lfloor f(r_i) \rfloor - 1 \text{ for all large }i
\end{equation}
and
\begin{equation}\label{eq: main_inequality_2_p_1/2}
\lfloor f(r_i)\rfloor = \lfloor f(r_{i+1} - 1) \rfloor \text{ for all large }i.
\end{equation}
To see why, choose $k,j$ such that $r_i = r_k^{(j)}$. We may suppose that $i$ is so large that $j \geq j_0$ and also that $k < q^{(j)}$, since the latter being false would give $r_i = r_{q^{(j)}}^{(j)} = r_0^{(j+1)}$, and we could then replace the pair $k,j$ with the pair $0,j+1$. Because $k < q^{(j)}$, we must have $\ell_j \leq r_i < \ell_{j+1}$, which implies $\lfloor f(r_i) \rfloor = j$. Furthermore $r_i \leq r_{i+1} - 1 = r_{k+1}^{(j)} - 1 < \ell_{j+1}$, so we also have $\lfloor f(r_{i+1}-1)\rfloor = j$, implying \eqref{eq: main_inequality_2_p_1/2}. Last, by \eqref{eq: basic_number_theory_applied_once} and \eqref{eq: basic_number_theory_consequence}, we have $r_{i+1}-r_i = d_{k+1}^{(j)}$ and so $\lfloor f(r_i) \rfloor = j \leq d_{k+1}^{(j)} \leq 2j-1 = 2 \lfloor f(r_i) \rfloor -1$.

We now check assumptions A1-A3. By \eqref{eq: main_inequality_2_p_1/2}, the set $R_i$, defined back in \eqref{eq: R_i_def}, is, for all large $i$, either equal to the rectangle $\widetilde{R}_i = ([r_i,r_{i+1}] \times [0,\lfloor f(r_i) \rfloor]) \cap \mathbb{Z}^2$, or equal $\widetilde{R}_i \cup \{(r_{i+1},\lfloor f(r_i)\rfloor+1)\}$. Therefore using \eqref{eq: main_inequality_p_1/2} and \cite[p.~316]{grimmettbook}, we have
\begin{align*}
&\mathbb{P}(\exists\text{ top-down open crossing of }R_i) \\ \geq~& \frac{1}{2}\mathbb{P}(\exists \text{ top-down open crossing of }\widetilde{R}_i) \\
\geq~& \frac{1}{2} \mathbb{P}(\exists \text{ top-down open crossing of } ([r_i,r_i + \lfloor f(r_i)\rfloor] \times [0, \lfloor f(r_i)\rfloor ]) \cap \mathbb{Z}^2) \\
\geq~&\frac{1}{4}.
\end{align*}
The factor $1/2$ appears in the second line by forcing the edge $\{(r_{i+1},\lfloor f(r_i) \rfloor +1)\}$ to be open. This shows A1 with $\Cr{c: r_i_assumption_1} = 1/4$.

Moving to A2, we estimate the difference of passage times using the following lemma, which will also be of use in the case $p>1/2$.
\begin{lem}\label{lem: A2_lemma}
There exist $\Cl[lgc]{c: my_new_lemma_constant}$ and an integer $s_0 = s_0(a,b)$ such that the following holds. If $s,t$ are integers such that $s_0 \leq s < t$, then
\[
\mathbb{E}\left[ T_{a,b}^B(0,P(t)) - T_{a,b}^B (0,P(s)) \right] \leq \begin{cases}
\Cr{c: my_new_lemma_constant} \left( \frac{t-s}{\lfloor f(s) \rfloor} + 1\right) & \quad\text{if } p=\frac{1}{2} \\
\Cr{c: my_new_lemma_constant} (t-s+\lfloor f(s) \rfloor) \exp\left( - \frac{\lfloor f(s) \rfloor}{\xi(1-p)}\right) & \quad\text{if } p > \frac{1}{2}.
\end{cases}
\]
\end{lem}
\begin{proof}
We again use the dual lattice, representing passage times using edge-disjoint closed dual paths. From \eqref{eq: new_Y_n} with $n=t$, $T_{a,b}^B(0,P(t))$ equals the maximal number of edge-disjoint closed dual paths in $\mathbb{G}_{a,b}^\ast(t)$ from $H_{a,b}^\ast(t)$ to $L_{a,b}^\ast(t)$. So let $\{\gamma_1, \dots, \gamma_\ell\}$ be a maximal set of such paths. Applying \eqref{eq: separating_claim} with $n=s$, we obtain
\[
T_{a,b}^B (0,P(s)) \geq \#\{k=1, \dots, \ell : \gamma_k \text{ separates } 0 \text{ from } P(s) \text{ in } \mathbb{G}_{a,b}(s)\},
\]
(here ``$\gamma_k$ separates 0 from $P(s)$'' means that the set $\{e \in \mathbb{E}^2 : e^\ast \text{ is an edge of }\gamma_k\}$ separates $0$ from $P(s)$) and so
\begin{align*}
&\mathbb{E}\left[ T_{a,b}^B(0,P(t)) - T_{a,b}^B (0,P(s)) \right] \\
\leq~& \mathbb{E}\#\{k=1, \dots, \ell : \gamma_k \text{ does not separate } 0 \text{ from } P(s) \text{ in } \mathbb{G}_{a,b}(s)\}.
\end{align*}

We observe that any $k$ such that $\gamma_k$ does not separate 0 from $P(s)$ in $\mathbb{G}_{a,b}(s)$ must fall into at least one of two classes. First, $\gamma_k$ may start (have its topmost vertex) to the right of the vertical line $\{x = s - \lfloor f(s) \rfloor\}$; in this case, it will contain a dual vertex of the set $D_{s,t} = \{s-\lfloor f(s) \rfloor + 1/2, \dots, t - 1/2\} \times \{\lfloor f(s-\lfloor f(s)\rfloor) \rfloor + 1/2\}$. If $\gamma_k$ does not, then it starts to the left of this line and therefore must contain a dual vertex $v$ with $v \cdot \mathbf{e}_1 < s - \lfloor f(s) \rfloor$. However such a $\gamma_k$ must also contain a dual vertex to the right of $\{x = s\}$, since otherwise it would separate $0$ from $P(s)$ in $\mathbb{G}(s)$. Therefore $\gamma_k$ must contain a left-right dual closed crossing of the rectangle $R_s = \{s - \lfloor f(s) \rfloor - 1/2, \dots, s + 1/2\} \times \{1/2, \dots, \lfloor f(s) \rfloor - 1/2\}$. Therefore if we set
\[
X_1 = \text{maximal number of closed dual paths connecting }D_{s,t} \text{ to } L_{a,b}^\ast(t)
\]
and
\[
X_2 = \text{maximal number of closed dual left-right crossings of } R_s,
\]
then we have
\[
\mathbb{E}[T_{a,b}^B(0,P(t)) - T_{a,b}^B(0,P(s))] \leq \mathbb{E}X_1 + \mathbb{E}X_2.
\]

Recall the definition of $U(h(n),n)$ from \eqref{eq: U_h_n_def} and note that
\begin{equation}\label{eq: cleaning_room_x_1}
\mathbb{E}X_1 \leq \mathbb{E}_{1-p} U(t-s+\lfloor f(s) \rfloor -1,\lfloor f(s - \lfloor f(s) \rfloor) \rfloor),
\end{equation}
where $\mathbb{E}_{1-p}$ is expectation with respect to the Bernoulli percolation measure with parameter $1-p$. Similarly, we have
\begin{equation}\label{eq: cleaning_room_x_2}
\mathbb{E}X_2 \leq \mathbb{E}_{1-p} U(\lfloor f(s) \rfloor -1, \lfloor f(s) \rfloor +1) \leq \mathbb{E}_{1-p}U(\lfloor f(s)\rfloor + 1, \lfloor f(s) \rfloor + 1).
\end{equation}
Let $s_0$ be so large that if $s \geq s_0$, then $2 \leq \lfloor f(s)\rfloor \leq \lfloor f(s-\lfloor f(s) \rfloor) \rfloor +1$. 
If $p=1/2$, we can use \eqref{eq: critical_upper} to obtain
\[
\mathbb{E}X_1 + \mathbb{E}X_2 \leq \Cr{c: critical_upper} \left( \frac{t-s+\lfloor f(s)\rfloor  -1}{\lfloor f(s-\lfloor f(s) \rfloor)\rfloor} + 1\right) \leq \Cr{c: critical_upper} \left( \frac{t-s}{\lfloor f(s) \rfloor -1} + 2\right) \leq 2\Cr{c: critical_upper}\left( \frac{t-s}{\lfloor f(s) \rfloor} + 1\right).
\]

If $p > 1/2$, then we use Prop.~\ref{prop: point_to_plane_probability} for any integers $h(n),n \geq 1$  to obtain
\[
\mathbb{E}_{1-p} U(h(n),n) \leq \Cr{c: point_to_plane_upper} (h(n)+1)\exp\left( - \frac{n}{\xi(1-p)}\right)
\]
and so, using \eqref{eq: cleaning_room_x_1} and \eqref{eq: cleaning_room_x_2}, we obtain a constant $\Cl[lgc]{c: cleaning_room_constant_1}$ such that
\begin{align*}
\mathbb{E}X_1+ \mathbb{E}X_2 &\leq \Cr{c: point_to_plane_upper}(t-s + \lfloor f(s) \rfloor ) \exp\left( - \frac{\lfloor f(s-\lfloor f(s) \rfloor) \rfloor}{\xi(1-p)} \right) \\
&+ \Cr{c: point_to_plane_upper}(\lfloor f(s) \rfloor + 2)\exp\left( - \frac{\lfloor f(s) \rfloor + 1}{\xi(1-p)}\right) \\
&\leq \Cr{c: cleaning_room_constant_1} (t-s+\lfloor f(s) \rfloor) \exp\left( - \frac{\lfloor f(s) \rfloor}{\xi(1-p)}\right).
\end{align*}
\end{proof}

To show A2 in the case $p=1/2$, we apply Lem.~\ref{lem: A2_lemma} with $s=r_i$ and $t = r_{i+1}$, along with \eqref{eq: main_inequality_p_1/2}, to find for large $i$
\[
\mathbb{E}[T_{a,b}^B(0,P(r_{i+1})) - T_{a,b}^B(0,P(r_i))] \leq \Cr{c: my_new_lemma_constant} \left( \frac{r_{i+1}-r_i}{\lfloor f(r_i)\rfloor} + 1 \right) \leq \Cr{c: my_new_lemma_constant} \left( \frac{2\lfloor f(r_i) \rfloor - 1}{\lfloor f(r_i) \rfloor} + 1 \right).
\]
This implies that A2 holds with $\Cr{c: r_i_assumption_2} = 3 \Cr{c: my_new_lemma_constant}$. 

For A3, we note, as in \eqref{eq: more_separating_equivalence}, that $T_{a,b}^B(P(r_i),P(r_{i+1}))$ equals the maximal number of disjoint closed sets in $\mathbb{E}_{a,b}$ separating $\{r_i\} \times \{0, \dots, \lfloor f(r_i) \rfloor\}$ from $\{r_{i+1}\} \times \{0, \dots, \lfloor f(r_{i+1}) \rfloor\}$. Any closed dual path with all its vertices in the rectangle $\{r_i + 1/2, \dots, r_{i+1}-1/2\} \times \{-1/2, \dots, \lfloor f(r_{i+1}-1)\rfloor + 1/2\}$ connecting the top to the bottom contains such a separating set, so given an integer $M>0$, we simply split this rectangle into $M$ subrectangles and produce a separating set in each. Similar to \eqref{eq: basic_number_theory}, if $n,M \geq 1$ are integers, there exist $M$ integers $d_1, \dots, d_M$ such that
\begin{equation}\label{eq: basic_number_theory_2}
n = d_1 + \dots + d_M \text{ with } \left\lfloor \frac{n}{M} \right\rfloor \leq d_j \leq \left\lfloor \frac{n}{M} \right\rfloor + 1 \text{ for all } j.
\end{equation}
Again this is obvious and comes from the division algorithm. Write $n = M \lfloor n/M \rfloor + r$, where $r \in \{0, \dots, M-1\}$, and set $d_1 = d_2 = \dots = d_r = \lfloor n/M \rfloor + 1$ and $d_{r+1} = \dots = d_M = \lfloor n/M\rfloor$. 

We apply \eqref{eq: basic_number_theory_2} with $n = r_{i+1}-r_i - 1$ (with $i$ so large that $r_{i+1}-r_i-1 \geq 1$) to write
\[
r_{i+1}-r_i-1 = d_1 + \dots + d_M \text{ with } \left\lfloor \frac{r_{i+1}-r_i-1}{M} \right\rfloor \leq d_j \leq \left\lfloor \frac{r_{i+1}-r_i-1}{M} \right\rfloor + 1 \text{ for all } j.
\]
Then we define rectangles $R_1, \dots, R_M$ by
\begin{align*}
R_j &=  \left\{ r_i + \frac{1}{2} + d_1 + \dots + d_{j-1}+1, \dots, r_i + \frac{1}{2} + d_1 + \dots + d_j\right\} \\
&\times \left\{-\frac{1}{2}, \dots, \lfloor f(r_{i+1}-1) \rfloor + \frac{1}{2} \right\},
\end{align*}
where we use the convention that $d_0 = 0$. Using these definitions and independence, we have
\begin{align}
&\mathbb{P}(T_{a,b}^B(P(r_i),P(r_{i+1})) \geq M) \nonumber \\
\geq~& \mathbb{P}(R_j \text{ has a top-down dual closed crossing for all } j=1, \dots, M) \nonumber \\
\geq~& \mathbb{P}\left( \begin{array}{c}
\left\{-\frac{1}{2}, \dots \left\lfloor \frac{r_{i+1}-r_i-1}{M}\right\rfloor - \frac{3}{2} \right\} \times \left\{ -\frac{1}{2}, \dots, \lfloor f(r_{i+1}-1) \rfloor + \frac{1}{2} \right\} \\
\text{ has a top-down dual closed crossing}
\end{array}
\right)^M. \label{eq: valid_also_for_p_bigger}
\end{align}
(We observe that this sequence of inequalities remains valid even if $p>1/2$.) By \eqref{eq: main_inequality_p_1/2} and \eqref{eq: main_inequality_2_p_1/2}, if $i$ is large, the above is no smaller than 
\[
\mathbb{P}\left( \begin{array}{c}
\left\{-\frac{1}{2}, \dots \left\lfloor \frac{\lfloor f(r_i) \rfloor -1}{M}\right\rfloor - \frac{3}{2} \right\} \times \left\{ -\frac{1}{2}, \dots, \lfloor f(r_i) \rfloor + \frac{1}{2} \right\} \\
\text{ has a top-down dual closed crossing}
\end{array}
\right)^M.
\]
The aspect ratio of this rectangle is $(\lfloor (\lfloor f(r_i) \rfloor - 1)/M \rfloor - 1) / (\lfloor f(r_i) \rfloor + 1)$, which converges to $1 /M > 0$ as $i \to \infty$. By the RSW theorem \cite[Eq.~(11.74),(11.76)]{grimmettbook}, therefore, there exists $\Cl[smc]{c: RSW_part_2} = \Cr{c: RSW_part_2}(M) > 0$ that does not depend on $a$ or $b$ such that
\[
\liminf_{i \to \infty} \mathbb{P}\left( \begin{array}{c}
\left\{-\frac{1}{2}, \dots \left\lfloor \frac{\lfloor f(r_i) \rfloor -1}{M}\right\rfloor - \frac{3}{2} \right\} \times \left\{ -\frac{1}{2}, \dots, \lfloor f(r_i) \rfloor + \frac{1}{2} \right\} \\
\text{ has a top-down dual closed crossing}
\end{array}
\right) \geq \Cr{c: RSW_part_2}.
\]
With the above estimates, this shows A3 in the case $p=1/2$, for $\Cr{c: r_i_assumption_3}(M) = \Cr{c: RSW_part_2}^M$, and completes the proof of Thm.~\ref{thm: main_variance} in the case $p=1/2$.

\bigskip
\noindent
{\bf Case 2a: $p>1/2$ and $a < \xi(1-p)$}. In this case, the definition of $(r_i) = (r_i(a,b))$ is similar to that in the previous case. Instead of splitting the region between $\ell_j$ and $\ell_{j+1}$ into rectangles of bounded aspect ratio, we will split it into rectangles of with width approximately $e^{j/\xi(1-p)}$ and height $j+1$.

We again start with a double sequence $(r_i^{(j)} : j \geq 0, i = 0, \dots, q^{(j)})$ defined as follows. First, for $j \geq 0$, we set 
\[
d^{(j)} = \left\lceil \exp\left( \frac{j+1}{\xi(1-p)}\right) \right\rceil
\]
and choose $j_0 = j_0(a,b)$ such that
\[
\ell_{j+1} - \ell_j \geq d^{(j)} \geq 6j > 0 \text{ for all } j \geq j_0.
\]
This is possible because of \eqref{eq: ell_j_asymptotic} and our assumption that $a < \xi(1-p)$. For $j < j_0$, we define
\[
q^{(j)} = 1, ~r_0^{(j)} = \ell_j, \text{ and } r_1^{(j)} = \ell_{j+1}.
\]
For $j \geq j_0$, we apply \eqref{eq: basic_number_theory} with $n = \ell_{j+1} - \ell_j$ and $d = d^{(j)}$ to obtain integers $q^{(j)} = \lfloor (\ell_{j+1}-\ell_j)/d^{(j)}\rfloor$ and $d_1^{(j)}, \dots, d_{q^{(j)}}^{(j)}$ with
\[
\ell_{j+1} - \ell_j = d_1^{(j)} + \dots + d_{q^{(j)}}^{(j)} \text{ and } d^{(j)} \leq d_i^{(j)} \leq 2d^{(j)}-1 \text{ for all } i,
\]
and then define
\[
r_0^{(j)} = \ell_j, r_1^{(j)} = \ell_j + d_1^{(j)}, \dots, r_{q^{(j)}}^{(j)} = \ell_j + d_1^{(j)} + \dots + d_{q^{(j)}}^{(j)} = \ell_{j+1}.
\]

Having defined the double sequence $(r_i^{(j)})$, we again enumerate the set of values $\{r_i^{(j)} : j \geq 0, i = 0, \dots, q^{(j)}\}$ in order as $0 = r_0 < r_1 < \dots$. Exactly as in \eqref{eq: main_inequality_p_1/2}, we have 
\begin{equation}\label{eq: main_inequality_p_bigger_a}
\left\lceil \exp\left(\frac{\lfloor f(r_i) \rfloor+1}{\xi(1-p)}\right) \right\rceil \leq r_{i+1}-r_i \leq 2\left\lceil \exp\left( \frac{\lfloor f(r_i) \rfloor+1}{\xi(1-p)}\right) \right\rceil-1 \text{ for all large } i
\end{equation}
and \eqref{eq: main_inequality_2_p_1/2} also holds. 

The argument for A1 (with $\Cr{c: r_i_assumption_1} = 1/4$) is the same as it was in the last case. The only requirements were that $p\geq 1/2$, and that both \eqref{eq: main_inequality_2_p_1/2} and $\lfloor f(r_i) \rfloor \leq r_{i+1}-r_i$ hold; the last of which follows for large $i$ from \eqref{eq: main_inequality_p_bigger_a}.

For A2, we apply Lem.~\ref{lem: A2_lemma} with $s = r_i$ and $t=r_{i+1}$, along with \eqref{eq: main_inequality_p_bigger_a}, for large $i$ to obtain
\begin{align*}
&\mathbb{E}\left[ T_{a,b}^B(0,P(t)) - T_{a,b}^B (0,P(s)) \right]  \\
\leq~& \Cr{c: my_new_lemma_constant} (r_{i+1}-r_i+\lfloor f(r_i) \rfloor) \exp\left( - \frac{\lfloor f(r_i) \rfloor}{\xi(1-p)}\right)  \\
\leq~&\Cr{c: my_new_lemma_constant} \left(2\left\lceil \exp\left( \frac{\lfloor f(r_i) \rfloor+1}{\xi(1-p)}\right) \right\rceil-1+\lfloor f(r_i) \rfloor \right) \exp\left( - \frac{\lfloor f(r_i) \rfloor}{\xi(1-p)}\right). 
\end{align*}
Because $f(r_i) \to \infty$ as $i \to \infty$, this implies A2 with $\Cr{c: r_i_assumption_2} = 2\Cr{c: my_new_lemma_constant} e^{1/\xi(1-p)}$.

To show A3, we again split our region into subrectangles. As remarked below \eqref{eq: valid_also_for_p_bigger}, the arguments leading to \eqref{eq: valid_also_for_p_bigger} remain valid in the current case. By applying \eqref{eq: main_inequality_p_bigger_a} and \eqref{eq: main_inequality_2_p_1/2} in \eqref{eq: valid_also_for_p_bigger}, we find for integer $M \geq 1$
\begin{align*}
&\mathbb{P}(T_{a,b}^B(P(r_i),P(r_{i+1})) \geq M) \\
\geq~& \mathbb{P}\left( \begin{array}{c}
\left\{-\frac{1}{2}, \dots \left\lceil \frac{\exp\left( \frac{\lfloor f(r_i)\rfloor +1}{\xi(1-p)} \right) -1}{M}\right\rceil - \frac{3}{2} \right\} \times \left\{ -\frac{1}{2}, \dots, \lfloor f(r_i) \rfloor + \frac{1}{2} \right\} \\
\text{ has a top-down dual closed crossing}
\end{array}
\right)^M.
\end{align*}
To estimate this, we return to Prop.~\ref{prop: rectangle_crossing_probability} and apply it with $n = \lfloor f(r_i)\rfloor + 1$ and $h(n) = \lceil (\exp((\lfloor f(r_i)\rfloor + 1)/\xi(1-p)) - 1)/M\rceil - 1$. If $i$ is large, we have $h(n) \geq 6n$ and so if we set $\Phi(s) = s/(s+1)$, we obtain the lower bound for the right side
\[
 \Phi\left( \Cr{c: rectangle_crossing_lower} \left( \left\lceil \frac{\exp\left( \frac{\lfloor f(r_i)\rfloor +1}{\xi(1-p)} \right) -1}{M}\right\rceil -1\right)\exp\left( - \frac{\lfloor f(r_i)\rfloor + 1}{\xi(1-p)}\right) \right)^M.
\]
Letting $i \to \infty$ produces
\[
\liminf_{i \to \infty} \mathbb{P}(T_{a,b}^B(P(r_i),P(r_{i+1})) \geq M) \geq \Phi\left( \frac{\Cr{c: rectangle_crossing_lower}}{M}\right)^M.
\]
In other words, A3 holds with $\Cr{c: r_i_assumption_3}(M) = \Phi(\Cr{c: rectangle_crossing_lower}/M)^M$.

\bigskip
\noindent
{\bf Case 2b: $p > 1/2$ and $b \leq a = \xi(1-p)$.} In this case, we must define the regions $R_i$ so that they are increasing unions of rectangles. This is because when $b>0$, the expected passage time from $P(\ell_j)$ to $P(\ell_{j+1})$ goes to zero as $j \to \infty$ and thus the region between $\ell_j$ and $\ell_{j+1}$ itself does not satisfy A3. When $b=0$, however, we can simply set $r_i = \ell_i$.

To this end, we define $r_i = \ell_{j(i)}$, where $j(i)$ is chosen as
\[
j(0) = 0, j(1) = 1, \text{ and } j(i+1) = j(i) + \left\lceil j(i)^{\frac{b}{\xi(1-p)}} \right\rceil \text{ for } i \geq 1.
\]
To check A1-A3, we again start with A1. Because $\lfloor f(\ell_j)\rfloor = j$ and $\ell_{j+1}-\ell_j > j$ for large $j$, we have for large $i$
\begin{align*}
&\mathbb{P}( \exists \text{ top-down open crossing of }R_i) \\
\geq~& \mathbb{P}( \exists \text{ top-down open crossing of } ([\ell_{j(i)}, \ell_{j(i)} + j(i)] \times [0,j(i)]) \cap \mathbb{Z}^2) \\
\geq~& \frac{1}{2},
\end{align*}
by \cite[p.~316]{grimmettbook}. This shows A1 with $\Cr{c: r_i_assumption_1} = 1/2$. 

For A2, we apply Lem.~\ref{lem: A2_lemma} for large $i$ to obtain
\begin{align}
\mathbb{E}[T_{a,b}^B(0,P(r_{i+1})) - T_{a,b}^B(0,P(r_i))] &= \sum_{k=j(i)}^{j(i+1)-1} \mathbb{E}[T_{a,b}^B(0,P(\ell_{k+1})) - T_{a,b}^B(0,P(\ell_k))]  \nonumber \\
&\leq \sum_{k=j(i)}^{j(i+1)-1} \Cr{c: my_new_lemma_constant} (\ell_{k+1}-\ell_k +\lfloor f(\ell_k) \rfloor) \exp\left( - \frac{\lfloor f(\ell_k) \rfloor}{\xi(1-p)}\right) \nonumber \\
&= \Cr{c: my_new_lemma_constant} \sum_{k=j(i)}^{j(i+1)-1}  (\ell_{k+1}-\ell_k + k) \exp\left( - \frac{k}{\xi(1-p)}\right). \label{eq: parking_lot_doorway}
\end{align}
Now we recall the asymptotic \eqref{eq: ell_j_asymptotic} for $\ell_{k+1}-\ell_k$, which implies for some constant $\Cr{c: my_new_lemma_constant_deux}$ that does not depend on $b$
\begin{align*}
\eqref{eq: parking_lot_doorway} \lesssim \Cr{c: my_new_lemma_constant} \left( e^{\frac{1}{\xi(1-p)}}-1\right) \sum_{k=j(i)}^{j(i+1)-1} \left( \frac{\xi(1-p)}{k} \right)^{\frac{b}{\xi(1-p)}} &\leq \Cl[lgc]{c: my_new_lemma_constant_deux} \sum_{k=j(i)}^{j(i+1)-1} k^{-\frac{b}{\xi(1-p)}} \\
&\leq \Cr{c: my_new_lemma_constant_deux} \left\lceil j(i)^{\frac{b}{\xi(1-p)}} \right\rceil j(i)^{-\frac{b}{\xi(1-p)}}.
\end{align*}
Letting $i \to \infty$ proves A2 with $\Cr{c: r_i_assumption_2} = \Cr{c: my_new_lemma_constant_deux}$.

Last, we show A3. If $b=0$ and $a = \xi(1-p)$, we do the same argument as in the last case. Namely, because $r_i = \ell_i$, we can again use Prop.~\ref{prop: rectangle_crossing_probability} and \eqref{eq: ell_j_asymptotic} to write for large $i$
\begin{align*}
\mathbb{P}(T_{a,b}^B(P(r_i),P(r_{i+1})) \geq M) 
&\geq \mathbb{P}\left( \begin{array}{c}
\left\{-\frac{1}{2}, \dots \left\lceil \frac{\ell_{i+1}-\ell_i -1}{M}\right\rceil - \frac{3}{2} \right\} \times \left\{ -\frac{1}{2}, \dots, i + \frac{1}{2} \right\} \\
\text{ has a top-down dual closed crossing}
\end{array}
\right)^M \\
&\geq \Phi\left( \Cr{c: rectangle_crossing_lower} \left( \left\lceil \frac{\ell_{i+1}-\ell_i -1}{M}\right\rceil - 1 \right) \exp\left( - \frac{i + 1}{\xi(1-p)}\right) \right)^M \\
&\gtrsim \Phi\left( \frac{\Cr{c: rectangle_crossing_lower}}{M} \left(1- e^{-\frac{1}{\xi(1-p)}}\right)\right)^M.
\end{align*}
This shows A3 with $\Cr{c: r_i_assumption_3}(M) = \Phi(\Cr{c: rectangle_crossing_lower} (1-e^{-1/\xi(1-p)})/M)^M$.

In the case $0 < b \leq \xi(1-p)$, we define a sequence of events $A_0, A_1, A_2, \dots$ by
\[
A_j = \left\{\exists \text{ top-down dual closed crossing of } \left[\ell_j + \frac{1}{2}, \ell_{j+1}-\frac{1}{2}\right] \times \left[ -\frac{1}{2}, j + \frac{1}{2}\right]\right\}.
\]
Because $\ell_{j+1}-\ell_j - 1 \geq 6(j+1)$ for large $j$, we can apply Prop.~\ref{prop: rectangle_crossing_probability} and \eqref{eq: ell_j_asymptotic} to obtain for large $j$ and a constant $\Cr{c: bedpan_deluxe}$ that does not depend on $b$
\begin{align*}
\mathbb{P}(A_j) \geq \Phi\left( \Cr{c: rectangle_crossing_lower} (\ell_{j+1}-\ell_j-1) \exp\left( - \frac{j+1}{\xi(p)}\right)\right) &\sim \Phi\left( \Cr{c: rectangle_crossing_lower} (1-e^{-\frac{1}{\xi(1-p)}})\left( \frac{j}{\xi(1-p)}\right)^{-\frac{b}{\xi(1-p)}}\right) \\ 
&\sim \Cr{c: rectangle_crossing_lower} (1-e^{-\frac{1}{\xi(1-p)}})\left( \frac{j}{\xi(1-p)}\right)^{-\frac{b}{\xi(1-p)}} \\
&\geq \Cl[smc]{c: bedpan_deluxe} j^{-\frac{b}{\xi(1-p)}}.
\end{align*}
Because of this, we can choose $j_0$ such that if $j \geq j_0$, then $\mathbb{P}(A_j) \geq p_j := (\Cr{c: bedpan_deluxe}/2)j^{-b/\xi(1-p)}$. Because the $A_j$ are independent, the family of indicator variables $(\mathbf{1}_{A_j})_{j \geq j_0}$ stochastically dominates a family $(\Upsilon_j)_{j \geq j_0}$ of independent Bernoulli random variables with parameters $(p_j)_{j \geq j_0}$. Therefore for integer $M \geq 1$, we have
\begin{equation}\label{eq: domination}
\mathbb{P}(T_{a,b}^B(P(r_i),P(r_{i+1})) \geq M) \geq \mathbb{P}\left( \sum_{j=j(i)}^{j(i+1)-1} \Upsilon_j \geq M\right).
\end{equation}
We are assuming that $b>0$, so $p_j \to 0$ as $j \to \infty$. Furthermore if $b < \xi(1-p)$, we have $j(i+1)/j(i) \to 1$ as $i \to \infty$ and so
\[
\sum_{j=j(i)}^{j(i+1)-1} p_j = \frac{\Cr{c: bedpan_deluxe}}{2} \sum_{j=j(i)}^{j(i+1)-1} \frac{1}{j^{\frac{b}{\xi(1-p)}}} \sim \frac{\Cr{c: bedpan_deluxe}}{2} \cdot \frac{j(i+1)-j(i)}{j(i)^\frac{b}{\xi(1-p)}} \sim \frac{\Cr{c: bedpan_deluxe}}{2}.
\]
If $b = \xi(1-p)$ instead, then we have $j(i+1)/j(i) \to 2$ as $i \to \infty$ and so
\[
\sum_{j=j(i)}^{j(i+1)-1} p_j = \frac{\Cr{c: bedpan_deluxe}}{2} \sum_{j=j(i)}^{j(i+1)-1} \frac{1}{j} \sim \frac{\Cr{c: bedpan_deluxe}}{2}  \cdot \log 2.
\]
In either case, we can use the Poisson limit theorem \cite[Thm.~3.6.1]{durrettbook} to obtain that the sum $\sum_{j=j(i)}^{j(i+1)-1} \Upsilon_j$ converges in distribution as $i \to \infty$ to a Poisson random variable with parameter equal to $\Cr{c: bedpan_deluxe}/2$ if $b \in (0,\xi(1-p))$ and equal to $\Cr{c: bedpan_deluxe}(\log 2)/2$ if $b = \xi(1-p)$. Therefore from \eqref{eq: domination}, we have
\begin{equation}\label{eq: rho_function}
\liminf_{i \to \infty} \mathbb{P}(T_{a,b}^B(P(r_i),P(r_{i+1})) \geq M) \geq \begin{cases}
e^{-\frac{\Cr{c: bedpan_deluxe}}{2}} \frac{\left( \frac{\Cr{c: bedpan_deluxe}}{2}\right)^M}{M!} &\quad\text{if } b \in (0,\xi(1-p)) \\
e^{-\Cr{c: bedpan_deluxe}\frac{\log 2}{2}} \frac{\left( \Cr{c: bedpan_deluxe}\frac{\log 2}{2}\right)^M}{M!} &\quad\text{if } b = \xi(1-p).
\end{cases}
\end{equation}
In either case, this shows A3 with $\Cr{c: r_i_assumption_3}(M)$ equal to the function on the right side of \eqref{eq: rho_function}.

Combining cases 2a and 2b, and setting $\Cr{c: r_i_assumption_1} = 1/4$, $\Cr{c: r_i_assumption_2} = \max\{2\Cr{c: my_new_lemma_constant} e^{1/\xi(1-p)},  \Cr{c: my_new_lemma_constant_deux}\}$, and $\Cr{c: r_i_assumption_3}(M) = \min\{\Phi(\Cr{c: rectangle_crossing_lower}/M)^M, \Phi(\Cr{c: rectangle_crossing_lower} (1-e^{-1/\xi(1-p)})/M)^M, \mathsf{V}(M)\}$, where $\mathsf{V}(M)$ is the function in \eqref{eq: rho_function}, we complete the proof of Thm.~\ref{thm: main_variance} in the case $p>1/2$.

\medskip
\noindent
{\bf Acknowledgements.} The authors thank Pengfei Tang for work on a preliminary version of this project, and Jack Hanson for helpful comments on a previous draft.

\end{document}